\documentclass[12pt]{article}
\usepackage{amssymb}
\usepackage{mathrsfs}
\usepackage{bm}
\usepackage{amsfonts}
\usepackage{graphicx}
\usepackage{setspace}
\usepackage[labelformat=parens]{subfig}

\usepackage{indentfirst}
\usepackage{colortbl,dcolumn}
\usepackage{color}
\usepackage{comment}
\usepackage{xcolor}
\usepackage{amsmath}
\usepackage{bbm} 
\usepackage{dsfont} 
\usepackage{psfrag}
\usepackage{booktabs}
\usepackage{array}
\usepackage{cite}
\usepackage{amsthm}
\numberwithin{equation}{section}
\usepackage[top=1in, bottom=1in, left=0.8in, right=0.8in]{geometry}
\newcommand{\R}{\mathbb{R}}
\newcommand{\N}{\mathbb{N}}
\newcommand{\E}{\mathbb{E}}
\renewcommand{\P}{\mathbb{P}}

\newcommand{\dd}{\text{d}}
\newtheorem{thm}{Theorem}[section]

\newtheorem{lem}[thm]{Lemma}
\newtheorem{prop}[thm]{Proposition}
\newtheorem{cor}[thm]{Corollary}
\newtheorem{rem}[thm]{Remark}
\newtheorem{example}[thm]{Example}
\newtheorem{assumption}[thm]{Assumption}
\allowdisplaybreaks[4] %公式跨页

\begin{document}
\title{
%Total variation error of an explicit scheme for Allen-Cahn type SPDEs with space-time white noise near the sharp interface limit
An explicit scheme for stochastic Allen-Cahn equations with space-time white noise near the sharp interface limit
%Long-time total variation error analysis of an explicit time-stepping scheme for stochastic Allen-Cahn equations with space-time white noise
%near sharp-interface limit
%contractive and non-contractive coefficients
}
%\footnotemark[2]
%\footnotetext[2]{This work was supported by National Natural Science Foundation of China (12471394, 12071488, 12371417) and the innovative project of graduate students of  Central South University  (2021zzts0040)}. 
                %for Distinguished Young Scholars}
                
\author{Yingsong Jiang, Chenxu Pang, Xiaojie Wang 
\thanks{All authors contributed equally. This work was supported by Natural Science Foundation of China (12471394, 12071488, 12371417) and the Postdoctoral Fellowship Program of CPSF (GZB2025715). \\
E-mail addresses: x.j.wang7@csu.edu.cn (Corresponding author), yingsong@csu.edu.cn, c.x.pang@csu.edu.cn.
}
\\
\footnotesize School of Mathematics and Statistics, Hunan Research Center of the Basic Discipline for Analytical  
\\
\footnotesize Mathematics, HNP-LAMA, Central South University, Changsha 410083, Hunan, China
}
\date{ }
\maketitle

\begin{abstract}
This article investigates time-discrete approximations of Allen–Cahn type stochastic partial differential equations (SPDEs) driven by space–time white noise  near the sharp interface limit $\epsilon \to 0$, where the small parameter $\epsilon$ is the diffuse interface thickness.
%where the interface parameter $\epsilon>0$ causes essential difficulties for stability and error control. 
We propose an explicit and easily implementable exponential integrator with a modified nonlinearity  for the considered problem. Uniform-in-time and uniform-in-$\epsilon$ moment bounds of the scheme are established and the convergence in total variation distance of order %$\mathcal{O}\big( \epsilon^{-\frac{q+1+2\gamma}{2}-\kappa} \tau^{\gamma \wedge \theta} \big)$ 
$O\big( T \cdot \text{Poly}(\epsilon^{-1}) \tau^{\gamma} \big), \gamma < \tfrac12$,
is established, between the law of the numerical scheme and that of the SPDE over $[0, T]$.
In contrast to the exponential dependence due to standard arguments, the obtained error bound depends on $\epsilon^{-1}$ and $T$ polynomially.
By incorporating carefully chosen method parameters, we only require a mild and $\epsilon$-independent restriction on the time step-size $\tau$, getting rid of the severe restriction $\tau = O (\epsilon^\sigma), \sigma \geq 1$ in the literature.
Also, a uniform-in-time error bound of order $O(\tau^{\gamma}), \gamma < \tfrac12$, is obtained for a fixed $\epsilon$ (e.g., $\epsilon = 1$), which improves the existing ones in the literature and matches the classical weak convergence rate in the globally Lipschitz setting. 
The error analysis is highly nontrivial due to the low regularity of the considered problem, the super-linear growth of the non-globally Lipschitz drift, the non-smooth observables inherent in the total variation metric and the presence of the small interface parameter $\epsilon \to 0$. 
These difficulties are addressed by introducing a new strategy of nonlinearity modification and establishing refined regularity estimates for the associated Kolmogorov equation to an auxiliary process in the context of non-smooth test functions. 
%which enables a rigorous weak error analysis in this setting.
Numerical experiments are included to demonstrate the theoretical convergence and the ability of interface-capturing for the proposed scheme.

\end{abstract}

\section{Introduction}
%infinite dimensional semilinear stochastic evolution equations (SEEs)
Throughout this paper, we are interested in the following
parabolic stochastic partial differential equations (SPDEs) 
in the Hilbert space $H := L^2(\mathcal{D}; \R)$ driven by space-time white noise:
\begin{equation}
\label{eq:SEE(in_Introduction)[TV4AC-25]}
\left\{\begin{array}{l}
\mathrm{d} X(t) 
= - A X(t) \, \mathrm{d} t 
+
\epsilon^{-1} 
F(X(t)) \, \mathrm{d} t 
+ 
\mathrm{d} W(t), 
\quad t > 0, \\
X(0) = X_0.
\end{array}\right.
\end{equation}
Here, $\mathcal{D} := (0, 1)$, the small parameter $\epsilon >0$ is the diffuse interface thickness,
%$\mathcal{D} \subset \R$ is a bounded spatial domain with smooth boundary, 
$-A$ is the Laplacian operator with homogeneous Dirichlet boundary conditions, $F$ is a nonlinear Nemytskii operator associated with a real-valued function $f \colon \R \rightarrow \R$ such that
$
    F(u)(x) := f( u(x) ), 
    x \in \mathcal{D}
$,
and $\{W(t)\}_{t \geq 0}$ is a cylindrical $I$-Wiener process (see Assumptions \ref{assump:A(linear_operator)[TV4AC-25]}-\ref{assump:F(Nonlinearity)[TV4AC-25]} below for details).
The considered equation includes a stochastic Allen–Cahn equation as a special case, which describes random phase    separation or interface motion under thermal fluctuations.
As $\epsilon \to 0$, the solution becomes nearly piecewise constant,  taking values close to the stable equilibria $\pm 1$ in two bulk regions separated by a diffusive interfacial layer of thickness $O(\epsilon)$  \cite{zhang2009numerical}.
%(Zhang et al., SIAM J Sci Conput, 2009).
Such a limiting behavior is commonly referred to as the 
sharp-interface limit, and
the stochastic Allen--Cahn dynamics formally
converge to a 
stochastic mean curvature flow \cite{kawasaki1982kinetic}.
However, it is a challenge to numerically resolve the 
$O(\epsilon)$-thick transition layer and 
to accurately capture such thin interfaces. In particular, extremely fine resolutions and high computational cost are required to effectively capture thin interfacial layers 
%has thus been an active research topic 
(see, e.g., \cite{funaki1999singular,weber2010short,hairer2012triviality,hairer2015large,Cui_Sun2024weaksharp,feng2017finite,cheng2020new} and references therein).

%In this settings, Kawasaki and Ohta [KO82] obd that in the low temperature limit, which corresponds to the sharp interface limit, the interface moves according to stochastic mean curvature flow. This observation raised a lot of interest and was later revisited by the math community.
%
%Over the past decades, numerical analysis of stochastic partial differential equations (SPDEs), which have been widely used in various scientific areas, has attracted increasing attention. Many monographs were proposed and made great progress on numerical approximations for models with nonlinearities being globally Lipschitz (see e.g.\cite{brehier2024SiamTV,Arnulf_weak4_global_lips,Chenchuchu_Maxwell,COHEN2020109382,Kamrani,WangANDQi201531}, to just mention a few). Nevertheless, the globally Lipschitz condition imposed on the nonlinear term is sometimes strict for practical use. For this reason, some scholars consider the case of nonlinearities possessing super-linear growth (see \cite{CUI2021weak,wang2020efficient,Cai2021weak4ACE,JENTZEN2019Burger,Arnulf_strong4non_global_lips,Brehier_Allen_C,feng2017finite,MajeeProhl_AC} for instance).
%
%To obtain some further development, we propose this paper concerning a class of parabolic SPDEs with possibly non-globally Lipschitz nonlinearity, and giving the error analysis based on the total variation distance.
% the stochastic Allen-Cahn equations, \
%
%More specifically,

For a fixed $\epsilon>0$ (e.g. $\epsilon=1$), 
the numerical analysis of the underlying model has been extensively examined in the literature (see e.g.\cite{brehier2025preconditioning,brehier2024SiamTV,WangANDQi201531,CUI2021weak,wang2020efficient,Cai2021weak4ACE,Arnulf_strong4non_global_lips,Brehier_Allen_C,feng2017finite,MajeeProhl_AC,liu-shen_2025geometric,wangyibo2024ACE,QiWang2019optimal}, to just mention a few).
On the contrary, the numerical analysis in the sharp-interface limit
%numerical analysis concerning the parameter 
 $\epsilon \rightarrow 0$
%$\epsilon$ that possibly goes to $0$ 
is much less studied.
%
%
%{\color{red} A main problem arises in the removal of the item $\exp(\epsilon^{-1})$
%exponential dependence of $\epsilon^{-1}$ in the final convergence result. Such item could naturally appear by standard arguments of convergence analysis, which, however, could introduce a great influence on the error estimate while $\epsilon \rightarrow 0$. }
In this regime, standard convergence arguments yield error bounds with exponential dependence on $\epsilon^{-1}$, which in turn would impose a severe restriction on the time-stepping step-size as $\epsilon$ tends to zero.
%and it influences the error estimate a lot while $\epsilon \rightarrow 0$.
%
%
%
%{\color{red} Recently,  \cite{Cui_Sun2024weaksharp} examines the weak convergence of a splitting scheme to the stochastic Allen-Cahn equation, with a polynomial dependence of $\epsilon^{-1}$ uncovered. More recently, concerning stochastic reaction-diffusion equations driven by trace-class noise, \cite{cui2024improving} applies the spectral Galerkin method and a fully discrete tamed Euler scheme, giving the error bound under the truncated Wasserstein distance that depends on $\epsilon^{-1}$ polynomially.}
%Recent progress has been made to overcome this issue. 
% In \cite{Cui_Sun2024weaksharp},  weak convergence of a splitting scheme for \eqref{eq:spde[TV4AC-25]} was established,
% with dependence on $\epsilon^{-1}$ polynomially. 
%
A natural and interesting question thus emerges: 
\vspace{0.3cm}

{\bf (Q).} {\it
Can one provide an error bound polynomially depending on $\epsilon^{-1}$ for a numerical approximation of SPDEs \eqref{eq:SEE(in_Introduction)[TV4AC-25]} near the sharp interface limit $\epsilon \to 0$?
}

\vspace{0.3cm}

Recently, the authors of \cite{Cui_Sun2024weaksharp} and \cite{cui2024improving} gave a positive answer to this question. More accurately, the authors of \cite{Cui_Sun2024weaksharp}  proposed a splitting method for \eqref{eq:SEE(in_Introduction)[TV4AC-25]}, where the phase flow of a parameterized nonlinear ODE with the nonlinearity $F$ should be exactly calculated. As the main focus of \cite{Cui_Sun2024weaksharp}, the authors established weak convergence rates of the splitting scheme in the context of smooth test functions $\varphi \in \mathcal{C}^{2}_{b}(H)$, with error bounds polynomially depending on $\epsilon^{-1}$ obtained under a severe restriction $\tau = O (\epsilon)$ on the step-size.
For SPDEs with trace-class noise, which is smoother than the space-time white one, the authors of \cite{cui2024improving} proposed a fully discrete tamed Euler scheme and derived error bounds polynomially depending on $\epsilon^{-1}$, in a truncated $L^1$-Wasserstein distance involving Lipschitz continuous test functions, under a even more severe restriction $\tau = O (\epsilon^{\sigma}), \sigma >1$. 

In the present work we restrict ourselves to SPDEs \eqref{eq:SEE(in_Introduction)[TV4AC-25]} driven by space-time white noise and aim to obtain error bounds of new explicit schemes under total variation distance, which only polynomially depend on $\epsilon^{-1}$ and do not impose the  severe restriction $\tau = O (\epsilon^\sigma), \sigma \geq 1$ on the step-size. 
More precisely, we introduce the following time-stepping scheme:
\begin{equation}
\label{eq:(intro)time_discretization[TV4AC-25]}
    X^{\tau}_{t_{m+1}}
    =
    E(\tau) X^{\tau}_{t_m} 
    + 
    \tau E(\tau)\epsilon^{-1}
    F_{\tau}\big( X^{\tau}_{t_m} \big)
    +
    \int_{t_m}^{t_{m+1}}
    E(t_{m+1}-s)\,\mathrm{d} W(s),
    \quad
    X_0^{\tau} = X_0,
\end{equation}
for $m \in \N_0$, where 
$F_{\tau}\colon L^{4q-2}(\mathcal{D}) \to H$, $q > 1$, is defined by
\begin{equation}
\label{eq:(intro)example_f_tau[TV4AC-25]}
    F_{\tau}(u)(x) := f_{\tau}(u(x)), \
    x \in \mathcal{D},
    \quad
    \text{ with }
    \quad
    f_{\tau}(v)
    :=
    \tfrac{ f(v) }
    {\big(1 + \beta \tau^{\theta} |v|^{\frac{2q-2}{\alpha}} \big)^{\alpha}},
    \quad v \in \R,
\end{equation}
where $\beta, \theta > 0$, $\alpha \in (0, \tfrac{1}{\theta})$ are method parameters and $2q-1$ is the degree of polynomial growth of $f$.
%for $q>1$, while $\alpha = 0$ for $q=1$.
%
%
To the best of our knowledge, the time-stepping scheme is new, even for deterministic Allen-Cahn equations.
Such a modification \eqref{eq:(intro)example_f_tau[TV4AC-25]} results in a fully explicit and easily implementable scheme, which simultaneously preserves the dissipativity of
$f$ (see Proposition \ref{prop:F_tau[TV4AC-25]}) and thus helps us establish the uniform-in-time and uniform-in-$\epsilon$ moment bounds of the numerical solution, under the restriction 
$
    \tau
    =
    O
    \big(
     \beta^{\frac{\alpha}{1-\theta\alpha}}
     \epsilon^{ 
    \frac{1}{1-\theta\alpha}  
    }
    \big)
$
on the time step-size (Theorem \ref{thm:X^tau-bound[TV4AC-25]}). 
The convergence in total variation (TV) distance of order %$\mathcal{O}\big( \epsilon^{-\frac{q+1+2\gamma}{2}-\kappa} \tau^{\gamma \wedge \theta} \big)$ 
$O\big( t_m \cdot \text{Poly}(\beta) \cdot \text{Poly}(\epsilon^{-1}) \tau^{\gamma} \big), \gamma < \tfrac12$,
is established, between the law of $X^{\tau}_{t_{m}}$ and $X(t_m)$, which depends on $\epsilon^{-1}$ and $T$ polynomially (Theorem \ref{Thm_main_sharpinterface[TV4AC-25]}). Also, a uniform-in-time error bound of order $O(\tau^{\gamma}), \gamma < \tfrac12$, is obtained for a fixed $\epsilon = 1$ (Corollary \ref{cor_main_uniform[TV4AC-25]}). These findings improve the existing TV convergence rate of order $\tfrac{\gamma}{2}$ in  \cite{Cui_Sun2024weaksharp} and 
match the classical weak convergence rate in the globally Lipschitz setting \cite{brehier2024FCM,Brehier2025PA}.
Another interesting finding in numerical experiments is that, decreasing the degree $\alpha$ in the scheme seemingly improves the computational accuracy. In addition, numerical results indicate a good  performance of the proposed scheme in interface-capturing.

Different from the splitting scheme proposed by \cite{Cui_Sun2024weaksharp}, the new scheme \eqref{eq:(intro)time_discretization[TV4AC-25]}-\eqref{eq:(intro)example_f_tau[TV4AC-25]} is easy to implement and more direct, as one does not need to exactly solve the phase flow of a nonlinear ODE with the nonlinearity $F$.
Also, we highlight that, the tamed scheme here is computational cheaper than tamed schemes introduced in \cite{cui2024improving}, where the Sobolev norm $\| \cdot \|_{\vartheta}$ needs to be computed per every time-step.
Indeed, by incorporating a flexible degree $\alpha \in (0, \tfrac{1}{\theta})$ (instead of a fixed degree $\tfrac12$ or $1$ in the literature) and a method parameter $\beta$ in the taming factor, we introduce a different taming strategy, which, through carefully choosing method parameters, can significantly relieve the required restriction 
$
    \tau
    =
    O
    \big(
     \beta^{\frac{\alpha}{1-\theta\alpha}}
     \epsilon^{ 
    \frac{1}{1-\theta\alpha}  
    }
    \big).
$
%
%
%{\color{red}
More precisely,
%under a particular scaling of the parameters that $\beta^\alpha = O(\epsilon^{-1})$,
by taking $\beta^\alpha = \epsilon^{-1}$, we get an $\epsilon$-independent restriction $\tau = O (1)$, while retaining an error bound with polynomial dependence on $\epsilon^{-1}$
%while remaining a polynomial dependence on $\epsilon^{-1}$ is still achieved in the $\beta$-dependent convergence estimate, thereby allowing greater flexibility in the choice of $\tau$ for simulations near the sharp-interface limit 
(see Theorem \ref{Thm_main_sharpinterface[TV4AC-25]} for details). This essentially overcomes the severe restriction $\tau = O (\epsilon^\sigma), \sigma \geq 1$ in the literature \cite{Cui_Sun2024weaksharp,cui2024improving}.

Distinct from weak convergence analysis with smooth test functions and other metrics like $L^1$-Wasserstein distance, the TV distance involves non-smoothness in nature, which makes the corresponding analysis more challenging.
%
%
% We mention that the seminal work \cite{brehier2024SiamTV} obtained 
% error estimates under the total variation distance of an accelerated exponential Euler scheme applied to SPDEs 
% with globally Lipschitz continuous nonlinearities
% driven by additive space-time white noise.
% The analysis was conducted via the Kolmogorov equations approach.
In the globally Lipschitz setting, the author of \cite{brehier2024SiamTV} derived the TV convergence of an accelerated exponential Euler scheme for SPDEs driven by space–time white noise via the Kolmogorov equation.
A crucial step for the analysis is to treat the TV error between
$\mathrm{law}( X^{\tau}_{t_m} )$ and 
$\mathrm{law}(X(t_m))$ 
as the weak error analysis 
%\eqref{eq:(intro)time_discretization[TV4AC-25]}
in the setting of non-smooth test functions.
However,
the presence of 
$\epsilon^{-1}$
%$\tfrac{1}{\epsilon}$ 
and the non-globally Lipschitz nonlinearities pose substantial difficulties for the construction of the scheme and its error analysis.
To overcome them,
we introduce the auxiliary process $\mathbb{X}^{\delta}$ defined by
\begin{equation}\label{eq:(intro)spde_auxiliary[TV4AC-25]}
        \mathbb{X}^{\delta}(t) = E(t) \mathbb{X}^{\delta}_0
               +
               \int_0^t E(t-s) 
               \epsilon^{-1}
               F_\delta( \mathbb{X}^{\delta}(s) ) \,\mathrm{d}s 
               +
               \int_0^t E(t-s) \,\mathrm{d}W(s),
               \ 
               t \geq 0,
    \end{equation}
and
decompose the weak error into the following two parts:
\begin{equation}
\label{eq:error-decomposition-introduction}
   | 
   \mathbb{E}
   [ \varphi(X(t_m))] 
   -
   \mathbb{E} 
   [ \varphi 
   ( X^{\tau}_{t_m} ) 
   ] 
   | 
    \leq
    \underbrace{
    |
       \mathbb{E}
    [ 
    \varphi(X(t_m)) 
    ]  
    - 
    \mathbb{E}
    [  
    \varphi
    (\mathbb{X}^{\delta}(t_m))
    ] 
    |
    }_{=: \text{Error}_{1}}
    +
    \underbrace{
    | 
    \mathbb{E}
    [ 
    \varphi
    (\mathbb{X}^{\delta}(t_m)) 
    ]
    -
    \mathbb{E}
    [ 
    \varphi 
    ( X^{\tau}_{t_m} ) 
    ] 
    | 
    }_{=: \text{Error}_{2}},
\end{equation}
where the non-smooth test functions $\varphi \in  \mathcal{C}_b^0(H)$ with $ \|\varphi \|_0 \leq 1$, according to the definition of the TV distance \eqref{tvDistance[TV4AC-25]}.
%
%Observing that the modification $F_{\delta}$ defined by \eqref{eq:(intro)example_f_tau[TV4AC-25]},
%{\color{red}  with $\tau$ therein replaced by $\delta$,}
%{\color{blue}
Here the modification   
$F_{\delta}$ is defined analogously to \eqref{eq:(intro)example_f_tau[TV4AC-25]} with
%{\color{red}
particular choices of parameters $\alpha, \beta, \theta$ 
%$\alpha, \beta =1, \theta = \tfrac12$ fixed, 
and
%}
$\tau$ therein replaced by $\delta$.
%}
As the modification
$F_{\delta}$
is globally Lipschitz continuous (cf. Lemma \ref{lem:f'_delta[TV4AC-25]}), one can rely on the use of  the Kolmogorov equation to estimate both error terms in \eqref{eq:error-decomposition-introduction}.
A major technical obstacle stems from the regularization estimates of the solution $\nu^\delta$ to the associated Kolmogorov equations in the context of non-smooth test functions and the sharp interface limit.
A direct use of the Bismut-Elworthy-Li formula ensures a rough bound for $D \nu^{\delta} (t, x)$ being of order $O\Big(\frac{\exp(\epsilon^{-1}t)}{t^{1/2}} \Big)$ (see Lemma \ref{lem:markov_exponential_estimate_exp[TV4AC-25]}). Despite the exponential dependence on both
$t$ and $\epsilon^{-1}$,  the rough bound is enough for the first error term in \eqref{eq:error-decomposition-introduction}, by noting the error
$
\text{Error}_{1} = \big| \mathbb{E} \big[ \nu^\delta \big(0, X(t_m) \big) \big] - \mathbb{E} \big[ \nu^\delta \big(t_m, X_0 \big) \big] \big| = O \big( \exp(\epsilon^{-1}t_m) \sqrt{\delta} \big)
$
vanishes as $\delta \to 0$.
Using the rough bound of $D \nu^{\delta} (t, x)$ and  based on more careful estimates with the ergodicity of \eqref{eq:(intro)spde_auxiliary[TV4AC-25]},
one can obtain improved regularization estimates of $D \nu^{\delta}$ with polynomial dependence on $\epsilon^{-1}$ (Proposition \ref{prop:markov_exponential_estimate_pol[TV4AC-25]}).
%Consequently, by virtue of the continuous version of \eqref{eq:(intro)time_discretization[TV4AC-25]} (see Proposition \ref{prop:X^tau(t)-bound[TV4AC-25]} and Lemma \ref{lem:X^tau-holder-conti[TV4AC-25]}) and the Kolmorogov equation,
The refined regularization estimates for $D \nu^{\delta} (t, x)$ thus leads us to good estimates of
$\text{Error}_{2}$ in \eqref{eq:error-decomposition-introduction} with order
$
O
(
%\epsilon^{ - \frac{q+1+2\gamma}{2} -\kappa }
\text{Poly}(\epsilon^{-1})
\tau^{ \gamma \wedge \theta }
) 
$
(see Theorem \ref{Thm_main_sharpinterface[TV4AC-25]}).

~

Our main contributions can be summarized as follows:
\begin{itemize}
\item 
%By incorporating several method parameters, 
A novel explicit and easily implementable time-stepping scheme 
%that enjoys a flexible taming strategy
is designed for SPDEs with non-globally Lipschitz nonlinearity and space-time white noise. 
The scheme is new, even for deterministic Allen-Cahn equations. Uniform-in-time and uniform-in-$\epsilon$ moment bounds of the scheme are established.
As clarified before, by carefully choosing proper method parameters, the analysis of the newly proposed scheme relies on an $\epsilon$-independent restriction on the time step-size $\tau$, instead of the severe restriction $\tau = O (\epsilon^\sigma), \sigma \geq 1$ in the literature.

\item 
%Total variation error bound is established near the sharp interface limit $\epsilon \to 0$ (see Theorem \ref{Thm_main_sharpinterface[TV4AC-25]}), which shows for any $\gamma < \frac12$,
Total variation error bounds are
established  near the sharp interface limit $\epsilon \to 0$ (Theorem \ref{Thm_main_sharpinterface[TV4AC-25]}) with  a polynomial dependence on $\epsilon^{-1}$ and the time length $t_m$:
\begin{equation}
\label{eq:(intro)error_bound_sharp_interface[TV4AC-25]}
    d_{\mathrm{TV}}
    (
     \mathrm{law}( X^{\tau}_{t_m} )
     ,
     \mathrm{law} (X(t_m) )
    )
    =
    O\big( t_m \cdot \text{Poly}(\epsilon^{-1}) \tau^{\gamma} \big),
%    \leq
%    C
%     (1 + t_m)
%     \epsilon^{ - \frac{q+1+2\gamma}{2} -\kappa }
%     \tau^{ 
%     \gamma
%     },
     \quad
     \gamma \in (0, \tfrac12).
\end{equation}
This error bound admits a  convergence rate of order $\gamma$ for any $ \gamma < \tfrac12$, twice that of convergence in TV distance in \cite[Theorem 5.6]{Cui_Sun2024weaksharp}.
%{\color{red}
%\item
%Step-size constraints (for the sharp-interface limit case) are relaxed
%through flexible parameter tuning 
%and even removed entirely under some specific choice of scheme parameters 

%(see Remark \ref{rem:tau_epsilon_condition[TV4AC-25]}).
%Step-size restrictions are relaxed through flexible parameter tuning,
%the constraint on $\tau$ can be weakened by adjusting $\alpha$, 
%and can even be entirely removed under the scaling $\beta^{\alpha} = o(\epsilon^{-1})$.
%}

\item 
A uniform-in-time error bound is obtained for $\epsilon = 1$ 
%(see 
(Corollary \ref{cor_main_uniform[TV4AC-25]}):
\begin{equation}
\label{eq:(intro)error_bound_uniform[TV4AC-25]}
    d_{\mathrm{TV}}
    (
     \mathrm{law}( X^{\tau}_{t_m} )
     ,
     \mathrm{law} (X(t_m) )
    )
     %C( X_0 ,q, \gamma, \theta)
     %, \epsilon)
     = 
     O
     \big(
     \tau^{ 
     \gamma
     } 
     \big)
     ,
     \quad
     \gamma \in (0,\tfrac12).
\end{equation}
Such a uniform-in-time error bound extends the existing ones in the globally Lipschitz regime \cite{brehier2024FCM,Brehier2025PA} to a non-globally Lipschitz setting.

\end{itemize}

The rest of this article is organized as follows. The next section presents some preliminaries. The explicit time-stepping scheme is introduced in Section \ref{sec:scheme[TV4AC-25]} and the uniform-in-time moment bound of the proposed scheme is proved in Section \ref{sec:moment-boundedness[TV4AC-25]}.
In Section \ref{sec:Kolmogorov_equation[TV4AC-25]}, an auxiliary process is introduced and the regularity estimates of its Kolmogorov equation are established.
Then we show main convergence results in TV distance in Section \ref{sec:convergence_analysis[TV4AC-25]}.
%after discussing the moment bounds of the continuous version of the scheme, the main convergence result in total variation distance  is illustrated by Theorem \ref{thm:X^tau-bound[TV4AC-25]}. For the case of fixed $\epsilon$, the uniform-in-time convergence in total variation distance is also presented by Corollary \ref{cor_main_uniform[TV4AC-25]}.
The numerical experiments are performed in Section \ref{sec:Numerical_experiments[TV4AC-25]} to verify the theoretic findings.

\section{Settings and the considered SPDEs}
\label{sec:SPDE_considered[TV4AC-25]}

\subsection{Notation}

Let $\N$ be the set of all positive integers and denote $\N_0 := \N \cup \{ 0\}$.
Define $a \wedge b := \min\{ a,b\}$, $a \vee b := \max\{ a,b\} $ for $ a,b \in \R$.
%and define $\tfrac{0}{0} :=1$.
%
%For a bounded domain $\mathcal{D} \subset \R$,
%{\color{cyan}
%by $L^r(\mathcal{D}; \R)$ ($L^r(\mathcal{D})$ or simply $L^r$) we denote the Banach space of $r$-integrable functions, equipped with the norm $\|\cdot\|_{L^r}$. In particular, let $H := L^2(\mathcal{D}; \R)$ be the real, separable Hilbert space endowed with the inner product $\langle \cdot, \cdot \rangle$ and norm $\|\cdot\| := \langle \cdot, \cdot \rangle^{1/2}$.
%}
By $L^r(\mathcal{D}; \R)$ ($L^r(\mathcal{D})$ or simply $L^r$) we denote the Banach space of $r$-integrable functions, equipped with the norm $\|\cdot\|_{L^r}$.
In particular, let $H := L^2(\mathcal{D}; \R)$ be the real, separable Hilbert space endowed with the inner product $\langle \cdot, \cdot \rangle$ and norm $\|\cdot\| := \langle \cdot, \cdot \rangle^{1/2}$.
%{\color{red}
%The space of bounded linear operators from $H$ to $H$ will be
%denoted by $\mathcal{L}(H)$,
%endowed with the usual operator norm $\| \cdot \|_{\mathcal{L}(H)}$. 
%}
Moreover,
we use $\mathcal{L}(H)$ to denote the Banach space of bounded linear operators on $H$, equipped with the operator norm $\|\cdot\|_{\mathcal{L}(H)}$.
 By $\mathcal{L}_2(H) \subset \mathcal{L}(H)$ ($\mathcal{L}_2$ for short), we denote the subspace consisting of all Hilbert-Schmidt operators from $H$ to $H$, which is also a separable Hilbert space, endowed with the scalar product $\langle \Gamma_1, \Gamma_2 \rangle_{\mathcal{L}_2(H)} := \sum_{n \in \N} \langle \Gamma_1 \eta_n, \Gamma_2 \eta_n \rangle$ and the norm $\| \Gamma \|_{\mathcal{L}_2(H)} 
 :=
 \left( \sum_{n \in \N}
 \| \Gamma  \eta_n \|^2 \right)^\frac{1}{2}$, 
independent of the choice of the orthogonal basis $\{ \eta_n\}_{n \in \N}$ of $H$. 
Also, we denote the Banach space consisting of continuous functions by $V := C(\mathcal{D}, \R)$,
endowed with the usual norm $\| \cdot \|_V$.
By $\mathbbm{1}_{ S }$ we denote the indicator function of the set $S$.
%By $\mathbbm{1}_{ \{x \in S\} }$ we denote the indicator function of the set $S$, satisfying $\mathbbm{1}_{ \{x \in S\} } = 1$ when $x \in S$, or otherwise, $\mathbbm{1}_{ \{x \in S\} } = 0$.
%{\color{red}
%Let $C$ or $C(a_1,a_2,...,a_n)$ be the positive constants that may change at different occurrences, while $C(a_1,a_2,...,a_n)$ further shows the dependence of parameters $a_1,a_2,..., a_n$.
%}
We also clarify the notation for Fr\'{e}chet derivatives: for a mapping $\phi: H \rightarrow \mathbb{R}$, its first derivative $D\phi(x) \in \mathcal{L}(H, \mathbb{R})$ and second derivative $D^2\phi(x) \in \mathcal{L}(H, \mathcal{L}(H, \mathbb{R}))$ are defined by
$$
D\phi(x).h 
:=
\langle
D\phi(x)
,
h 
\rangle 
\
\text{and}
\
D^2\phi(x).
(h, k)
:=
\langle
D^2\phi(x).h
,
k
\rangle
, 
\quad 
\forall\, h, k \in H.
$$
Throughout the paper,
%Here and throughout, 
we denote
$C$ as a generic positive constant, which may change from line to line and whose dependence on parameters will be indicated in the notation $C(\cdot)$.
%To present a concrete relationship with the parameters, 
%
Unless otherwise stated, 
%{\color{red}
%the constants $C$ and $C(a_1,a_2,...,a_n)$ are independent of $\epsilon^{-1}$ and the step-size $\tau$.
%}
all such constants are independent of $\epsilon$ and the step-size $\tau$.
%
%In what follows, we make some introductions about the total variation distance and the considered model.

Next we recall the TV distance between two Borel probability distributions $\mu_1$ and $\mu_2$ on $H$, defined by
\begin{equation}
d_{\mathrm{TV}} (\mu_1, \mu_2 )
:= 
\sup_{\varphi \in \mathcal{B}_b(H), \|\varphi \|_0 \leq 1  }
\left| \int \varphi(x) \mu_1 (\mathrm{d}x) - \int \varphi(x) \mu_2 (\mathrm{d}x) \right| ,
\end{equation}
where we denote the set of all bounded and measurable mappings from $H$ to $\R$ by $\mathcal{B}_b(H)$,
and $ \| \varphi \|_0 := \sup _{x \in H} |\varphi(x)| $ .
By $\mathcal{C}^0_b(H)$, we denote the set of bounded and continuous mappings from $H$ to $\R$ and by $\mathcal{C}^1_b(H)$ the subspace of $\mathcal{C}^0_b(H)$ consisting of all functions with bounded first order derivatives.
%Denote $ \| \varphi \|_1 := \sup _{x \in H}  \sup _{h \in H, \|h\| \leq 1 } |D \varphi(x). h| $.
%
%{\color{red}
%Thanks to the fact that any bounded and measurable function
%$\varphi \in \mathcal{B}_b(H)$ 
%can be approximated by a series of bounded and continuous functions, more precisely, there exists a sequence
%$(\varphi_k)_{k \in \N}$
%with 
%$\varphi_k \in \mathcal{C}^0_b(H), k \in \N$
%which satisfies 
%$\sup _{k \in \N} \sup_{x \in H} | \varphi_k(x) | < %\infty$ 
%and for all 
%$x \in H$
%there will be
%$\varphi_k \rightarrow \varphi$ when $k \rightarrow \infty$ {\color{red}\textbf{Reference!}}. We can check the following equivalent representation of the total variation distance (see, e.g., \cite{brehier2024SiamTV,Brehier2025PA}):
%}
Since every bounded, measurable function $\varphi \in \mathcal{B}_b(H)$ can be approximated pointwise by a sequence $(\varphi_k)_{k \in \N}$ in $\mathcal{C}_b^0(H)$, where $\sup_{k \in \N} \sup_{x \in H} |\varphi_k(x)| < \infty$ and $\varphi_k(x) \to \varphi(x)$ as $k \to \infty$ for all $x \in H$, the TV distance between two Borel probability distributions $\mu_1$ and $\mu_2$ can be equivalently rewritten as (see, e.g., \cite{brehier2024SiamTV,Brehier2025PA}):

\begin{equation}\label{eq:TV-weakformula[TV4AC-25]}
   d_{\mathrm{TV}} (\mu_1, \mu_2 )  
    = 
    \sup_{\varphi \in \mathcal{C}_b^0(H),  \| \varphi \|_0 \leq 1}
     \left| \int \varphi(x) \mu_1(\mathrm{d}x) - \int \varphi(x) \mu_2(\mathrm{d}x) \right| .
\end{equation}
%{\color{red}
%for all Borel probability distributions $\mu_1$ and $\mu_2$.
%}
%Moreover, 
%{\color{red}
%if $X$ is a $H$-valued random variable, of which the distribution is denoted by $\text{law}(X)$, then for any Borel probability distribution $\mu$ on $H$, one has
%}
Given an $H$-valued random variable $X$ and a Borel probability measure $\mu$ on $H$, the TV distance between the law of $X$, denoted by $\text{law}(X)$, and $\mu$ is thus given by
\begin{equation}\label{tvDistance[TV4AC-25]}
\begin{split}
   d_{\mathrm{TV}} ( \text{law}(X), \mu  )  
    &= 
    \sup_{ \varphi \in \mathcal{B}_b(H), \|\varphi \|_0 \leq 1 }
    \left| \E [\varphi(X)] - \int \varphi(x) \mu (\mathrm{d}x) \right| \\
    &=
    \sup_{ \varphi \in  \mathcal{C}_b^0(H), \|\varphi \|_0 \leq 1 }
   \left| \E [\varphi(X)] - \int \varphi(x) \mu (\mathrm{d}x) \right|  .
\end{split}
\end{equation}

\subsection{Main assumptions and the well-posedness of SPDEs}

Throughout this paper, we focus on the following  parabolic SPDEs in the Hilbert space $H$:
\begin{equation}\label{eq:considered_SEE[TV4AC-25]}
\left\{\begin{array}{l}
\,\mathrm{d} X(t)  = - A X(t) \,\mathrm{d} t + \epsilon^{-1} F(X(t)) \,\mathrm{d} t + \,\mathrm{d} W(t), 
\quad t > 0, \\
X(0) = X_0,
\end{array}\right.
\end{equation}
%{\color{red}Here,
%the parameter $\epsilon \in (0,1]$. the operators $A,F$, the noise process $W$ and the initial value $X_0$ satisfy some main assumptions as follows.
%}
%{\color{cyan}Here, $\epsilon \in (0,1]$, and $A$, $F$, $W$, and $X_0$ are as specified below, where some main assumptions are imposed to ensure the well-posedness and to facilitate our subsequent analysis.}
where $\epsilon \in (0,1]$ is the diffuse interface thickness, and $A$, $F$, $W$, and $X_0$ are as specified below.  
To ensure the well-posedness of \eqref{eq:considered_SEE[TV4AC-25]} and to facilitate our subsequent analysis, we impose the following assumptions.

%%%% Linear Operator A %%%%%%

\begin{assumption}[Linear Operator $A$]\label{assump:A(linear_operator)[TV4AC-25]}
%Let $\mathcal{D} \subset \R$ be a bounded spatial domain with smooth boundary.
Let $\mathcal{D} := (0, 1)$ and let $H := L^2(\mathcal{D}; \R)$.
Let $-A \colon Dom(A) \subset H \rightarrow H$ be the Laplacian on $\mathcal{D}$ with homogeneous Dirichlet boundary conditions, i.e., $-Au = \Delta u$, 
$u \in Dom(A) := H^2(\mathcal{D}) \cap H^1_0(\mathcal{D})  $.
\end{assumption}

%{\color{red}
%Note that Assumption \ref{assump:A(linear_operator)[TV4AC-25]} ensures the existence of
%}
Under this assumption,
the operator $A$ admits an eigensystem
$\left\{\lambda_j, e_j\right\}_{j \in \mathbb{N}}$ in $H$ satisfying
$ A e_j = \lambda_j e_j$,
with
$\left\{ \lambda_j \right\}_{j \in \mathbb{N}}$ being an increasing sequence such that 
$\lambda_j \sim  j^2$.
%$\lambda_j \sim c_A j^2$.
%for some constant $c_A \in (0, \infty)$.
%and
%$\lim_{j \rightarrow \infty} \lambda_j = \infty$.
Further, it  holds
\begin{equation}\label{eq:A_HS-norm_gamma[TV4AC-25]}
%\Big\| A^{ \frac{\gamma-1}{2} } \Big\|_{\mathscr{L}_2^0}
%: =
 \big\| A^{ \frac{\gamma-1}{2} } 
\big\|_{ \mathcal{L}_2(H) }
< 
  \infty,
\quad
\text{for any } \gamma < \tfrac12.
\end{equation}
In addition, $-A$ generates an analytic and contractive semi-group, denoted by 
$ E(t):=e^{-At},t \geq 0$.
By means of the spectral decomposition, we define the fractional powers of $A$, i.e., $A^\vartheta$ for $\vartheta \in \mathbb{R}$ \cite[Appendix B.2]{Kruse2014}. 
Denote the interpolation spaces by $\dot{H}^\vartheta := Dom( A^{\frac{\vartheta}{2}} ), \vartheta \in \mathbb{R}$, which are separable Hilbert spaces equipped with the inner product 
$\langle \cdot, \cdot \rangle_\vartheta := \langle A^{\frac{\vartheta}{2}} \cdot, A^{\frac{\vartheta}{2}} \cdot \rangle$ 
and the norm
$\| \cdot \|_\vartheta := \| A^{\frac{\vartheta}{2}} \cdot \| = \langle \cdot, \cdot \rangle_\vartheta^{1/2}$. 
The following regularity properties are well-known (see e.g. \cite{Pazy1983}):
for any $t>0, \vartheta \geq 0, \varsigma \in[0,1]$,
\begin{equation}\label{eq:E(t)_semigroup_property[TV4AC-25]}
    \left\|E(t) \right\|_{\mathcal{L}(H)} 
    \leq 
    e^{-\lambda_1 t}
,
\quad
    \left\| A^{\vartheta} E(t) \right\|_{\mathcal{L}(H)} 
    \leq 
    C t^{-\vartheta}
, 
%\quad
%    \left\| A^{\varsigma} E(t) \right\|_{\mathscr{L}(H)} 
%    \leq 
%    C e^{- \lambda_1 (1 - \varsigma )t} t^{-\varsigma},
\quad
    \left\| A^{-\varsigma} (I - E(t)) \right\|_{\mathcal{L}(H)} 
    \leq 
    C t^{\varsigma}.
\end{equation}
As indicated by \cite[(2.4)]{brehier2022ESAIM} and \cite[(6.2)]{thomee2007galerkin}, for any $t >0$,
\begin{equation}\label{eq:E(t)_semigroup_property_V[TV4AC-25]}
    \|E(t)x\|_{V}
    \leq
    C
    (1 \wedge t)^{-\frac14}
    e^{-ct} 
    \| x \| ,
    \
    \forall
    x
    \in 
    H,
    \quad
    \text{and}
    \quad
    \| E(t) x\|_{V} \leq \|x \|_V,
    \ 
    \forall
    x 
    \in
    V.
\end{equation}
%Moreover, the Sobolev embedding inequality illustrates that for $\vartheta \in (\frac{d}{4},1]$, for any $x \in V$ it holds that
%\begin{equation}\label{eq:Sobolev_embedding[TV4AC-25]}
%    \| x \|_V \leq C \| A^{\vartheta} x \|.
%\end{equation}

%%%%%% noise process %%%%%%.

\begin{assumption}[Noise Process]\label{assump:W(noise)[TV4AC-25]}
Let $\{W(t)\}_{t \geq 0}$ be a cylindrical $I$-Wiener process with respect to a filtered probability space $(\Omega, \mathscr{F}, \left\{\mathscr{F}_t\right\}_{t \in [0, \infty)}, \P)$, 
represented
%{\color{blue} admitted}
by a formal series
\begin{equation}
    W(t) := \sum_{n = 1}^{\infty} \widetilde{\beta}_n(t) e_n, \quad t \geq 0,
\end{equation}
where $\big\{ \widetilde{\beta}_n(t)\big\}_{n \in \N}, t \geq 0$, is the sequence of independent real-valued standard Brownian motions adapted to the filtration $\left\{\mathscr{F}_t\right\}_{t \geq 0}$, and $\left\{e_n\right\}_{n \in \N}$ is the complete orthonormal basis of $H$.
\end{assumption}

%%%% Nonlinearity F %%%%%%

\begin{assumption}[Nonlinearity]\label{assump:F(Nonlinearity)[TV4AC-25]}
Let 
$q > 1$ be any integer
%$q =1,2,3$ 
and let $F\colon L^{4 q-2}(\mathcal{D}) \rightarrow H$ be a nonlinear Nemytskii operator given by 
\begin{equation}
    F(u)(x) := f( u(x) ), 
    \quad
    x \in \mathcal{D},
\end{equation}
where $f (v) = -c_f v^{2q-1} + f_0(v), v \in \R$ with $c_f >0$ and $f_0 \colon \R \rightarrow \R$ being  twice differentiable. 
%{\color{red}
%In addition, the mapping $f_0$ satisfies that for all $v \in \R$, there exist constants $c_{f,0}, c_{f,1} > 0$ such that
%}
In particular, there exist constants $c_{f,0}, c_{f,1} > 0$ such that
\begin{align}
   |f_0(v)| 
   &\leq
   c_{f,0}( 1 + |v|^{2q-2}),  \quad \forall v \in \R,
   \\
   |f_0'(v)|+|f_0''(v)|
   & \leq  
   c_{f,1}
   (
   1
   +
   %\mathbbm{1}_{ \{q >1 \} }
   |v|^{2q-3} 
   ), \quad \forall v \in \R.
\end{align}

\end{assumption}

By Assumption \ref{assump:F(Nonlinearity)[TV4AC-25]}, one can find constants $L_f \in \R$ and $ c_0, c_1, c_2, c_3, c_4, c_5 >0$ such that, for any $u, v \in \R$,
\begin{align}
\label{eq:asuume_f'(u)[TV4AC-25]}
 f'(u) 
  & \leq
  L_f ,
\\
\label{eq:asuume_(u+v)f(u)[TV4AC-25]}
   (  u + v ) 
    f( u )
   & \leq
    - c_0
     |u|^{2q}
    + c_1
     |v|^{2q}
    + c_2,
%\\
%\label{eq:asuume_|f(u)-f(v)|[TV4AC-25]}
%  | f( u) - f( v ) |
%   & \leq
%   c_3
%   ( |u|^{2q-2} + |v|^{2q-2} ) | u - v |
%   + 
%   c_4 
%   | u - v |.
\\
\label{eq:asuume_|f(u)|[TV4AC-25]}
    |f(u)| 
   &
    \leq
      c_3 |u|^{2q-1} 
      +
      c_4 |u|
      +
      c_5 .
\end{align}
%
%Also, by setting $v=0$ in \eqref{eq:asuume_|f(u)-f(v)|[TV4AC-25]}, it is straightforward that for all $u \in \R$,
%\begin{equation}
%\end{equation}
%where the constant $c_5 := |f(0)| \geq 0$.
%
%{\color{red}
%A typical example of the nonlinearity $f$ satisfying Assumption \ref{assump:F(Nonlinearity)[TV4AC-25]} is 
%$f(u) := a_0 + a_1 u + a_2 u^2 + a_3 u^3$
%    with $a_3 < 0, a_0,a_1,a_2 \in \R$.
%Such SPDEs are called stochastic Allen–Cahn equations in the literature \cite{wang2020efficient,Cai2021weak4ACE,CUI2021weak,Brehier_Allen_C,QiWang2019optimal,HuangC_ShenJ2023Mathcomp}.
%}
A typical example fulfilling Assumption \ref{assump:F(Nonlinearity)[TV4AC-25]} is the cubic nonlinearity 
(i.e., the case $q=2$)
$f(u) := a_0 + a_1 u + a_2 u^2 + a_3 u^3$
 with $a_3 < 0, a_0,a_1,a_2 \in \R$.
Such SPDEs with cubic nonlinearity are commonly termed as stochastic Allen–Cahn equations (\cite{CUI2021weak,Brehier_Allen_C,HuangC_ShenJ2023Mathcomp,feng2017finite}).

\begin{assumption}[Initial value]\label{assump:X_0(Initial Value)[TV4AC-25]}
    Let the initial value $ X_0\colon \Omega \rightarrow H $ be an $ \mathscr{F}_0 / \mathcal{B}(H) $-measurable random variable.
    For any $ p \geq 1 $ and for some
    $\varrho > \tfrac{ 1 }{2} $, there exists a constant $ C(\varrho,p)>0 $ depending on $ \varrho, p $ such that
\begin{equation}
    \| X_0 \|_{ L^p( \Omega; \dot{H}^{ \varrho  }) } 
    \leq
     C(\varrho, p) 
    < 
    \infty.   
\end{equation}
    
\end{assumption}

The next lemma provides uniform-in-time regularity properties of
the stochastic convolution 
%{\color{red}
%which have been extensively studied, with the aid of the Sobolev embedding inequality.
%}
%which are obtained via the standard Sobolev embedding arguments 
(see, e.g., the proof of \cite[Lemma 2.2]{Cui_Sun2024weaksharp}).

\begin{lem}\label{lem:O_bound[TV4AC-25]}
     Let Assumptions \ref{assump:A(linear_operator)[TV4AC-25]}, \ref{assump:W(noise)[TV4AC-25]} be fulfilled. For any $p \in[2, \infty)$ and $\gamma \in\left[0, \frac{1}{2}\right)$, there exist constants
     $C(p)$ and
     %depending on $p$ and
     $C(p, \gamma)$
     %depending on $p, \gamma$ 
     such that the stochastic convolution 
     \begin{equation}
     \mathcal{O}_t:=\int_0^t E(t-s) \,\mathrm{d} W(s), \quad t \geq 0
     \end{equation}
     obeys
\begin{equation}
\begin{aligned}    
      \sup_{t \geq 0}
      \big\|
         \mathcal{O}_{t} 
        \big\|_{L^p( \Omega; V )} 
     \leq
      C(p)
      <
      \infty,
\qquad 
        \sup_{t \geq 0} \big\| 
        \mathcal{O}_{t} 
        \big\|_{L^p( \Omega; \dot{H}^{\gamma} )}
       \leq
        C( p ,\gamma)
        <
        \infty.
%\\
%&
%\left\|\mathcal{O}_t-\mathcal{O}_s\right\|_{L^p(\Omega, H)} 
% \leq
%C(p)
%\left\|A^{\frac{\gamma-1}{2}}\right\|_{\mathscr{L}_2(H)}(t-s)^{\frac{\gamma}{2}}, \quad   t>s \geq 0 .
\end{aligned}
\end{equation}

\end{lem}

%{\color{red}
%Under all the assumptions stated above, the well-posedness and the moment bounds of the SPDEs \eqref{eq:considered_SEE[TV4AC-25]} are straightforward, which are stated as follows (see, e.g., \cite[Chapter 6]{cerrai2001second} and \cite[Lemma 2.2]{Cui_Sun2024weaksharp}).
%}

%Based on the assumptions stated above, together with the regularity properties of the stochastic convolution,
%Based on the assumptions and  regularity properties of the stochastic convolution above,
As a consequence,
we obtain the well-posedness of the mild solution to SPDEs \eqref{eq:considered_SEE[TV4AC-25]} as follows (see, e.g., \cite[Chapter 6]{cerrai2001second} and \cite[Lemma 2.2]{Cui_Sun2024weaksharp}).
%, \cite[Theorem 3.5]{wang2024IMA} 
%and \cite[Lemma 2.2]{Cui_Sun2024weaksharp}).

%%%%% wellposedness %%%%%%%%

%{\color{red} uniform V-Bound of exact solution (see Lemma 2.2 in Cui-Sun)}

\begin{thm}
    Suppose Assumptions \ref{assump:A(linear_operator)[TV4AC-25]}-\ref{assump:X_0(Initial Value)[TV4AC-25]} are satisfied. Then, the model \eqref{eq:considered_SEE[TV4AC-25]} admits a unique mild solution $\{ X(t) \}_{t \geq 0}$ with continuous sample paths defined by
\begin{equation}\label{eq:spde[TV4AC-25]}
        X(t) = E(t) X_0
               +
               \int_0^t E(t-s) \epsilon^{-1}
               F( X(s) ) \,\mathrm{d}s 
               +
               \mathcal{O}_t,
               \ 
               t \geq 0,
        \ 
        \mathbb{P} \text{-a.s.}
\end{equation}
Moreover, 
%{\color{red}
%for any $p \geq 2$, there exists a constant $C(p,q) >0$ such that for any $t\geq 0$,
%}
for any $p \geq 2$,
there exists a constant $C(p,q) >0$ such that
\begin{equation}\label{eq:bound_X(t)[TV4AC-25]}
    \sup_{t\geq 0} 
    \|X(t) \|_{L^p( \Omega; V )}
    \leq
    C(p,q)( 1 + \|X_0 \|_V ).
%    C(p,q)( \|X_0 \|_V + \epsilon^{-1} ).
\end{equation}
\end{thm}

%{\color{red}
%Since the $I$-Wiener process satisfies the non-degenerate condition, an exponential convergence to equilibrium for the SPDE \eqref{eq:spde[TV4AC-25]} can be attained (see, e.g., \cite[Chapter 7]{cerrai2001second},  \cite[Lemma 4.6]{Cui_Sun2024weaksharp}). 
%}
%The last assertion can, e.g., be found in \cite[Lemma 2.2]{Cui_Sun2024weaksharp}.
Owing to the non-degeneracy of the $I$-Wiener process, the SPDE \eqref{eq:spde[TV4AC-25]} admits the exponential convergence to the equilibrium (see, e.g., %\cite[Chapter~7]{cerrai2001second}, 
\cite[Lemma~4.6]{Cui_Sun2024weaksharp}).
%which is formalized in the following proposition.

\begin{prop}\label{prop:V-uniform_ergodicity_X(t)[TV4AC-25]}
    Let Assumptions \ref{assump:A(linear_operator)[TV4AC-25]}-\ref{assump:F(Nonlinearity)[TV4AC-25]} hold.
    By $X(t,x), t\geq 0$ we denote the unique solution of \eqref{eq:spde[TV4AC-25]} that initiates at $x \in H$.
    Then there exists a constant $C > 0$ independent of $\epsilon^{-1}$ 
    %and a constant $r> 0$ 
    such that for any $\varphi \in \mathcal{B}_b(H)$, $t \geq 0$, and $x_1, x_2 \in V \cap \dot{H}^{\gamma}$, it holds
    \begin{equation}
        \left| 
        \E
        \left[ \varphi( X(t,x_1) ) \right] 
        - 
        \E\left[ \varphi( X(t,x_2) ) \right] \right|
        \leq
        C
        \| \varphi \|_{0}
        e^{- r t},
    \end{equation} 
%{\color{red} maybe lack of something? $\|x_{1}-x_{2}\|$?}
%
where the rate $r := -\tfrac{1}{2}  \log ( 1 - \vartheta ) $
with $ e^{-\widetilde{C}  \frac{1}{\epsilon^2} }
\leq
\vartheta 
\leq
e^{-\widetilde{c} \frac{1}{\epsilon^2} } $ 
for some 
$\widetilde{C} \geq \widetilde{c} \geq 0 $.
\end{prop}

%We also quote from \cite[Remark 4.7]{Cui_Sun2024weaksharp} that the constant $r > 0$ in Proposition \ref{prop:V-uniform_ergodicity_X(t)[TV4AC-25]} is estimated by $r = -\tfrac{1}{2} log ( 1 - \vartheta ) $ where $ e^{-\widetilde{C}  \frac{1}{\epsilon^2} } \leq \vartheta  \leq e^{-\widetilde{c} \frac{1}{\epsilon^2} } $ for some $\widetilde{C} \geq \widetilde{c} \geq 0 $.

\section{The proposed explicit time-stepping scheme}
\label{sec:scheme[TV4AC-25]}

In this section, we aim to design an explicit time-stepping scheme
to approximate the underlying problem \eqref{eq:considered_SEE[TV4AC-25]} temporally.
For a uniform step-size $\tau > 0 $,
we propose the following scheme:

%We now propose an explicit time-stepping scheme as follows:
\begin{equation}\label{eq:time_discretization[TV4AC-25]}
    X^{\tau}_{t_{m+1}}
    = 
    E(\tau) X^{\tau}_{t_m} 
    + 
    \tau E(\tau)
    \epsilon^{-1}
    F_{\tau}
    \big( X^{\tau}_{t_m} \big)
    +
    \int_{t_m}^{t_{m+1}}
    E(t_{m+1}-s) \,\mathrm{d} W(s),
    \quad
    X_0^{\tau} =X(0),
\end{equation}
%{\color{red} for $m \in \N_0$,}
where
we denote $t_m := m \tau$, 
$m \in \N_0$.
%for the uniform step-size $\tau > 0 $,
%$\Delta W_m := W(t_{m+1}) - W(t_m)$, $ m \in \N_0$,
%choices of parameters
Here, $F_\tau \colon L^{4q-2}(\mathcal{D}) \to H$ is the Nemytskii operator associated with the function $f_\tau \colon \mathbb{R} \to \mathbb{R}$:
\begin{equation}\label{eq:f_tau[TV4AC-25]}
F_{\tau}(u)(x)
        : =
        f_{\tau}( u(x) ), \quad x \in \mathcal{D},
\quad
       f_{\tau}( v )
       :=
       \frac{ f( v ) }
       { \big(  1 
            + 
            \beta \tau^{\theta}
          |v|^{ \frac{2q-2}{\alpha} } \big)^\alpha
        },
        \quad
        v \in \R,
    \end{equation}
where 
$\beta, \theta > 0$, $\alpha \in (0, \tfrac{1}{\theta})$.
%$\alpha, \beta, \theta >0$ satisfying $\theta \alpha < 1$.
%{\color{red}
%Moreover,
%we let $\alpha =0$ for the case $q=1$, while
%$\alpha \in (0,1]$ for $q > 1$.
%We mention that when $q=1$ (i.e., $f$ is globally Lipschitz) and hence $\alpha=0$,  the scheme \eqref{eq:time_discretization[TV4AC-25]} reduces to the standard exponential Euler scheme \cite{wang2014DCDSweak}. 
% When $\alpha=\tfrac12$, the underlying scheme is similar to that proposed by \cite{neufeld2025non}.
% For $\alpha=1$, the underlying scheme coincides with those in \cite{sabanis2016euler,angeli2025uniform}. }
%For the modification \eqref{eq:f_tau[TV4AC-25]} above, we let $\alpha =0$ for $q=1$, and $\alpha \in (0,1]$ for $q > 1$.
On different choices of method parameters, 
%the proposed scheme reduces to several existing methods 
the proposed modifications $f_\tau$ of the nonlinearity reduce to several existing ones
for stochastic ordinary differential equations (SODEs).
For example, the modification $f_{\tau}$ with $\alpha = 1/2$ is similar to that in \cite{neufeld2025non} and
taking $\alpha = 1$ makes the  modification $f_{\tau}$ similar to the tamed Euler method considered in \cite{sabanis2016euler, angeli2025uniform}.
 
%\end{itemize}
% When $q = 1$ (that is, when $f$ is globally Lipschitz continuous), so that $\alpha = 0$, the scheme \eqref{eq:time_discretization[TV4AC-25]} reduces to the standard exponential Euler method (see \cite{wang2014DCDSweak}). 
% Furthermore, for $\alpha = 1/2$ and $\beta = 0$, the scheme is similar to that in \cite{neufeld2025non}.
% If $\alpha = 1$ and $\beta = 0$, the scheme coincides with those considered in \cite{sabanis2016euler, angeli2025uniform}.

Throughout the paper,
%we put an additional condition for the time step-size $\tau$ {\color{blue}and} the parameters of the underlying model.
% {\color{red}\textbf{Not Rigorous!}}
we put an additional condition on
the time step-size $\tau$, the diffuse interface thickness $\epsilon$ and the regularization parameters $\beta, \theta, \alpha$ for $f_\tau$ \eqref{eq:f_tau[TV4AC-25]} as follows:
\begin{equation}
\label{eq:epsilon_condition_f_tau[TV4AC-25]}
    2
    c_3
    ^2 
    \tau^{ 1-\theta\alpha } 
    \leq
    c_0
    \beta^{\alpha}
    \epsilon,
\end{equation}
%\begin{equation}\label{eq:epsilon_condition_f_tau[TV4AC-25]}
%    2c_f^2 \tau^{ 1-\theta\alpha } \leq c_0 \beta^{\alpha} \epsilon,
%\end{equation}
where the parameters
%$\beta, \theta, \alpha$ come from \eqref{eq:f_tau[TV4AC-25]} and 
$
c_0,c_3
%,c_4
$ stem from 
\eqref{eq:asuume_(u+v)f(u)[TV4AC-25]}-\eqref{eq:asuume_|f(u)|[TV4AC-25]}.
We mention that a similar condition like \eqref{eq:epsilon_condition_f_tau[TV4AC-25]} was also required in \cite{Cui_Sun2024weaksharp,cui2024improving}. At first sight, such a restriction depending on the small parameter $\epsilon$ is rather strict. However, by carefully choosing proper method parameters, the severe restriction can be significantly relieved and only an $\epsilon$-independent restriction is needed (see Remark \ref{rem:tau_epsilon_condition[TV4AC-25]} and Theorem \ref{Thm_main_sharpinterface[TV4AC-25]} below).
%
%{\color{red}
%$c_f,c_0$ are from 
%Assumption \ref{assump:F(Nonlinearity)[TV4AC-25]} and \eqref{eq:asuume_(u+v)f(u)[TV4AC-25]}.
%}
%Assumption \ref{assump:F(Nonlinearity)[TV4AC-25]}.
%{\color{red} \textbf{Assumption Environment?}}
%

%%%%% balance relationship (old version) %%%
\iffalse
{\color{cyan}
It is worth noting that the condition \eqref{eq:epsilon_condition_f_tau[TV4AC-25]} reveals 
%that the time step-size $\tau$ is restricted to the parameter $\epsilon$
the relationship between the time step-size $\tau$ and the parameter $\epsilon$ that typically depends on the parameters of the scheme, namely, $\alpha, \beta$ and $\theta$.
Specifically, one has
\begin{equation}
    \tau
    =
    O
    \big(
     \beta^{ 
    \frac{\alpha}{1-\theta\alpha}  
    }
     \epsilon^{ 
    \frac{1}{1-\theta\alpha}  
    }
    \big). 
\end{equation}
As shown,
for some desired convergence order $\theta$, 
the choice of the time step-size $\tau$ is flexible by adjusting the parameters $\alpha$ and $\beta$.
}
%
%
%
\fi
%%%%% balance relationship (old version) %%%

\begin{rem}\label{rem:tau_epsilon_condition[TV4AC-25]}

%\textbf{
The restriction \eqref{eq:epsilon_condition_f_tau[TV4AC-25]} 
illustrates how the step-size $\tau$ 
should be scaled with respect to $\epsilon$ 
and the regularization parameters $\alpha$,
$\beta$ 
and the convergence rate $\theta$.
\iffalse
that is,
\begin{equation}
    \tau
    =
    O
    \big(
     \beta^{ \frac{\alpha}{1-\theta\alpha}  }
     \epsilon^{ 
    \frac{1}{1-\theta\alpha}  
    }
    \big). 
\end{equation}
\fi
%More precisely,
In particular,
for $\beta$ and $\theta$ fixed,
%given fixed $\beta$ and some desired convergence order~$\theta$, 
the parameter $\alpha$ can be chosen flexibly to relax the constraint on $\tau$ in sharp-interface simulations. 
%one can adjust $\alpha$ to relax the restriction on the time step-size $\tau$ for simulating in the sharp interface limit.
%the time step-size $\tau$ can be flexibly adjusted through $\alpha$ to relax  the restriction of 
%and $\beta$  
%simulations in the sharp interface limit.
%
Notably, 
%{\color{blue}
%the restriction on $\tau$ vanishes entirely %with respect to $\epsilon$ 
%}
%the restriction on $\tau$ no longer deteriorates as $\epsilon \rightarrow 0$ if one takes
as implied by Theorem \ref{Thm_main_sharpinterface[TV4AC-25]},
by taking $\beta^{\alpha} = O(\epsilon^{-1})$,
%and therefore 
%that relieves \eqref{eq:epsilon_condition_f_tau[TV4AC-25]}
we just need an $\epsilon$-independent restriction
%such that 
$2
    c_3
    ^2 
    \tau^{ 1-\theta\alpha } 
    \leq
    c_0
    $
%serving an $\epsilon$-independent restriction $\tau=O(1)$.
%while
%the limitation \eqref{eq:epsilon_condition_f_tau[TV4AC-25]}
%would be significantly relieved
\iffalse
which
significantly relieves
the requirement
$
    \tau
    =
    O
    \big(
     \beta^{\frac{\alpha}{1-\theta\alpha}}
     \epsilon^{ 
    \frac{1}{1-\theta\alpha}  
    }
    \big)
$
\eqref{eq:epsilon_condition_f_tau[TV4AC-25]}
%namely,
%indicating that
to an $\epsilon$-independent restriction
$\tau=O(1)$
\fi
%if one takes
%under the scaling
%for the choice that
%one acquires an $\epsilon$-independent restriction on $\tau$,
%thereby ensuring
%As a result,
%and consequently,
%{\color{magenta} the restriction on $\tau$ no longer deteriorates as $\epsilon \rightarrow 0$.}
%if one takes $\beta^{\alpha} = O(\epsilon^{-1})$  thereby yielding an $\epsilon$-independent restriction on 
%the time step-size $\tau$.
to promise error estimates still with a polynomial dependence on $\epsilon^{-1}$.
%Moreover, we highlight that the restriction can even be removed with the choice of  $\beta^{\alpha} = o(\epsilon^{-1})$.
%In this setting, the convergence results depending on $\beta$ also reveal a polynomial dependence on $\epsilon^{-1}$.
%
\end{rem}

Under the condition \eqref{eq:epsilon_condition_f_tau[TV4AC-25]}, we have several key properties of the modified function $f_\tau$.

%{\color{blue}
%Under condition \eqref{eq:epsilon_condition_f_tau[TV4AC-25]}, the following proposition establishes several key properties of the modified nonlinearity $f_\tau$.
%}

\begin{prop}\label{prop:F_tau[TV4AC-25]}
      Let Assumption \ref{assump:F(Nonlinearity)[TV4AC-25]} and condition \eqref{eq:epsilon_condition_f_tau[TV4AC-25]}
      hold.
For $\tau >0$, 
the modified function $f_{\tau}$ 
%determined by 
\eqref{eq:f_tau[TV4AC-25]} satisfies  the following conditions:
    there exist constants $ \widetilde{c}_1, \widetilde{c}_2
    > 0
    $ independent of $\beta, \theta, \tau$ and $\epsilon^{-1}$, such that for any $u,v \in \R$,
    \begin{align}
    \label{eq:(u+v)f_tau(u)[TV4AC-25]}
    2 
    ( u + v ) f_{\tau}(u) 
    +
    \tau \epsilon^{-1}
    | f_{\tau}(u) |^2
  &  \leq  
%    \widetilde{c}_0
%    (1 + \beta \tau^\theta)^{-2\alpha}
-
    \tfrac
   { 
   c_0
    }
    { 2(1 + \beta \tau^\theta)^{2\alpha} }
    |u|^2 
    +
    \widetilde{c}_1
    (
    1
    +
    |v|^{2q}
    )
    +
    \widetilde{c}_2
    \beta^\alpha
    \tau^{\theta \alpha}
    ,  
    \\ 
    \label{eq:|F_tau(u)|[TV4AC-25]}
    | f_{\tau}(u) |
   & \leq
    |f(u)|, 
    \\
    \label{eq:|F_tau(u)-F(u)|[TV4AC-25]}
    | f_{\tau}(u) - f(u) |
  &  \leq
    \alpha \beta
    \tau^{\theta}
     |u|^{ \frac{2q-2}{\alpha} }
     |f(u)|,
    \end{align}
where 
$c_0$ is given by \eqref{eq:asuume_(u+v)f(u)[TV4AC-25]} 
and
$\beta , \theta > 0$, $\alpha \in (0, \tfrac{1}{\theta})$ 
come from 
\eqref{eq:f_tau[TV4AC-25]}. 
\end{prop}
\begin{proof}
%{\color{red}
%For simplification, one may rewrite
%}
%For convenience, 
By introducing 
\begin{equation}
\Theta(x) := ( 1 + \beta \tau^{\theta}  x )^{-\alpha}, 
    \quad
    x \geq 0,
\end{equation}
we rewrite $f_{\tau}$ as
\begin{equation}
f_{\tau}(u) = f(u)\Theta(|u|^{ \frac{2q-2}{\alpha} }), \quad u \in \R.
\end{equation}
%
%{\color{red}
%Note that \eqref{eq:|F_tau(u)|[TV4AC-25]} is straightforward since $|\Theta(x)| \leq 1$ for all $x \geq 0$, and \eqref{eq:|F_tau(u)-F(u)|[TV4AC-25]} is also easily obtained by noticing the fact $\Theta(0)=1$, together with
%}
As $|\Theta(x)| \leq 1$ for all $x \geq 0$, the assertion \eqref{eq:|F_tau(u)|[TV4AC-25]} follows immediately.
Noting that $\Theta(0) = 1$ and
\begin{align}
| \Theta'(x) |
    = 
    \left|
    \tfrac{\alpha \beta \tau^{\theta}  }{ ( 1 + \beta \tau^{\theta}  x )^{\alpha+1}  }
    \right|
\leq
  \alpha \beta \tau^{\theta} 
,
\quad
x \geq 0,
\end{align}
we can infer
\begin{align}
        | f_{\tau}(u) - f(u) |
    =
      |f(u)|
      \cdot
      \big|\Theta( |u|^{\frac{2q-2}{\alpha}} ) - \Theta(0)\big|
    \leq
       \alpha \beta \tau^{\theta} |u|^{\frac{2q-2}{\alpha}} |f(u)|
\end{align}
and thus the assertion \eqref{eq:|F_tau(u)-F(u)|[TV4AC-25]} is validated.
%
%
%{\color{red}
%In what follows, it suffices to check 
%\eqref{eq:(u+v)f_tau(u)[TV4AC-25]},
%which must be handled more carefully.
%}
Now
it remains to verify  
\eqref{eq:(u+v)f_tau(u)[TV4AC-25]}, which requires more efforts.
%
%
%{\color{red}
%First, one looks into the case $q=1,\alpha=0$,
%}
\iffalse
We first consider the case $q = 1, \alpha = 0$,
namely,
$f_{\tau} = f$.
In view of
\eqref{eq:asuume_(u+v)f(u)[TV4AC-25]} 
and 
\eqref{eq:asuume_|f(u)|[TV4AC-25]},
Assumption \ref{assump:F(Nonlinearity)[TV4AC-25]} and 
\eqref{eq:asuume_(u+v)f(u)[TV4AC-25]},
%\eqref{eq:asuume_|f(u)|[TV4AC-25]}, 
% and the condition
% \eqref{eq:epsilon_condition_f_tau[TV4AC-25]} with $\alpha=0$,
one has
\begin{align}
      2( u + v ) f_{\tau}( u ) + \tau \epsilon^{-1} |f_{\tau}( u )|^2
& \leq
   -2 c_0 |u|^2 + 2 c_1 |v|^2 + 2c_2 + 2 c_f^2 \tau \epsilon^{-1} |u|^2 + 4 c_{f,0}^2 \tau \epsilon^{-1}.
\notag\\
% & \leq
%    - c_0 |u|^2 + 2 c_1 |v|^2 + 2c_2 + \tfrac{c_5^2 c_0}{ (c_3 + c_4)^2 } .
\end{align}
Moreover, from the condition
\eqref{eq:epsilon_condition_f_tau[TV4AC-25]} with $\alpha=0$, that is
\begin{equation}
2(c_3 + \mathbbm{1}_{\{q=1\}} c_4)^2
\tau
\epsilon^{-1}
\leq 
c_0,
\end{equation}
we infer
\begin{equation}
2( u + v ) f_{\tau}( u ) + \tau \epsilon^{-1} |f_{\tau}( u )|^2
\leq
- c_0 |u|^2 
+
2 c_1 |v|^2
+
2c_2
+
\tfrac
{ c_5^2 c_0 }
{ (c_3 + c_4)^2 } 
\end{equation}
as asserted.
%{\color{red}
%where we also employed the condition \eqref{eq:epsilon_condition_f_tau[TV4AC-25]} for $\alpha=0$, i.e.,
%$2 (c_3 + c_4)^2 \tau \epsilon^{-1} \leq c_0 $.
%}

Secondly, we %{\color{red}investigate for}
address
the case $q >1, \alpha \in (0,1]$.

\fi
%
%
Using
%Assumption \ref{assump:F(Nonlinearity)[TV4AC-25]} and 
\eqref{eq:asuume_(u+v)f(u)[TV4AC-25]}-\eqref{eq:asuume_|f(u)|[TV4AC-25]}
%once more 
gives
\begin{align}
  & 2( u + v ) f_{\tau}( u ) + \tau \epsilon^{-1} |f_{\tau}( u )|^2
\notag\\
   & =
   \frac{  2( u + v ) f( u ) }
   { (1 + \beta \tau^{\theta}  
       |u|^{ \frac{2q-2}{\alpha} } )^\alpha
      } 
   +
   \frac{ \tau \epsilon^{-1} |f( u )|^2
        }{ (1 + \beta \tau^{\theta}  
       |u|^{ \frac{2q-2}{\alpha} } )^{2\alpha}
      } 
\notag\\
   & \leq
   \frac{ 2 ( -c_0 |u|^{2q} + c_1 |v|^{2q} + c_2 ) }
   { (1 + \beta \tau^{\theta}   
       |u|^{ \frac{2q-2}{\alpha} } )^\alpha
      } 
   +
   \frac{ \tau \epsilon^{-1}
         ( c_3|u|^{2q-1} + c_4|u| + c_5 )^2
        }{ (1 + \beta \tau^{\theta}   
       |u|^{ \frac{2q-2}{\alpha} } )^{2\alpha}
      } 
\notag\\
   & \leq
   2 c_1 |v|^{2q}
 +
   2 c_2
 +
   \frac{ -2c_0 |u|^{2q} 
         (1 + \beta \tau^{\theta}  
       |u|^{ \frac{2q-2}{\alpha} } )^\alpha 
       +
         \tau \epsilon^{-1}
         ( c_3|u|^{2q-1} + c_4|u| + c_5 )^2
    }
         { (1 + \beta \tau^{\theta}   
       |u|^{ \frac{2q-2}{\alpha} } )^{2\alpha}
      }.
\end{align}
Note that 
$ (1+x)^\alpha \geq 2^{\alpha-1}(1+x^\alpha),
    x \geq 0$ for any $\alpha \in (0,1]$
(see, e.g., \cite[(3.16)]{wang2025uniform}),
while
$ (1 + x)^\alpha \geq 1 + x^\alpha, x \geq 0$ for $\alpha >1$.
Accordingly, 
%there exists some $\Tilde{\alpha} >0$ such that
for $\Tilde{\alpha} := \min\{ \alpha, 1 \}$ it holds
\begin{equation}
    (1 + x)^\alpha
    \geq
    2^{\Tilde{\alpha} - 1}
    (1 + x^\alpha),
    \quad
    \text{for any}
    \ 
    x \geq 0,
    \alpha >0.
\end{equation}
This together with
the Young inequality, namely,
$ ( c_3|u|^{2q-1} + c_4|u| + c_5 )^2 \leq 2^{\Tilde{\alpha}} c_3^2 |u|^{4q-2} + \hat{c}_5  $ for some constant $\hat{c}_5 > 0$ depending on $(2^{\Tilde{\alpha}}-1)^{-1}$, further implies

\begin{align}
  & 2( u + v ) f_{\tau}( u ) + \tau \epsilon^{-1} |f_{\tau}( u )|^2
\notag
\\
   & \leq
   2 c_1 |v|^{2q}
 +
   2 c_2
 +
   \frac{ -
          2^{\Tilde{\alpha}}
          c_0 |u|^{2q} 
         (1 + \beta^{\alpha} \tau^{\theta \alpha}
         |u|^{2q-2} )
        +
          2^{\Tilde{\alpha}}
          c_3^2 \tau \epsilon^{-1}
          |u|^{4q-2}
        +
           \hat{c}_5 \tau \epsilon^{-1} }
         { (1 + \beta \tau^{\theta}   
       |u|^{ \frac{2q-2}{\alpha} } )^{2\alpha}
      }
\notag\\
   & \leq
   2 c_1 |v|^{2q}
 +
   2 c_2 + \tfrac{\hat{c}_5 c_0 \beta^\alpha \tau^{\theta \alpha} }{ 2 c_3^2 }
 +
   \frac{ -  2^{\Tilde{\alpha}} c_0 |u|^{2q} 
          -  \tfrac12 \cdot 2^{\Tilde{\alpha}}
             c_0 
             \beta^{\alpha} \tau^{\theta \alpha } 
             |u|^{4q-2}
         }
   { ( 1 + \beta \tau^{\theta}  |u|^{ \frac{2q-2}{\alpha} } )^{2\alpha} }
\notag\\
& \quad
+
   \frac{ 
          -  \tfrac{1}{2} \cdot 2^{\Tilde{\alpha}}
             c_0 
             \beta^{\alpha} \tau^{\theta \alpha } 
             |u|^{4q-2}
          +
           2^{\Tilde{\alpha}}
          c_3^2 \tau \epsilon^{-1}
          |u|^{4q-2}
         }
   { ( 1 + \beta \tau^{\theta}  |u|^{ \frac{2q-2}{\alpha} } )^{ 2\alpha } }
\notag\\
   & \leq
   2 c_1 |v|^{2q}
 +
   2 c_2 + \tfrac{\hat{c}_5 c_0 \beta^\alpha \tau^{\theta \alpha}}{ 2 c_3^2 }
 +
  %2^{\Tilde{\alpha}}
   \bigg[
   \frac{ -  c_0 |u|^{2q-2} 
          -  \tfrac12 
             c_0 
             \beta^{\alpha} \tau^{\theta \alpha } 
             |u|^{4q-4}
         }
   { ( 1 + \beta \tau^{\theta}  |u|^{ \frac{2q-2}{\alpha} } )^{ 2\alpha} }
\Bigg]
\cdot
  |u|^2,
\label{eq:Ex2_(u+v)f_tau(u)+tau|f(u)|^2[TV4AC-25]}
\end{align}
where the last inequality
%{\color{red}was due to }
follows from
the condition \eqref{eq:epsilon_condition_f_tau[TV4AC-25]} for $q>1$, i.e., $2 c_3^2 \tau \epsilon^{-1} \leq c_0 \beta^\alpha \tau^{\theta \alpha}$.

To proceed further,
%with the verification \eqref{eq:(u+v)f_tau(u)[TV4AC-25]},
we consider two cases: $|u|\leq 1$ and $|u|>1$.
%
%{\color{gray}Without loss of generality,
%{\color{red} one takes 
%} 
%we assume
%$\tau^* \leq 1$.}
 %
For the former case $|u| \leq 1$, one easily derives from \eqref{eq:Ex2_(u+v)f_tau(u)+tau|f(u)|^2[TV4AC-25]} that 
%for any $\widetilde{c}_0 >0$,
\begin{align}\label{eq:|u|<1_verify[TV4AC-25]}
 2
 ( u + v )
 f_{\tau}( u )
 + 
 \tau
 \epsilon^{-1}
 |f_{\tau}( u )|^2
&\leq
   2
   c_1
   |v|^{2q}
 +
   2 
   c_2 
+ 
   \tfrac{
   \hat{c}_5 
   c_0 
   \beta^\alpha
   \tau^{\theta \alpha}
   }
   {
   2 
   c_3^2 
   }
\notag
\\
&\leq
    - 
    \tfrac
    { c_0 }
    { 2(1 + \beta \tau^\theta)^{2\alpha} }
  |u|^{2}
+
   2 
   c_1 
   |v|^{2q}
+
   2 
   c_2 
+  
    \tfrac
    { c_0 }
    { 2(1 + \beta \tau^\theta)^{2\alpha} }
+
   \tfrac{
   \hat{c}_5 
   c_0 
   \beta^\alpha
   \tau^{\theta \alpha}
   }
   {
   2 
   c_3^2 
   }
,
\end{align}
as required.
For the other case $|u|>1$, we introduce an auxiliary function defined by
\begin{align}
\Upsilon(x)
    :=
     \frac{ - 
             c_0 x 
            - 
             \tfrac12 
             c_0 
             \beta^{\alpha}
             \tau^{\theta \alpha } 
             x^2
         }
   { 
    ( 
      1 
      +
      \beta 
      \tau^{\theta} 
      x^{  \frac{1}{\alpha} } 
    )^{ 2\alpha }
   },
   \quad
   x \geq 0
   , \quad
   \alpha >0,
\end{align}
and therefore, the equation \eqref{eq:Ex2_(u+v)f_tau(u)+tau|f(u)|^2[TV4AC-25]} is recast as 
\begin{align}
\label{eq:|u|>1_verify[TV4AC-25]}
   2( u + v ) f_{\tau}( u ) + \tau \epsilon^{-1} |f_{\tau}( u )|^2
\leq
2 c_1 |v|^{2q}
+ 
 2 c_2 
 + 
 \tfrac{\hat{c}_5 c_0 \beta^\alpha \tau^{\theta \alpha}}{ 2 c_3^2 }
 +
%2^{\Tilde{\alpha}}
\Upsilon(|u|^{2q-2}) 
\cdot  |u|^2.
\end{align}
Now it suffices to estimate the upper bound of $\Upsilon$,
%
%
%In what follows, we attempt to prove
for which we claim 
\begin{equation}\label{eq:Upsilon_bound[TV4AC-25]}
    \sup_{x \geq 1}
    \Upsilon(x)
    \leq
     \tfrac{ - c_0 }{ 2 
  (1 + \beta \tau^\theta )^{2\alpha} }
  .
%    =
%    \tfrac{ -c_0 }{ 2 
%  (1 + \beta )^{2\alpha} }.
\end{equation}
%where, without loss of generality, one takes $\tau^* = 1$.
To validate the claim \eqref{eq:Upsilon_bound[TV4AC-25]},
one derives
\begin{equation}\label{eq:Upsilon'(x)[TV4AC-25]}
\Upsilon'(x)
=
     \frac{ -  c_0 
            -  c_0
             \beta^{\alpha} \tau^{\theta \alpha } 
             x
            +
             c_0 
             \beta \tau^{\theta} 
             x^\frac{1}{\alpha}
         }
   { ( 1 + \beta \tau^{\theta}  x^{\frac{1}{\alpha}} )^{ 1+{2\alpha} } }.
\end{equation}
For $\alpha \geq 1$, 
one sees
\begin{equation}
c_0
\beta 
\tau^\theta 
x^{\frac{1}{\alpha}}
\leq
c_0
\big(
1
+
(
\beta 
\tau^\theta 
x^{\frac{1}{\alpha}}
)^\alpha
\big)
=
c_0
+
c_0
\beta^\alpha \tau^{\theta \alpha} x,
\end{equation}
leading to $\Upsilon'(x) \leq 0$,
and thus
\begin{align}
\sup_{x\geq 1}
    \Upsilon(x)
\leq
    \Upsilon(1)
=
   \tfrac{ - c_0  
            - \frac12 
             c_0 
             \beta^{\alpha} \tau^{\theta \alpha } 
         }
   { ( 1 + \beta \tau^{\theta}  )^{ {2\alpha} } }
\leq
 \tfrac{ -c_0 }{  
  (1 + \beta \tau^\theta  )^{2\alpha} }
  .
%=
%\tfrac{ -c_0 }{  
%  (1 + \beta )^{2\alpha} }.
\end{align}
For $\alpha \in (0,1)$,
following a same way as in \cite[(3.23)-(3.25) of Proposition 3.3]{wang2025uniform}, 
there exists $x^* > 0$ such that
\begin{equation}
    \Upsilon'(x) \leq 0 \ \text{for}\ 
    x \in[0,x^*],
    \quad
    \Upsilon'(x) \geq 0 \ \text{for}\  
    x \geq x^*.
\end{equation}
% {\color{blue}
% it follows from a same way as in \cite[(3.23)-(3.25) of Proposition 3.3]{wang2025uniform}
% that there exists 
% $x^*:=
%     \tfrac{y^*}{ 
%     \beta^\alpha \tau^{\theta \alpha }
%     }$ such that
% \begin{equation}
%     \Upsilon'(x) \leq 0 \ \text{for}\ 
%     x \in[0,x^*],
%     \quad
%     \Upsilon'(x) \geq 0 \ \text{for}\  
%     x \geq x^*,
% \end{equation}
% }
%namely, $\Upsilon(x)$ is decreasing for $x \in[0,x^*]$, while increasing for $x \geq x^*$. Hence we 
%{\color{red}and infers that} 
 Hence,
$\sup_{x\geq 1}\Upsilon(x)$ is bounded by 
$\Upsilon(1)
    \vee
    \lim_{x \rightarrow \infty} 
\Upsilon(x)$, i.e.,
\begin{align}
    \Upsilon(x)
\leq
    \Upsilon(1)
    \vee
    \lim_{x \rightarrow \infty} 
\Upsilon(x)
=
   \tfrac{ - c_0  
            - \frac12 
             c_0 
             \beta^{ \alpha } \tau^{\theta \alpha } 
         }
   { ( 1 + \beta \tau^{\theta} )^{ {2\alpha} } }
\vee
  \tfrac{-c_0}{
   2 \beta^\alpha \tau^{\theta \alpha}
  } 
\leq
 \tfrac{ -c_0 }{  
  2 (1 + \beta \tau^\theta  )^{2\alpha} }
  ,
% \tfrac{ -c_0 }{ 2 
%  (1 + \beta )^{2\alpha} },
\end{align}
as claimed.
Inserting this bound into \eqref{eq:|u|>1_verify[TV4AC-25]} thus completes the proof.
\end{proof}

\iffalse
%%%%%%%%%%%%%%%%%%%%%%%%%%%%%%%%%%%%%%%%%%%%
%% verification of the third condition %%%%%
\begin{itemize}
    \item Verifications of \eqref{eq:|F_tau(u)|[TV4AC-25]}
    and
    \eqref{eq:|F_tau(u)-F(u)|[TV4AC-25]}.
    
\end{itemize} 

The verification of
    \eqref{eq:|F_tau(u)|[TV4AC-25]} is straightforward and thus is omitted.
Next, to validate \eqref{eq:|F_tau(u)-F(u)|[TV4AC-25]}, we introduce an auxiliary function
\begin{equation}
    \Theta(x) := ( 1 + \beta \tau^{\theta}  x )^{-\alpha}, 
    \quad
    x \geq 0.
\end{equation}
Using this notation, $f_{\tau}$ can be rewritten as
$f_{\tau}(u) = f(u)\Theta(u^{ \frac{2q-2}{\alpha} }),u \in \R  $. Noting that
\begin{align}
| \Theta'(x) |
    = 
    \left|
    \tfrac{\alpha \beta \tau^{\theta}  }{ ( 1 + \beta \tau^{\theta}  x )^{\alpha+1}  }
    \right|
\leq
  \alpha \beta \tau^{\theta} 
,
\quad
x \geq 0,
\end{align}
and $\Theta(0)=1$, we see for all $u \in \R$,
\begin{align}
        | f_{\tau}(u) - f(u) |
    =
      |f(u)|
      \cdot
      \big|\Theta( u^{\frac{2q-2}{\alpha}} ) - \Theta(0)\big|
    \leq
       \alpha \beta \tau^{\theta} u^{\frac{2q-2}{\alpha}} |f(u)|,
\end{align}
validating the
property \eqref{eq:|F_tau(u)-F(u)|[TV4AC-25]}.
% with $\widetilde{c}_4 = \alpha \beta$ and$\widetilde{l}= \frac{2q-2}{\alpha} $. 

%% verification of the third condition %%%%%
%%%%%%%%%%%%%%%%%%%%%%%%%%%%%%%%%%%%%%%%%%%%
\fi

\section{Uniform moment bounds of numerical approximations}
\label{sec:moment-boundedness[TV4AC-25]}

%{
%\color{cyan} When analyzing the convergence rates, a crucial element is to derive uniform moment bounds for the proposed scheme. To get started, one recasts the time-stepping scheme \eqref{eq:time_discretization[TV4AC-25]} as}

The subsequent error analysis relies on
uniform-in-time moment bounds for the proposed scheme.
To get started,
we recast the time-stepping recursion 
\eqref{eq:time_discretization[TV4AC-25]} as follows:
%By iteration, the time-stepping scheme \eqref{eq:time_discretization[TV4AC-25]} can be recast as
\begin{align}\label{eq:discret_sum[TV4AC-25]}
     X^{\tau}_{t_{m}}
    & =
      Y^{\tau}_{t_{m}}
     +
     \mathcal{O}_{t_{m}}
    ,
    \quad
    \text{ with }
    \quad
      Y^{\tau}_{t_{m}} : =
      E( t_{m} ) X^{\tau}_{0}
      +
      \tau
      \sum_{k=0}^{m-1}
      E(t_{m} - t_k)
      \epsilon^{-1}
      F_{\tau}
     \big( X^{\tau}_{t_k} \big),
\end{align}
for $m \in \N_0$ with $ X^{\tau}_{0} = X_0$. 
Clearly,
%{\color{red}
%it is deduced that
%}
it follows from
$
Y^{\tau}_{0}
=
X_0
$
that,
for $m \in \N_0$,
\begin{equation}\label{eq:Y^tau[TV4AC-25]}
    Y^{\tau}_{t_{m+1}}
    =
    E(\tau) Y^{\tau}_{t_{m}}
    +
    \tau E(\tau) \epsilon^{-1}
    F_\tau( Y^{\tau}_{t_{m}} + \mathcal{O}_{t_{m}} ).
\end{equation}

\begin{comment}
Indeed, for any \(m\in\mathbb{N}_0\),
we obtain the following recursion,
\begin{equation}
\begin{aligned}
Y^{\tau}_{t_{m+1}}
&=E(t_{m+1})X^{\tau}_0
+\tau\sum_{k=0}^{m}E(t_{m+1}-t_k)\epsilon^{-1}F_{\tau}(X^{\tau}_{t_k})\\[2mm]
&=E(\tau)E(t_m)X^{\tau}_0
+\tau\sum_{k=0}^{m-1}E(\tau)E(t_m-t_k)\epsilon^{-1}F_{\tau}(X^{\tau}_{t_k})
+\tau E(\tau)\epsilon^{-1}F_{\tau}(X^{\tau}_{t_m})\\
&=E(\tau)
\Big(
E(t_m)X^{\tau}_0
+\tau\sum_{k=0}^{m-1}E(t_m-t_k)\epsilon^{-1}F_{\tau}(X^{\tau}_{t_k})
\Big)
+\tau E(\tau)\epsilon^{-1}F_{\tau}(X^{\tau}_{t_m})\\
&=E(\tau)Y^{\tau}_{t_m}
+\tau E(\tau)\epsilon^{-1}F_{\tau}(Y^{\tau}_{t_m}+\mathcal{O}_{t_m}),
\end{aligned}
\end{equation}
where the second equality follows from the semigroup identity
$E(t_{m+1}-t_k)=E(\tau)E(t_m-t_k)$.
\end{comment}

Before proceeding further, we introduce a modified version of \cite[Lemma 4.3]{wang2025uniform},
tailored for our analysis to 
%adjusting 
refine the estimate (4.9)
%below 
therein.

%sharpens
%modifiesthe estimate (4.9) therein.

%we recall a basic the inequality (a slight modification of \cite[Lemma 4.3]{wang2025uniform} by refining the estimate (4.9) therein)

\begin{lem}\label{lem:A+rB_inequal[TV4AC-25]}
    Let $\rho \geq 1$ be any integer.
    Then for any $\mathbb{A} ,\mathbb{B} \geq 0$ and  $r, \upsilon > 0$, it holds that
\begin{equation}
    (
    \mathbb{A} 
    +
    r
    \mathbb{B}
    )^{\rho}
 \leq
    e^{ 
    (\rho-1) 
    \upsilon 
    r
    } 
    \mathbb{A}^{\rho}  
    + 
     r
\left( 
   r ^{\rho-1} 
  + 
  (1 +  ( \tfrac{2}{ \upsilon } )^{\rho-1} )
  ( 1 + r^{\rho-1} )
  e^{ \rho-1 }
 \right)  
   \mathbb{B}^{\rho}.
\end{equation}
\end{lem}
\begin{proof}
In a similar manner as (4.6)-(4.8) in \cite[Lemma 4.3]{wang2025uniform},
%with $2q-1$ therein replaced by $\rho$,
%where we further replace $2q-1$ by $\rho$ therein, 
it holds
\begin{equation}
    (
    \mathbb{A} 
    +
    r
    \mathbb{B}
    )^{\rho}
 \leq
 \mathbb{A}^{\rho}
 +
 r^\rho
 \mathbb{B}^\rho
 +
 \tfrac
 {(\rho-1) \upsilon r}
 {2}
 e^{
 \frac
 { (\rho-1) \upsilon r }
 {2}
 }
 \mathbb{A}^\rho
 +
   \Big(
   \sum_{j=1}^{\rho-1}
       \tfrac{(\rho-1)!}{(j-1)!(\rho-j)!}
       r^j 
    ( \tfrac{2}{ \upsilon } )^{\rho-j}
   \Big)
   \mathbb{B}^\rho,
\end{equation}
where for the term 
$  
\sum_{j=1}^{\rho-1}
       \tfrac{(\rho-1)!}{(j-1)!(\rho-j)!}
       r^j 
    ( \tfrac{2}{ \upsilon } )^{\rho-j}
$
we provide a slight modification of \cite[(4.9)]{wang2025uniform} such that
\begin{equation}
\sum_{j=1}^{\rho-1}
       \tfrac{(\rho-1)!}{(j-1)!(\rho-j)!}
       r^j 
    ( \tfrac{2}{ \upsilon } )^{\rho-j}
\leq
    \sum_{j=1}^{\rho-1}
       \tfrac{(\rho-1)^{j-1} }{(j-1)!}
       r^j 
    ( \tfrac{2}{ \upsilon } )^{\rho-j}
\leq
r
%\big( 
(
1
+
(
\tfrac
{2}
{\upsilon}
)^{\rho-1}
)
(
1
+
r^{\rho-1}
)
e^{\rho-1}
.
%\big).
\end{equation}
The desired result then follows immediately by further employing that
\begin{equation}
\Big(
1
+
\tfrac
{(\rho-1)\upsilon r }{2}
e^{
\frac
{(\rho-1)\upsilon r }{2}
}
\Big)
\mathbb{A}^\rho
\leq
\big(
1
+
\tfrac
{(\rho-1)\upsilon r }{2}
\big)
e^{
\frac
{(\rho-1)\upsilon r }{2}
}
\mathbb{A}^\rho
\leq
e^{
(\rho-1)\upsilon r
}
\mathbb{A}^\rho.
\end{equation}
\end{proof}

The following theorem illustrates the uniform (in both time and $ \epsilon$) moment bounds of the numerical approximation $X^\tau$.

\begin{thm}[Uniform moment bounds of $X^\tau$]
\label{thm:X^tau-bound[TV4AC-25]}
    Let Assumptions \ref{assump:A(linear_operator)[TV4AC-25]}-\ref{assump:X_0(Initial Value)[TV4AC-25]} and condition \eqref{eq:epsilon_condition_f_tau[TV4AC-25]} hold. 
     For any $p \geq 1$ 
     and any integer $\rho \geq 1$, there exists 
      a constant $  C( p,\rho, q, \alpha)>0$
     %and $ C( X_0,p, q,  \alpha, \theta) >0 $ 
     such that the numerical approximation $ X^{\tau}_{t_m}, m \in \N_0$ produced by \eqref{eq:time_discretization[TV4AC-25]} obeys
    \begin{align}
        \sup_{m \in \N_0} 
        \big\|
          X^{\tau}_{t_m} 
        \big\|_{L^p( \Omega; L^{2\rho}(\mathcal{D}) )}
   &  \leq
    \big\| X_0 \big\|_{L^{p}(\Omega; L^{2\rho}(\mathcal{D}) )}
  +
     C( p, \rho, q, \alpha)
     \cdot
     \big(
     1 
     +
     (
     \beta
     \tau^{\theta} 
     )^{ \alpha(2 - \frac{1}{2\rho}) }
     \big).
 %    \quad
 %    {\color{red}
 %    e^{ C(1 + \beta^\alpha \tau^{\theta \alpha}) }    }
\end{align}

%%%%%%%% uniform bound in V-norm %%%%%%%%%%%
%%%%%%%%%%%%%%%%%%%%%%%%%%%%%%%%%%%%%%%%%%%
\iffalse
\begin{align}
\\
        \sup_{m \in \N_0} 
        \big\|
         X^{\tau}_{t_m} 
        \big\|_{L^p( \Omega; V )}
        +
         \sup_{m \in \N_0} 
        \big\| 
          X^{\tau}_{t_m} 
        \big\|_{L^p( \Omega; \dot{H}^{\gamma} )}
   &  \leq
       C( X_0,p, q,\alpha, \theta) 
       \epsilon^{- 1}.
    \end{align}
%%%%%%%% uniform bound in V-norm %%%%%%%%%%%
%%%%%%%%%%%%%%%%%%%%%%%%%%%%%%%%%%%%%%%%%%%
\fi

\end{thm}

\begin{proof}

Thanks to \eqref{eq:discret_sum[TV4AC-25]} and Lemma \ref{lem:O_bound[TV4AC-25]}, it suffices to 
derive
a 
priori estimates for $Y^{\tau}_{t_m},m \in \N_0$ in $L^{2\rho}$-norm.
%for which we divide into two steps.
%As the \textbf{Step 1}, we aim to prove the 
%uniform moment bounds of $Y^{\tau}$ in $L^{2\rho}$-norm.
%
%Before it, we 
Recall that for any $ t \geq 0 $ and integer $\rho \geq 1$, the semi-group operator $ E(t):=e^{-At},t \geq 0$ satisfies 
 a contractive property:
% {\color{red} that}
 $ \| E(t) u \|_{L^{2\rho}} \leq \| u \|_{ L^{2\rho} }$,
$u \in L^{2\rho}(\mathcal{D})$
(see, e.g., \cite[Proposition 4.2]{wang2025uniform}).
As a result, we deduce from  \eqref{eq:Y^tau[TV4AC-25]} that
\begin{equation}\label{eq::Y^tau(L^4q-2_bound)first_step_inprf[TV4AC-25]}
\begin{aligned}
    \big\|
    Y^{\tau}_{t_{m+1}} 
    \big\|_{ L^{2\rho} }
&=    
    \big\|
    E(\tau)
    \big(  Y^{\tau}_{t_m} 
       +
       \tau
       \epsilon^{-1}
       F_{\tau}(  Y^{\tau}_{t_m} +  \mathcal{O}_{t_m} ) 
    \big)
    \big\|_{ L^{2\rho} }   
&\leq    
    \big\|
     Y^{\tau}_{t_m} 
       +
       \tau
       \epsilon^{-1}
       F_{\tau}(  Y^{\tau}_{t_m} +  \mathcal{O}_{t_m})
    \big\|_{ L^{2\rho} } . 
\end{aligned}
\end{equation}
%Next, we proceed with a further derivation
%To proceed, we now turn to the estimates  in $\R$.
Bearing  \eqref{eq:(u+v)f_tau(u)[TV4AC-25]} in mind and denoting
$
\widetilde{c}_0 
:=
\frac
{ c_0 }
{ 4(1 + \beta \tau^\theta)^{2\alpha} }
$
for short, one acquires
\begin{align}
\label{equation:expansion-Y}
&\quad
     \big| 
     Y^{\tau}_{t_m}(\cdot) 
     + 
     \tau
     \epsilon^{-1}
     f_{\tau}(  Y^{\tau}_{t_m} +  \mathcal{O}_{t_m} )(\cdot) 
     \big|^2
  \notag \\
&
 =
     \big| Y^{\tau}_{t_m}(\cdot) \big|^2
     +
     2 \tau
     \epsilon^{-1}
     Y^{\tau}_{t_m}(\cdot) 
     f_{\tau}(  Y^{\tau}_{t_m}(\cdot) +  \mathcal{O}_{t_m}(\cdot) )
     +
     \tau^2
     \epsilon^{-2}
     \big| 
     f_{\tau}
     \big(  Y^{\tau}_{t_m}(\cdot) +  \mathcal{O}_{t_m}(\cdot) \big)
     \big|^2
  \notag \\
&
\leq
     \big| Y^{\tau}_{t_m}(\cdot) \big|^2
     +
     \tau
     \epsilon^{-1}
     \Big(
     -
      2 \widetilde{c}_0
%     \tfrac
%     { }
%     { (1 + \beta \tau^\theta)^{2\alpha} }
     \big|
     Y^{\tau}_{t_m}(\cdot) +  \mathcal{O}_{t_m}(\cdot) 
     \big|^2
     +
     \widetilde{c}_1 
     \big( 
        1
        + 
       \big|
       \mathcal{O}_{t_m}(\cdot)
       \big|^{2q} 
     \big)
     +
     \widetilde{c}_2
     \beta^\alpha
     \tau^{\theta \alpha}
     \Big)
     \notag \\
&
\leq
 \big( 
     1 
     - 
     \widetilde{c}_0 
     \tau
     \epsilon^{-1}
%     \tfrac
%     {
%     }
%     {
%     (1 + \beta \tau^\theta)^{2\alpha}
 %    }
\big)
     \big| Y^{\tau}_{t_m}(\cdot) \big|^2
+   
     2 \widetilde{c}_0 
     \tau
     \epsilon^{-1}
%     \tfrac
%     {
%     }
%     {
%     (1 + \beta \tau^\theta)^{2\alpha}
%     }
     \big| \mathcal{O}_{t_m}(\cdot) \big|^2
+ 
\tau \epsilon^{-1}
\Big(
     \widetilde{c}_1 
      \big( 
        1 
        +
        \big|
        \mathcal{O}_{t_m}(\cdot)
        \big|^{2q} 
      \big)
+
%    \tfrac
%    { c_0 \widetilde{c}_2 }
%    { 2 c_3^2 }
 \widetilde{c}_2
    \beta^{\alpha}
    \tau^{\theta \alpha}
\Big)
  \notag \\
&
\leq
     e^{
     -
     \widetilde{c}_0
     \tau
     \epsilon^{-1}
%     \frac
%     {
%     }
%     {
%     (1 + \beta \tau^\theta)^{2\alpha}
%     }  
        }
     \big| Y^{\tau}_{t_m}(\cdot) \big|^2
     +
     \tau
     \epsilon^{-1}
\Big(
     (2\widetilde{c}_0 + \widetilde{c}_1 ) 
     \big(  1 + 
       \big|
          \mathcal{O}_{t_m}(\cdot)
        \big|^{2q }
     \big)
+
%    \tfrac
%    { c_0 \widetilde{c}_2 }
%    { 2 c_3^2 }
 \widetilde{c}_2
    \beta^{\alpha}
    \tau^{\theta \alpha}
\Big),
\end{align}
%{\color{red} where we used  the Young inequality that $ -4 \widetilde{c}_0 \tau \epsilon^{-1} Y^{\tau}_{t_m}(\cdot) \mathcal{O}_{t_m}(\cdot) \leq \widetilde{c}_0 \tau \epsilon^{-1}|Y^{\tau}_{t_m}(\cdot)|^2 + 4 \widetilde{c}_0 \tau \epsilon^{-1}|\mathcal{O}_{t_m}(\cdot)|^2$, as well as  the fact that $1-y < e^{-y}$ for any $y > 0$.  }
where 
we 
further used
the Young inequality 
\begin{equation} 
%\widetilde{c}_0
%\tau
%\epsilon^{-1}
-
4 \widetilde{c}_0
%\tfrac
%     {  }
%     { (1 + \beta \tau^\theta)^{2\alpha} }
Y^{\tau}_{t_m}(\cdot)
\mathcal{O}_{t_m}(\cdot)
\leq
%\widetilde{c}_0
%\tau 
%\epsilon^{-1}
%\tfrac
%     { \widetilde{c}_0 }
%     { (1 + \beta \tau^\theta)^{2\alpha} }
 \widetilde{c}_0
|Y^{\tau}_{t_m}(\cdot)|^2 
+ 
4
\widetilde{c}_0 
%\tau
%\epsilon^{-1}
|\mathcal{O}_{t_m}(\cdot)|^2,
\end{equation}
as well as the inequality $1-y\le e^{-y}$ for $y>0$
and the fact that $|\mathcal{O}_{t_m} (\cdot )|^2 \le 1 + |\mathcal{O}_{t_m} (\cdot )|^{2q}$ for $q > 1$.
%as well as the condition \eqref{eq:epsilon_condition_f_tau[TV4AC-25]}, i.e.,
%$2c_3^2 \tau \epsilon^{-1} \leq c_0 \beta^{\alpha} \tau^{ \theta \alpha }$.
%
\iffalse
Before proceeding further, 
we recall 
%a basic 
the inequality 
(a slight modification of \cite[Lemma 4.3]{wang2025uniform} by refining the estimate (4.9) therein)
%see, e.g., \cite[Lemma 4.3]{wang2025uniform}
that 
\begin{equation}\label{eq:A+B_inequal[TV4AC-25]}
    (
    \mathbb{A} 
    +
    r
    \mathbb{B}
    )^{\rho}
 \leq
    e^{ 
    (\rho-1) 
    \upsilon 
    r
    } 
    \mathbb{A}^{\rho}  
    + 
     r
\left( 
   r ^{\rho-1} 
  + 
  (1 +  ( \tfrac{2}{ \upsilon } )^{\rho} )
  ( 1 + r^{\rho-1} )
  e^{ \rho-1 }
 \right)  
   \mathbb{B}^{\rho},
\end{equation}
for any $\mathbb{A} ,\mathbb{B} \geq 0$ and $r, \upsilon > 0$.
%
\fi
%
%
Therefore,
%{\color{red}
%by taking $\rho$-th power on both sides and setting $\upsilon= \widetilde{c}_0 $ in \eqref{eq:A+B_inequal[TV4AC-25]} 
%}
%
by taking $\rho$-th power on both sides of \eqref{equation:expansion-Y} and applying Lemma \ref{lem:A+rB_inequal[TV4AC-25]} with
$r = \tau \epsilon^{-1}$, 
$
\upsilon
= 
\widetilde{c}_0$
and
\begin{equation}
\mathbb{A}
 =
     e^{
     -
     \widetilde{c}_0
     \tau
     \epsilon^{-1} 
        }
     | Y^{\tau}_{t_m}(\cdot) |^2,
\quad
\mathbb{B}
 =
        (
            2\widetilde{c}_0
            +
            \widetilde{c}_1 
          )
       (  1 + 
             |
             \mathcal{O}_{t_m}(\cdot)
             |^{2q }
       )
       +
       %    \tfrac
%    { c_0 \widetilde{c}_2 }
%    { 2 c_3^2 }
 \widetilde{c}_2
       \beta^\alpha
       \tau^{ \theta \alpha },
\end{equation}
one further gets
\begin{equation}
\begin{aligned}
&\big|
    Y^{\tau}_{t_m}(\cdot) + \tau \epsilon^{-1} F_{\tau}(  Y^{\tau}_{t_m} +  \mathcal{O}_{t_m} )(\cdot) 
    \big|^{2\rho}
\\
&
\leq
    e^{   
    (\rho-1)
    \widetilde{c}_0 
    \tau
    \epsilon^{-1} 
    }
    \Big( 
    e^{
    -     
    \widetilde{c}_0 
    \tau
    \epsilon^{-1} 
    }
     \big|Y^{\tau}_{t_m}(\cdot)\big|^2
     \Big)^{\rho} 
\\
&\quad
    + 
    \tfrac{\tau}{\epsilon}
\left( 
  ( \tfrac{\tau}{\epsilon} )^{\rho-1} 
  + 
( 1
    + 
    ( 
    \tfrac
    { 
    2 
    }
    { \widetilde{c}_0 } 
    )^{\rho-1} 
)
( 
1
+
( \tfrac{\tau}{\epsilon} )^{\rho-1}
)
  e^{ \rho-1}
 \right)  
   \Big(
         (
            2 \widetilde{c}_0
            +
            \widetilde{c}_1 
          )
       \big(  1 + 
             \big|
             \mathcal{O}_{t_m}(\cdot)
             \big|^{2q }
       \big)
       +
%    \tfrac
%    { c_0 \widetilde{c}_2 }
%    { 2 c_3^2 }
 \widetilde{c}_2
       \beta^\alpha
       \tau^{ \theta \alpha }
    \Big)^{\rho}
\\
&\leq 
    e^{
    - 
    \widetilde{c}_0 
    \tau
    \epsilon^{-1} 
    }
     \big|Y^{\tau}_{t_m}(\cdot)\big|^{2\rho} 
     +
     C 
     \tau \epsilon^{-1} 
     \big(
     1
     +
     (
     \beta
     \tau^{\theta }
     )^{\alpha( 4\rho-3)}
     \big)
     \big(
     1 +  \big|
          \mathcal{O}_{t_m}(\cdot)
        \big|^{2q\rho }
     \big)
     ,
\end{aligned}
\end{equation}
where we used the condition \eqref{eq:epsilon_condition_f_tau[TV4AC-25]} that $\tau \epsilon^{-1} \leq \frac{c_0}{2 c_3^2} \beta^\alpha \tau^{\theta \alpha}$
and recalled
$
\widetilde{c}_0 
=
\frac
{ c_0 }
{ 4(1 + \beta \tau^\theta)^{2\alpha} }
$.
%{\color{red}
%$\tau \epsilon^{-1} \leq \frac{c_0}{2 c_3^2} \beta^\alpha \tau^{\theta \alpha}$
%}
%as well as the Young inequality.
%
By integrating the above inequality over $\mathcal{D}$ and recalling 
%{\color{red}
%the equation \eqref{eq::Y^tau(L^4q-2_bound)first_step_inprf[TV4AC-25]}
%}
\eqref{eq::Y^tau(L^4q-2_bound)first_step_inprf[TV4AC-25]}, we thus obtain
\begin{equation}
\begin{aligned}
    \big\| Y^{\tau}_{t_{m+1}} \big\|_{L^{2\rho}}^{2\rho}
&
\leq
    e^{
    - 
    \widetilde{c}_0 
    \tau
    \epsilon^{-1} 
    }
    \big\| Y^{\tau}_{t_m} \big\|_{L^{2\rho}}^{2\rho}
  +
     C 
     \tau \epsilon^{-1} 
     \big(
     1
     +
      (
     \beta
     \tau^{\theta }
     )^{\alpha( 4\rho-3)}
     \big)
    \Big(
      1 + \big\| \mathcal{O}_{t_m} \big\|_V^{ 2q \rho } 
    \Big)  
\\
&
\leq
     e^{
    - 
    (m+1)
    \widetilde{c}_0 
    \tau
    \epsilon^{-1} 
    }
    \big\| Y^{\tau}_{0} \big\|_{L^{2\rho}}^{2\rho}
  +
     C
    \tau 
    \epsilon^{-1} 
     \big(
     1
     +
     (
     \beta
     \tau^{\theta }
     )^{\alpha( 4\rho-3)}
     \big)
    \sum_{k=0}^{m}
    e^{
    - 
    (m-k)
    \widetilde{c}_0 
    \tau
    \epsilon^{-1} 
    }
    \Big(
      1 + \big\| \mathcal{O}_{t_k} \big\|_V^{ 2q \rho } 
    \Big).
\end{aligned}
\end{equation}
Furthermore, one notes that 
%for all $m \in \N$,
%as well as the fact that
% $\tau \epsilon^{-1} \sum_{k=0}^{m}
%    e^{ -(m-k)\widetilde{c}_0 \epsilon^{-1}\tau}
%    < \infty$ for all $m \in \N$,
\begin{equation}
\tau \epsilon^{-1} 
    \sum_{k=0}^{m}
    e^{
    -  
    (m-k)
    \widetilde{c}_0 
    \tau
    \epsilon^{-1} 
    }
=
\frac
{
\tau \epsilon^{-1} 
(
1
- 
e^{-\widetilde{c}_0 \tau \epsilon^{-1} (m+1) } 
)
}
{ 1 - e^{-\widetilde{c}_0 \tau \epsilon^{-1} } }
\leq
\frac
{ \tau \epsilon^{-1}  }
{
1 
- 
e^{
- 
\widetilde{c}_0 
\tau
\epsilon^{-1} 
}
}
\leq
\tfrac
{ 1}
{\widetilde{c}_0 }
+
\tau
\epsilon^{-1}
\leq
C
\big(
1
+
(\beta \tau^\theta)^{2\alpha}
\big)
,
\end{equation}
where we employed the inequality that 
$\frac{y}{1 - e^{-y}} \leq 1+ y$ 
for all $y>0$, and recalled again the condition \eqref{eq:epsilon_condition_f_tau[TV4AC-25]} together with 
$
\widetilde{c}_0 
=
\frac
{ c_0 }
{ 4(1 + \beta \tau^\theta)^{2\alpha} }
$.
Using the fact
$Y^{\tau}_{0} = X_0$ and Lemma \ref{lem:O_bound[TV4AC-25]}
then ensures for any $p \geq 1$ and $m \in \N$,
\begin{equation}\label{eq:Y^tau-(4q-2)bound[TV4AC-25]}
    \big\|
      Y^{\tau}_{t_{m}} 
    \big\|_{L^{p}(\Omega; L^{2\rho}(\mathcal{D}) )}
\leq
    \big\| X_0 \big\|_{L^{p}(\Omega; L^{2\rho}(\mathcal{D}) )}
  +
     C( p, \rho, q, \alpha)
     \cdot
     \big(
     1
     +
     (
     \beta
     \tau^{\theta }
     )^{ \alpha (2 - \frac{1}{2\rho}) }
     \big).  
\end{equation}
%
%
%The \textbf{Step 2} is devoted to show the final estimates for the uniform moment bounds of $X^{\tau}$.
%
%
This together with Lemma \ref{lem:O_bound[TV4AC-25]} implies
\begin{equation}
            \sup_{m \in \N_0} 
        \big\|
          X^{\tau}_{t_m} 
        \big\|_{L^p( \Omega; L^{2\rho}(\mathcal{D}) )}
     \leq
    \big\| X_0 \big\|_{L^{p}(\Omega; L^{2\rho}(\mathcal{D}) )}
  +
     C( p, \rho, q, \alpha)
     \cdot
     \big(
     1
     +
     (
     \beta
     \tau^{\theta }
     )^{ \alpha (2 - \frac{1}{2\rho}) }
     \big),
\end{equation}
as required.
\end{proof}

\section{Regularity estimates of the Kolmogorov equation}
\label{sec:Kolmogorov_equation[TV4AC-25]}
%
%
%
%%%%%%%%%%%%%%%%%%%%%%%%%%%%%%
\iffalse
{\color{cyan}
This section is devoted to the regularity estimates of the Kolmogorov equation, which serves as a 
%common and 
powerful tool in weak convergence analysis,
particularly in situations where the test function admits low regularity,
%and their regularity estimates,
as possessed by,
for instance,
the total variation distance.
%
%In such settings, the Kolmogorov equations provide essential tools and therefore, 
%recalling the considered infinite-dimensional problem \eqref{eq:considered_SEE[TV4AC-25]},
%an auxiliary process with Lipschitz continuous nonlinearity would be necessary.
However, due to the infinite-dimensional setting of the underlying problem \eqref{eq:considered_SEE[TV4AC-25]} and the presence of a nonlinearity with possibly superlinear growth,
a direct application is not feasible. 
%
%To overcome this difficulty, 
To this end,
we introduce an auxiliary process with Lipschitz continuous nonlinearity.
This strategy allows us to establish the desired regularity properties of the associated Kolmogorov equation and thereby facilitates the subsequent convergence analysis.
%
%
%
%%%%%%%%%%%%%%%%%%%%%%%%%%%%%%%
\fi

The subsequent weak convergence analysis relies on the Kolmogorov equation associated with the stochastic dynamics.  
For the considered SPDEs \eqref{eq:considered_SEE[TV4AC-25]}, the drift grows superlinearly and thus violates Lipschitz continuity,
which prevents a direct use of the Kolmogorov equation.
To resolve this, we construct an auxiliary process whose drift is modified to be globally Lipschitz continuous, 
for which the corresponding Kolmogorov equation is well posed.  
This enables us to establish regularity estimates essential for the subsequent weak convergence analysis.

\subsection{An auxiliary process and its uniform moment bound
}

%To this end, 
Inspired by the proposed modification of $F$ 
as described in \eqref{eq:f_tau[TV4AC-25]}, we introduce the following auxiliary process:
%
%
%To do this,  we introduce an auxiliary process as follows,
%To carry out the weak error analysis by means of the Kolmogorov equation,
%it is introduce an auxiliary process
%}
%{\color{red} Inspired by the example \eqref{eq:f_tau[TV4AC-25]},}
\begin{equation}\label{eq:spde_auxiliary[TV4AC-25]}
        \mathbb{X}^\delta(t)
        = 
        E(t) \mathbb{X}^\delta_0
               +
               \int_0^t E(t-s) 
               \epsilon^{-1}
               F_\delta( \mathbb{X}^{\delta}(s) ) \,\mathrm{d}s 
               +
               \int_0^t E(t-s) \,\mathrm{d}W(s),
               \quad X_0^{\delta} \in H, \
               t \geq 0,
    \end{equation}
where $F_\delta \colon L^{4q-2}(\mathcal{D}) \rightarrow H$ is a nonlinear Nemytskii operator such that
$F_\delta(u)(x) := f_\delta( u(x) )$
with $f_\delta \colon \R \rightarrow \R$ defined by
\begin{equation}\label{eq:f_delta[TV4AC-25]}
       f_{\delta}( u )
       :=
       \tfrac{ f( u ) }
       {   1 
            + 
            \sqrt{\delta}
          |u|^{ 2q-2 }
        },
        \quad
        u \in \R,
    \end{equation}
where we let $\delta\in (0,1]$ be
%is chosen 
sufficiently small.
Moreover, for any 
$\upsilon, \psi \in L^{4q-2}(\mathcal{D})$,
we define
\begin{equation}
\label{eq:def_DF_delta[TV4AC-25]}
    \big(
    F_{\delta}'(\upsilon)(\psi)
    \big)
    (x)
    :=
    f_{\delta}'( \upsilon(x)  )
    \psi(x),
    \quad
    x \in \mathcal{D}.
\end{equation}

%We note that by taking the parameters $\alpha, \beta = 1$ in the previously designed modification \eqref{eq:f_tau[TV4AC-25]}, we obtain \eqref{eq:f_delta[TV4AC-25]}, which serves as a typical example thereof.  For this  case, an analogous condition to \eqref{eq:epsilon_condition_f_tau[TV4AC-25]}, namely, $2 c_3^2 \sqrt{\delta} \leq c_0 \epsilon$, can always be satisfied here by letting $\delta \to 0$. Consequently, the regularized nonlinearity $f_\delta$ inherits similar properties to those established in Proposition \ref{prop:F_tau[TV4AC-25]} with fixed $\alpha, \beta =1, \theta=\tfrac12$ and $\tau$ therein replaced by $\delta$.

Note that the modification \eqref{eq:f_delta[TV4AC-25]} is a special case of \eqref{eq:epsilon_condition_f_tau[TV4AC-25]} with $\alpha= \beta = 1$ and $\theta=\tfrac12$.
In this particular setting, the counterpart of condition \eqref{eq:epsilon_condition_f_tau[TV4AC-25]} reduces to
$2 c_3^2 \sqrt{\delta} \leq c_0 \epsilon$, 
which can be always satisfied as $\delta \to 0$ finally.
Consequently,
such a regularized nonlinearity $f_\delta$  \eqref{eq:f_delta[TV4AC-25]} inherits 
similar
properties to those established in Proposition \ref{prop:F_tau[TV4AC-25]} with fixed $\alpha= \beta =1, \theta=\tfrac12$ and $\tau$ therein replaced by $\delta$.

%
%
%%%%%%%%%%%%%%%%%%%%%%%%%%%%%%%%%%%%%%%%%%%%%%%
%%%%%%%%%  Properties of f_\delta %%%%%%%%%%%%%
\iffalse
one knows
    \begin{align}
    \label{eq:(u+v)f_delta(u)[TV4AC-25]}
    2 
    ( u + v ) f_{\delta}(u) 
    +
    \delta \epsilon^{-1}
    | f_{\delta}(u) |^2
  &  \leq 
    - 2 \widetilde{c}_{0} |u|^2 
    +
    2 \widetilde{c}_{1} 
    ( 1 + |v|^{2q } ),  
    \\ 
    \label{eq:|F_delta(u)|[TV4AC-25]}
    | f_{\delta}(u) |
   & \leq
    \widetilde{c}_{2} 
    |f(u)|,
    \\
    \label{eq:|F_delta(u)-F(u)|[TV4AC-25]}
    | f_{\delta}(u) - f(u) |
  &  \leq
    \widetilde{c}_{3}
    \delta^{\theta}
     ( 1 + |u|^{ \widetilde{l} } ) |f(u)|,
    \end{align}
for some constants $\widetilde{c}_{0}, \widetilde{c}_{1}, \widetilde{c}_{2}, \widetilde{c}_{3} >0$
independent of $\delta$.

\fi
%%%%%%%%%  Properties of f_\delta  %%%%%%%%%%%%%
%%%%%%%%%%%%%%%%%%%%%%%%%%%%%%%%%%%%%%%%%%%%%%%%
%
%
%{\color{gray}
%Define
%$
%    \big(
%    F'_\delta(\upsilon)(\psi)
%    \big)
%    (x)
%    :=
%    f'_\delta( \upsilon(x)  )
%    \psi(x),
%    x \in \mathcal{D},
%$
%{
%\color{red}
%Fréchet derivative?
%}
%for any 
%$\upsilon, \psi\in L^{4q-2}(\mathcal{D})$.
%}
The following lemma 
shows additional properties for the derivative $f'_\delta$.

\begin{lem}\label{lem:f'_delta[TV4AC-25]}
Let Assumption \ref{assump:F(Nonlinearity)[TV4AC-25]} hold with $q >1$ and let $f_\delta$ be defined by \eqref{eq:f_delta[TV4AC-25]}.
%and let the condition \eqref{eq:epsilon_condition_f_delta[TV4AC-25]} hold.
Then there exist
constants $\widetilde{L}_f$ and $\widetilde{l}_f$ 
independent of $ \delta$ such that
\begin{align}
\label{eq:bound_f'_delta[TV4AC-25]}
%   \sup_{\delta \in (0,1]}
    \sup_{u \in \R}
    f_\delta' (u) 
&   
\leq
    \widetilde{L}_f
    ,  
\\
   \label{eq:|Df_delta|_V[TV4AC-25]}
% {\color{red}
%    \sup_{\delta \in (0,1]}
%    }
   | f_\delta' (u) |
&   
\leq
\widetilde{l}_f 
\big(
1
+
(
|u|^{2q-2}
\wedge
\delta^{ - \frac12 }
)
\big).
\end{align}

\end{lem}

\iffalse
{\color{red}
\textbf{
Has condition (5.3) been used?
}
}
\fi

\begin{proof}

%For $q=1$, one has $f_\delta(u) = f(u)$ and thus $f_\delta'(u) = f'(u)$ for all $u \in \R$, so that \eqref{eq:bound_f'_delta[TV4AC-25]} holds with $\widetilde{L}_f = L_f$ by \eqref{eq:asuume_f'(u)[TV4AC-25]}.
%
%Now consider the case $q > 1$.
Denote $\mathrm{sgn}(u) := \mathbbm{1}_{\{ u \geq 0 \}} - \mathbbm{1}_{\{ u < 0 \}}$
for $u \in \R$.
It is straightforward to deduce
\begin{equation}\label{eq:derivative_f'_del[TV4AC-25]}
\begin{aligned}
f_\delta'(u)
&= 
\frac{
f'(u)
\big( 
1 + \sqrt{\delta} |u|^{2q-2} 
\big)
-
f(u)
(2q-2)
\sqrt{\delta} 
|u|^{2q-3} 
\mathrm{sgn}(u)
}{
( 1 + \sqrt{\delta} |u|^{2q-2} )^{2}
}.
\end{aligned}
\end{equation}
Owing to Assumption \ref{assump:F(Nonlinearity)[TV4AC-25]}, we have
\begin{equation}\label{eq:ineq._|f|[TV4AC-25]}
\begin{aligned}
f'(u) 
\leq -c_f
(2q-1)
u^{2q-2} + c_{f,1}(1 + |u|^{2q-3}),
\quad
|f(u)| \leq c_f |u|^{2q-1} + c_{f,0}(1 + |u|^{2q-2}),
\end{aligned}
\end{equation}
leading to 
\begin{equation}\label{eq:ineq._derivative_f'_del[TV4AC-25]}
\begin{aligned}
f_\delta'(u)
&\leq
 \frac{ \big( -c_f (2q-1) |u|^{2q-2} + c_{f,1}(1 + |u|^{2q-3} ) \big)
          (  
          1 
            + 
             \sqrt{\delta}
          |u|^{ 2q-2 } )
          }
    {(  1 
            + 
           \sqrt{\delta}
          |u|^{ 2q-2 } )^{2} }
\\&\quad
+
         \frac{ 
      \big( c_f|u|^{2q-1} +  c_{f,0}( 1 + |u|^{2q-2} ) \big)
      (2q-2) \sqrt{\delta}
      |u|^{ 2q-3 } 
          }
    { (  1 
            + 
            \sqrt{\delta}
          |u|^{ 2q-2 } )^{2} } \\
&= 
         \frac{ 
     -c_f (2q-1)|u|^{2q-2} 
     +
     c_{f,1}(1 + |u|^{2q-3} ) 
          }
    {  (  1 
            + 
            \sqrt{\delta}
          |u|^{ 2q-2 } )^{2} }
\\
&\quad
+    
\frac{ 
    -c_f
     \sqrt{\delta}
     |u|^{ 4q-4 }
     +
      c_{f,1}
       \sqrt{\delta}
      (1 + |u|^{2q-3} ) 
          |u|^{ 2q-2 } 
        +
      c_{f,0}
       \sqrt{\delta}
      ( 1 + |u|^{2q-2} )
      (2q-2)
      |u|^{ 2q-3 } 
          }
    { (  1 
            + 
              \sqrt{\delta}
          |u|^{ 2q-2 } )^{2} }.
\end{aligned}
\end{equation}
%The following estimates hold true
Before moving on, we know that
for all $u \in \R$, there exist constants
%$C_{1}>0$, $C_{2}>0$ 
$C_1, C_2 >0$
independent of $ \delta$
while depending on $c_{f,1}, c_{f,0},q$ such that
\begin{equation}
\begin{aligned}
-
c_f (2q-1)
|u|^{2q-2} 
+
c_{f,1}(1 + |u|^{2q-3}) &\leq C_{1},
\\
-c_f
     |u|^{ 4q-4 }
     +
      c_{f,1}(1 + |u|^{2q-3} ) 
          |u|^{ 2q-2 } 
        +
      c_{f,0}( 1 + |u|^{2q-2} )
      &(2q-2)
      |u|^{ 2q-3 } 
      \leq C_{2}.
\end{aligned}
\end{equation}
Thus,
for any $u \in \R$ and $\delta >0$,
\begin{equation}
\begin{aligned}
f_\delta'(u)
\leq
\tfrac{ 
      C_{1}
          }
    {(  1 
            +    
            \sqrt{\delta}
          |u|^{ 2q-2 }
    )^{2} 
    }
+   \tfrac{ 
      C_{2}
      \sqrt{\delta}
          }
    {(  1 
            + 
            \sqrt{\delta}
          |u|^{ 2q-2 } 
    )^{2}
    }
\leq
C_1 + C_2,
\end{aligned}
\end{equation}
as asserted.
%Noting that $\delta \in (0,1]$, one immediately obtains \eqref{eq:bound_f'_delta[TV4AC-25]} with some constant $\widetilde{L}_f$ independent of $\delta$.
%
%In a similar manner,
Moreover,
Assumption \ref{assump:F(Nonlinearity)[TV4AC-25]} 
ensures
\begin{equation}
    |f'(u)|
    \leq
    c_f
    (2q-1)
    |u|^{2q-2}
    +
    c_{f,1}
    ( 
    1 
    +
    %\mathbbm{1}_{ \{q>1\} }
    |u|^{2q-3}
    ),
\end{equation}
which together with \eqref{eq:derivative_f'_del[TV4AC-25]}-\eqref{eq:ineq._|f|[TV4AC-25]}
%and \eqref{eq:ineq._derivative_f'_del[TV4AC-25]} 
implies
 for any $u \in \R$,
\begin{equation}\label{eq:|Df_delta|_V-derive[TV4AC-25]}
\begin{aligned}
|f'_\delta(u)|   
    &\leq
         \tfrac{ 
     C (
     1
     +
     |u|^{2q-2} 
     +
      \sqrt{\delta}
    |u|^{ 4q-4 }
    )
          }
    {(  1 
            + 
              \sqrt{\delta}
          |u|^{ 2q-2 } )^{2} }.
\end{aligned}    
\end{equation}
Hence,
it holds that for any $u \in \R$,
\begin{equation}
\begin{aligned}
|f'_\delta(u)|   
   % &
    \leq
     C (
     1
     +
     |u|^{2q-2} 
    )
 +
        \tfrac{ 
        C   \sqrt{\delta}
    |u|^{ 4q-4 } 
          }
    { 1 
            + 
              \sqrt{\delta}
          |u|^{ 2q-2 }  }   
%\\&
\leq
   % \widetilde{l}_f 
    C
    (
     1
     +
     |u|^{2q-2} 
    ),
\end{aligned}    
\end{equation}
and
\begin{equation}
|f'_\delta(u)|   
\leq
    C
    \Big(
    1
    +
    \delta^{ - \frac12 }
    +
    \tfrac{ 
    |u|^{ 2q-2 }
          }
    { 
       1 
       + 
         \sqrt{\delta}
        |u|^{ 2q-2 } 
    }
    \Big)
\leq
C 
(
1
+
\delta^{ - \frac12 }
).
\end{equation}
%
%
%Consequently,
%\begin{equation}
%|f'_\delta(u)|   
%\leq
%   \widetilde{l}_f 
%   \big(
%   1 
%   +
%   (
%   |u|^{2q-2}
%   \wedge
%   \delta^{ - \alpha \theta }
%   )
%   \big)
%,
%\end{equation}
%
%where $\widetilde{l}_f$ independent of $\delta$. 
The desired result \eqref{eq:|Df_delta|_V[TV4AC-25]} then follows immediately.
The proof is thus finished.
\end{proof}

%{\color{cyan} Recalling the property \eqref{eq:(u+v)f_tau(u)[TV4AC-25]} inherited by the regularized nonlinearity $f_\delta$, with $\tau$ therein replaced by $\delta$, the uniform moment bound of  $\mathbb{X}^{\delta}$ will thus be obtained, as stated in the following lemma.   }

\iffalse
Since the property \eqref{eq:(u+v)f_tau(u)[TV4AC-25]} is preserved by the regularized nonlinearity $f_\delta$, with
{\color{red}
parameters $\alpha,\beta = 1, \theta=\frac12$ fixed and} $\tau$ replaced by $\delta$,
the uniform moment bound of 
$\mathbb{X}^{\delta}$ follows, as stated in the following lemma.
\fi

Since $f_\delta$ in \eqref{eq:f_delta[TV4AC-25]} is a particular case of the modification $f_\tau$ in \eqref{eq:f_tau[TV4AC-25]} with $\alpha=\beta=1$ and $\theta=\tfrac12$,
the coercivity-type property \eqref{eq:(u+v)f_tau(u)[TV4AC-25]} still holds for $f_\delta$ under the condition $2 c_3^2 \sqrt{\delta} \leq c_0 \epsilon$.
Consequently, $\mathbb{X}^\delta$ admits the following uniform moment bounds.

\begin{lem}\label{lem:X^delta-Bound[TV4AC-25]}
       Let Assumptions \ref{assump:A(linear_operator)[TV4AC-25]}-\ref{assump:X_0(Initial Value)[TV4AC-25]}
       hold and
     let $ \mathbb{X}^\delta(t), t \geq 0$ be defined by \eqref{eq:spde_auxiliary[TV4AC-25]}.
     For any $p \geq 1$ and any integer $\rho \geq 1$,
     there exist constants $  C( p, q, \rho
     %\alpha 
     ) $ and $C(p,q)$ independent of $\delta$ such that
\begin{align}
\label{eq:X^delta-L^2rho-bound[TV4AC25]}
&
\sup_{t \geq 0}
\big\|
 \mathbb{X}^{\delta}(t) 
\big\|_{L^p( \Omega; L^{2\rho} )}
\leq
%e^{\frac{-\widetilde{c}_0 \epsilon^{-1} t}{4}}
\| \mathbb{X}^{\delta}_0\|_{ L^p(\Omega; L^{2\rho} ) }
+
C(p,q, \rho)
,
\\
\label{eq:X^delta-V-bound[TV4AC25]}
&
\sup_{t \geq 0}
\big\|
\mathbb{X}^{\delta} (t) 
\big\|_{L^p(\Omega;V)}
\leq
    \|\mathbb{X}^{\delta}_0 \|_{L^p(\Omega;V)}
    +
    C(p,q)
    \epsilon^{-1}
    \Big( 1 + \| \mathbb{X}^{\delta}_0\|^{2q-1}_{ L^{(2q-1)p}(\Omega; L^{4q-2} ) } \Big).
\end{align}
\end{lem}
\begin{proof}
  Letting $\mathbb{Y}^{\delta}(t) := \mathbb{X}^{\delta}(t) - \mathcal{O}_t, t\geq 0$, one shows
\begin{equation}
    \frac{ \partial \mathbb{Y}^{\delta}(t) }{ \partial t }
    =
     - A \mathbb{Y}^{\delta}(t) + \epsilon^{-1} F_\delta\big( \mathbb{Y}^{\delta}(t) + \mathcal{O}_t \big)  ,
     \
     t >0,
    \quad
     \mathbb{Y}^{\delta}(0)=\mathbb{X}^{\delta}_0.
\end{equation}
Applying integration by parts 
and 
Assumption \ref{assump:A(linear_operator)[TV4AC-25]},
we deduce that, for all $t >0$,
\begin{equation}
    \begin{aligned}
        \frac{ 
        \mathrm{d}
        \|\mathbb{Y}^{\delta}(t)\|^{2\rho}
        _{ L^{2\rho} }
        }
        { \mathrm{d} t }
   &=
         2\rho
        \int_{\mathcal{D}}
        \big( 
        \mathbb{Y}^{\delta}(t)(x) 
        \big)^{2\rho-1}
        \frac{ 
        \partial  \mathbb{Y}^{\delta}(t)(x)
        }
        { \partial t }
        \,\mathrm{d}x   
\\
   &=
        2\rho 
        \int_{\mathcal{D}}
        ( \mathbb{Y}^{\delta}(t)(x) )^{2\rho-1}
        \cdot
        \big( 
        -A 
        \mathbb{Y}^{\delta}(t)(x) 
        +
        \epsilon^{-1} 
        f_\delta( \mathbb{Y}^{\delta}(t)(x) + \mathcal{O}_t(x) ) 
        \big)
        \,\mathrm{d}x         
\\
   &=
       - 2\rho
       \langle
        ( \mathbb{Y}^{\delta}(t) )^{2\rho-1}
        ,
         A \mathbb{Y}^{\delta}(t)
       \rangle   
\\&\quad
       +
       2\rho 
       \int_{\mathcal{D}}
       ( \mathbb{Y}^{\delta}(t)(x) )^{2\rho-2}
        \cdot
        \epsilon^{-1}
      ( \mathbb{Y}^{\delta}(t)(x) )
      f_\delta( \mathbb{Y}^{\delta}(t)(x) + \mathcal{O}_t(x) 
      )
    \,\mathrm{d}x 
\\
   &\leq
       - 2\rho(2\rho-1)
       \langle
        ( \mathbb{Y}^{\delta}(t)  )^{2\rho-2}
        \nabla( \mathbb{Y}^{\delta}(t) )
        ,
        \nabla( \mathbb{Y}^{\delta}(t)  )
       \rangle 
\\&\
       +
        \rho
        \epsilon^{-1}
        \hspace{-0.3em}
       \int_{\mathcal{D}}
       ( \mathbb{Y}^{\delta}(t)(x) )^{2\rho-2}
        \big(
            - 
            2
            \widetilde{c}_{0,\delta} 
            | \mathbb{Y}^{\delta}(t)(x) + \mathcal{O}_t(x) |^2
            +
            \widetilde{c}_{1}
            ( 1 + |\mathcal{O}_t(x) |^{2q} )
            +
            \widetilde{c}_2
            \sqrt{\delta}
        \big)
    \,\mathrm{d}x,
\end{aligned}
\end{equation}
\iffalse
where the last inequality 
%{\color{red}stands,
%due to the fact
%}
{\color{blue} follows from}
\begin{equation}
    2 ( u + v ) f_{\delta}(u) 
    \leq
    -
%    \tfrac{c_0}{2(1 + \beta \tau^\theta)^{2\alpha}}
    2
    \widetilde{c}_{0,\delta}
    |u|^2
    +
     \widetilde{c}_1
     ( 1 + |v|^{2q} )
     +
     \widetilde{c}_2
      \sqrt{\delta}
     ,
     \quad
     \text{for all }
     u,v \in \R,
\end{equation}
{\color{red}
with notation $\widetilde{c}_{0,\delta} := \tfrac{c_0}{4(1 +\sqrt{\delta})^2}$,
similarly obtained as in Proposition \ref{prop:F_tau[TV4AC-25]}
for condition \eqref{eq:epsilon_condition_f_tau[TV4AC-25]} therein reducing to
$2 c_3^2 \sqrt{\delta} \leq c_0 \epsilon$,
which is always fulfilled 
%in the limit
as $\delta \to 0$.
}
\fi
where the last inequality follows from the coercivity estimate for $f_\delta$ under condition $2 c_3^2 \sqrt{\delta} \leq c_0 \epsilon$:
\begin{equation}
    2 ( u + v ) f_{\delta}(u) 
    \leq
    -
%    \tfrac{c_0}{2(1 + \beta \tau^\theta)^{2\alpha}}
    2
    \widetilde{c}_{0,\delta}
    |u|^2
    +
     \widetilde{c}_1
     ( 1 + |v|^{2q} )
     +
     \widetilde{c}_2
      \sqrt{\delta}
     ,
     \quad
     \text{for all }
     u,v \in \R,
\end{equation}
with 
 $\widetilde{c}_{0,\delta} := \tfrac{c_0}{4(1 +\sqrt{\delta})^2}$.
Since
$\langle
        ( \mathbb{Y}^{\delta}(t)  )^{2\rho-2}
        \nabla( \mathbb{Y}^{\delta}(t) )
        ,
        \nabla( \mathbb{Y}^{\delta}(t)  )
       \rangle \geq 0$,
       one further utilizes
the Young inequality 
\begin{equation}
  -4
  \widetilde{c}_{0,\delta}
  \mathbb{Y}^{\delta}(t)(x)
  \mathcal{O}_{t}(x)
\leq
 \widetilde{c}_{0,\delta}
  |\mathbb{Y}^{\delta}(t)(x)|^2 
+ 
  4
  \widetilde{c}_{0,\delta}
  |\mathcal{O}_t(x)|^2 
\end{equation}
to obtain
\begin{equation}
    \begin{aligned}
        \frac{ 
        \mathrm{d}
        \|\mathbb{Y}^{\delta}(t)\|^{2\rho}
        _{ L^{2\rho} }
        }
        { \mathrm{d} t } 
&\leq
        \rho 
        \epsilon^{-1}
        \hspace{-0.3em}
       \int_{\mathcal{D}}
       ( \mathbb{Y}^{\delta}(t)(x) )^{2\rho-2}
        \big(
            - 
            \widetilde{c}_{0,\delta}
            |\mathbb{Y}^{\delta}(t)(x)|^2
            +
             2
            \widetilde{c}_{0,\delta}
            |\mathcal{O}_t(x) |^2
            +
            \widetilde{c}_{1}
            ( 1 + |\mathcal{O}_t(x) |^{2q} )
            +
            \widetilde{c}_2
            \sqrt{\delta}
        \big)
    \,\mathrm{d}x 
\\
    &\leq    
    \int_{\mathcal{D}}
    - 
    \tfrac
    { \rho 
    \widetilde{c}_{0,\delta}
      \epsilon^{-1} }
    { 2 }
       ( \mathbb{Y}^{\delta}(t)(x) )^{2\rho}
     +
     C
     \epsilon^{-1} 
     \big(
     1 
     + 
     |\mathcal{O}_t(x) |^{2q\rho}
     +
     \delta^{ \frac{\rho}{2} }
     \big) 
    \,\mathrm{d}x 
\\
&= 
-   
\tfrac
    { \rho 
    \widetilde{c}_{0,\delta}
      \epsilon^{-1} }
    { 2 }
      \|\mathbb{Y}^{\delta}(t)\|^{2\rho}_{L^{2\rho}}
+ C \epsilon^{-1} 
\Big( 
1
+
\|\mathcal{O}_t\|^{2q\rho}_{V } 
     +
\delta^{ \frac{\rho}{2} }
\Big),
    \end{aligned}   
\end{equation}
where the Young inequality was also employed to get the second inequality.
Multiplying both sides of the above inequality by
$e^{ 
    \frac
    { \rho 
    \widetilde{c}_{0,\delta}
      \epsilon^{-1} t }
    { 2 }
}$ implies
\begin{equation}
\begin{aligned}
\frac{  \mathrm{d}
       \big(
       e^{ 
           \frac
    { \rho 
    \widetilde{c}_{0,\delta}
      \epsilon^{-1} t }
    { 2 }
       }
       \|\mathbb{Y}^{\delta}(t)\|^{2\rho}_{ L^{2\rho} }
       \big)
     }
     { \mathrm{d} t }
&\leq 
C 
       e^{ 
           \frac
    { \rho 
    \widetilde{c}_{0,\delta}
      \epsilon^{-1} t }
    { 2 }
       }
\epsilon^{-1} 
\Big( 
1 
+  
\|\mathcal{O}_t\|^{2q\rho}_{V}
+
\delta^{ \frac{\rho}{2} }
\Big).
\end{aligned}
\end{equation}
The Gronwall inequality then leads to
\begin{equation}
\begin{aligned}
       e^{ 
           \frac
    { \rho 
    \widetilde{c}_{0,\delta}
      \epsilon^{-1} t }
    { 2 }
       }
  \|\mathbb{Y}^{\delta}(t)\|^{2\rho}_{ L^{2\rho} }
&\leq
  \|\mathbb{Y}^{\delta}(0)\|^{2\rho}_{ L^{2\rho} }
  +
  C 
  \epsilon^{-1} 
  \int_0^t
       e^{ 
           \frac
    { \rho 
    \widetilde{c}_{0,\delta}
      \epsilon^{-1} s }
    { 2 }
       }
  \Big( 
  1 
  +  
  \sup_{s \geq 0 }
  \|\mathcal{O}_s\|^{2q\rho}_{V}
  +
\delta^{ \frac{\rho}{2} }
  \Big)
  \,\mathrm{d}s 
\\
&
=
  \|\mathbb{X}^{\delta}_0\|^{2\rho}_{ L^{2\rho} }
  +
  \tfrac{ 2 C }{\rho \widetilde{c}_{0,\delta} }
  \Big(
       e^{ 
           \frac
    { \rho 
    \widetilde{c}_{0,\delta}
      \epsilon^{-1} t }
    { 2 }
       }
  -
  1 
  \Big)
  \Big( 
  1 
  +  
 \sup_{s \geq 0 }
  \|\mathcal{O}_s\|^{2q\rho}_{V}
  +
\delta^{ \frac{\rho}{2} }
  \Big).
\end{aligned}
\end{equation}
Therefore, we get
\begin{equation}
\begin{aligned}
     \|\mathbb{Y}^{\delta}(t)\|_{ L^p(\Omega; L^{2\rho} ) }
    & \leq
     e^{
     \frac{
     -\widetilde{c}_{0,\delta}
     \epsilon^{-1} 
     t
     }
     {4}
     }
     \| \mathbb{X}^{\delta}_0\|_{ L^p(\Omega; L^{2\rho} ) }
     +
     C 
     (1 + \sqrt{\delta})^{ \frac{1}{\rho} }
     \Big(
     1
     + 
     \sup_{s \geq 0 }
     \|\mathcal{O}_s\|^q_{L^p(\Omega; V )}
     +
    \delta^{\frac{1}{4} }
     \Big)
\\
&\leq
e^{
\frac{
-
\widetilde{c}_{0,\delta}
\epsilon^{-1}
t
}
{4}
}
\| \mathbb{X}^{\delta}_0\|_{ L^p(\Omega; L^{2\rho} ) }
+
C(p,q, \rho ) ,
\end{aligned}
\end{equation}
where we recalled $\widetilde{c}_{0,\delta} := \tfrac{c_0}{4(1 +\sqrt{\delta})^2}$.
This together with Lemma \ref{lem:O_bound[TV4AC-25]} immediately gives \eqref{eq:X^delta-L^2rho-bound[TV4AC25]}.
Moreover, one uses \eqref{eq:E(t)_semigroup_property_V[TV4AC-25]}
%and \eqref{eq:E(t)_semigroup_property_V[TV4AC-25]}
%that $\|E(t)x\|_{V} \leq C (1 \wedge t)^{-\frac14} e^{-ct} \| x \| $ for any $t >0$ and $x \in H$ (see, e.g., \cite[(2.4)]{brehier2022ESAIM}),
%the Sobolev embedding inequality that $ \| x \|_V \leq C \| A^{\vartheta} x \|$, for $\vartheta \in (\frac{1}{4},1 \wedge \frac{\varrho}{2} ) $, $x \in V$, with $\varrho$ coming from Assumption \ref{assump:X_0(Initial Value)[TV4AC-25]},
%
%by following similar approach as Step 2 in Theorem \ref{thm:X^tau-bound[TV4AC-25]} and
to arrive at 
\begin{equation}\label{eq:deduce_X^delta_V_norm[TV4AC-25]}
\begin{aligned}
  &
  \sup_{t \geq 0}
    \| \mathbb{X}^{\delta} (t) \|_{L^p(\Omega;V)}
   \leq
   \sup_{t \geq 0}
   \bigg[
    \| E(t) \mathbb{X}^{\delta}_0 \|_{L^p(\Omega;V)}
    +
    \bigg\| 
    \int_0^t 
    E(t-s) 
    \epsilon^{-1}
    F_\delta(  \mathbb{X}^{\delta}(s) )
    \,\mathrm{d}s
    \bigg\|_{L^p(\Omega;V)}
    +
    \| 
   \mathcal{O}_t
    \|_{L^p(\Omega;V)}
    \bigg]
\\
   & \leq
    \| \mathbb{X}^{\delta}_0 \|_{L^p(\Omega;V)}
    +
    C \epsilon^{-1}
    \sup_{t \geq 0}
    \int_0^t
    e^{ 
    -
    \frac12
    \lambda_1
    (t-s)
    }
    (t-s)^{-\frac14}
   \Big( 1 
   +
     \|\mathbb{X}^{\delta}(s)\|^{2q-1}_{L^{(2q-1)p}(\Omega;L^{4q-2})} 
   \Big)
    \,
     \mathrm{d}s
    +
    C(p)
\\
   & \leq
    \|\mathbb{X}^{\delta}_0 \|_{L^p(\Omega;V)}
    +
    C(p,q)
    \epsilon^{-1}
    \Big( 1 + \| \mathbb{X}^{\delta}_0\|^{2q-1}_{ L^{(2q-1)p}(\Omega; L^{4q-2} ) } \Big),
\end{aligned}
\end{equation}
where we
used the  fact that
$
    \int_0^t
    e^{ 
    -
    \frac12
    \lambda_1
    (t-s)
    }
    (t-s)^{-\frac14}
     \,
     \mathrm{d}s
$ 
is uniform-in-time bounded.
The proof is thus finished.
\end{proof}

\subsection{The Kolmogorov equation and the regularity estimates}

To carry out the weak error analysis,
for all $\varphi \in \mathcal{C}^0_b(H)$, let us introduce the functions $\nu^\delta(t, x)$  with $t \geq 0, x \in H$, defined by 
\begin{equation}\label{eq:markov_nu[TV4AC-25]}
    \nu^\delta(t, x) := \mathbb E[\varphi (\mathbb{X}^{\delta}(t,x))],
\end{equation}
where $\mathbb{X}^{\delta}(t,x)$ solves \eqref{eq:spde_auxiliary[TV4AC-25]} with initial value $x$.
%One infers from \eqref{eq:|Df_delta|_V-derive[TV4AC-25]} that $\sup_{u \in \R} |f_\delta'(u)| \leq C \delta^{-\alpha \theta} $,
Since %\eqref{eq:|Df_delta|_V[TV4AC-25]}
%}
%namely, 
%{\color{red}
%the process $\mathbb{X}^{\delta}(t,x), t\geq 0$ admits a Lipschitz nonlinearity with the Lipschitz constant depending on $\delta$. 
%}
%{\color{blue}
%which 
%implies that
$F_\delta \colon H \to H$ is globally Lipschitz continuous with Lipschitz constant depending on $\delta$,
as indicated by \eqref{eq:|Df_delta|_V[TV4AC-25]},
%{\color{red}
%Therefore, it is known that $\nu^\delta(t, x)$ is the solution of the Kolmogorov equation (see, e.g., \cite{BREHIER2018Kolmogorov}):
%}
it then follows from \cite[Section 2]{BREHIER2018Kolmogorov} that  $\nu^\delta$  is the unique classical solution of the Kolmogorov equation satisfying
\begin{equation}\label{eq:Kolmogorov.nu[TV4AC-25]}
    \left\{ \begin{array}{l}
         \partial_t \nu^\delta(t,x) =  D \nu^\delta(t,x). (-Ax + \epsilon^{-1} F_\delta(x)) + \frac{1}{2} \sum_{j \in \mathbb N}D^2 \nu^\delta(t, x).(e_j, e_j), \\
         \nu^\delta(0, \cdot) = \varphi(\cdot), 
    \end{array}\right.
\end{equation}
with $t \geq 0, x\in H$.
%
%
%
%In what follows, we first present an estimate for the Kolmogorov equation 
%{\color{red}
%that depends on $\epsilon^{-1}$ exponentially.
%}
%{\color{blue}
%We are now in a position to establish a rough estimate for the derivatives of $\nu^\delta$.
%}
%
In what follows,
we first begin with a rather rough estimate for the derivatives of $\nu^\delta$, and later 
%{\color{red}refine}
sharpen
it to the desired estimate that
%{\color{red} overcomes the exponential dependence on $\epsilon^{-1}$}
reduces the dependence on $\epsilon^{-1}$ from exponential to polynomial order
(as shown in Proposition
\ref{prop:markov_exponential_estimate_pol[TV4AC-25]}).

\begin{lem}
\label{lem:markov_exponential_estimate_exp[TV4AC-25]}
%{\color{red}
%Under Assumptions \ref{assump:A(linear_operator)[TV4AC-25]}-\ref{assump:X_0(Initial Value)[TV4AC-25]},
%for all $\varphi \in \mathcal{C}^0_b(H)$, with $t > 0$, for all $x, y \in H$, then $\nu^\delta(t,x)$ defined by \eqref{eq:markov_nu[TV4AC-25]} obeys
%}
Under Assumptions \ref{assump:A(linear_operator)[TV4AC-25]}--\ref{assump:X_0(Initial Value)[TV4AC-25]},  
let $\varphi \in \mathcal{C}_b^0(H)$ and $\nu^\delta(t,x)$ be given by \eqref{eq:markov_nu[TV4AC-25]}.  
Then, for all $t>0$ and $x,y \in H$, 
\begin{equation}
\begin{aligned}
    |D \nu^\delta(t, x).y| 
   &\leq 
    e^{
    \epsilon^{-1} 
    \widetilde{L}_f
 %   (1 + \beta \sqrt{\delta})
    t}
     t^{-\frac{1}{2}}
     \| \varphi \|_0 
     \|y\|  .
\end{aligned}
\end{equation}

\end{lem}

\begin{proof}

%{\color{red}
%When $\nu^\delta(0, \cdot) = \varphi(\cdot)$ is of $\mathcal{C}^0_b(H)$, by a Bismut-Elworthy-Li formula (see \cite{cerrai2001second} for instance) the derivative $D \nu^\delta(t, x).y$ satisfies
%}
Recall that
$\nu^\delta(0, \cdot) = \varphi(\cdot)$ 
with $ \varphi \in \mathcal{C}^0_b(H)$.
Then the Bismut-Elworthy-Li formula (see \cite{cerrai2001second} for instance) yields
\begin{equation}\label{lem_D_of_C0}
    D \nu^\delta(t, x). y = t^{-1} \mathbb E \left[ \varphi(\mathbb{X}^{\delta}(t,x)) \int_0^t \langle  \xi^\delta_y(s,x), \mathrm{d}W(s) \rangle \right],
\end{equation}
with 
$ \xi^\delta_y(t,x), t\geq 0$,
given by 
\begin{equation}\label{eq:xi_def[TV4AC-25]}
\left\{\begin{array}{l}
\mathrm{d} 
\xi^\delta_y(t,x) 
= 
-A 
 \xi^\delta_y(t,x)
 \,\mathrm{d} t
 +
 \epsilon^{-1} 
F_\delta'(\mathbb{X}^{\delta}(t,x))
 \xi^\delta_y(t,x)
 \,\mathrm{d} t 
 ,
 \quad
 t > 0, 
 \\
 \xi^\delta_y(0,x)
 =
 y .
\end{array}\right.
\end{equation}
%
%{\color{red}
%Recalling \eqref{eq:def_DF_delta[TV4AC-25]} and \eqref{eq:bound_f'_delta[TV4AC-25]} one knows
%}
In view of \eqref{eq:def_DF_delta[TV4AC-25]} and \eqref{eq:bound_f'_delta[TV4AC-25]}, we obtain
\begin{equation}
  %\sup_{\delta >0}
 % \sup_{x \in H}
    \langle
    F_\delta'(x) y 
    ,
    y
    \rangle
    \leq 
    \widetilde{L}_f 
 %   (1 + \beta \sqrt{\delta})
    \|y\|^2
    ,
    \quad 
    \forall 
    x, y \in H.
\end{equation}
Then it follows from Assumption \ref{assump:A(linear_operator)[TV4AC-25]} that
\begin{equation}\label{eq:xi_y(deduce)[TV4AC-25]}
\begin{split}
       \frac{
       \mathrm{d}
       \|\xi^\delta_y(t,x)\|^2 
       }
       {\mathrm{d} t}
      &=
       2 
       \Big\langle 
       \frac{\mathrm{d} \xi^\delta_y(t,x)}
       {\mathrm{d} t } 
       , 
       \xi^\delta_y(t,x) 
       \Big\rangle
\\
       &=
       2  \langle -A \xi^\delta_y(t,x), \xi^\delta_y(t,x) \rangle 
       + 
       2 \epsilon^{-1} 
       \langle 
       F_\delta' (\mathbb{X}^{\delta}(t,x)) \xi^\delta_y(t,x)
       ,
       \xi^\delta_y(t,x) \rangle \\
       &\leq
       2 \epsilon^{-1} \widetilde{L}_f 
   %    (1 + \beta \sqrt{\delta})
       \|\xi^\delta_y(t,x)\|^2 .
\end{split}
\end{equation}
The Gronwall inequality further yields
\begin{equation}\label{eq:bound_xi_|y|[TV4AC-25]}
\| \xi^\delta_y(t,x)\|
\leq 
e^{\epsilon^{-1} 
\widetilde{L}_f
%(1 + \beta \sqrt{\delta})
t
}
\|y\|.   
\end{equation}
%
%
%{\color{red}
%Applying It\^o's isometric formula to the expression (\ref{lem_D_of_C0}), together with the bound of $\|\xi^\delta_y(t,x)\|$, one can deduce that
%}
Using the Hölder inequality, the It\^o isometry and \eqref{eq:bound_xi_|y|[TV4AC-25]} shows
\begin{equation}\begin{split}
    |D \nu^\delta(t, x). y| 
    &\leq
    t^{-1} \| \varphi \|_0 \mathbb E \bigg[\bigg|  \int_0^t \langle  \xi^\delta_y(s,x), \mathrm{d} W(s) \rangle \bigg| \bigg] \\
    &\leq
    t^{-1} \| \varphi \|_0 \left( \int_0^t \mathbb E [ \| \xi^\delta_y(s,x)\|^2 ] \,\mathrm{d}s \right)^{\frac{1}{2}}\\
    &\leq
    e^{
    \epsilon^{-1} 
    \widetilde{L}_f
 %   (1 + \beta \sqrt{\delta})
    t
    }
     t^{-\frac{1}{2}}
     \| \varphi \|_0 
     \|y\|,
    \end{split}
\end{equation}
for all $t >0$ and all $x, y \in H$.
The proof is thus completed.
\end{proof}

Bearing the above lemma in mind,
one obtains the following result that demonstrates
the error between the original problem
$X$ 
%$X(\cdot,X_0)$ 
and the auxiliary process
$\mathbb{X}^{\delta}$.
%$\mathbb{X}^{\delta}(\cdot,X_0)$ \eqref{eq:spde_auxiliary[TV4AC-25]}.
%$X(t,X_0),t \geq 0$ and $\mathbb{X}^{\delta}(t,X_0),t \geq 0$ 

\begin{prop}\label{prop:error_of_X^delta-X[TV4AC-25]}

Under Assumptions \ref{assump:A(linear_operator)[TV4AC-25]}-\ref{assump:X_0(Initial Value)[TV4AC-25]}, let $X(t, X_0), t\geq 0$ be the mild solution of the model \eqref{eq:spde[TV4AC-25]} and $\mathbb{X}^{\delta}(t, X_0), t\geq 0$ be the auxiliary process defined by \eqref{eq:spde_auxiliary[TV4AC-25]} with the initial value $\mathbb{X}^{\delta}_0 = X_0$. 
%{\color{red}
%Then for any $\theta >0 $, there exists a positive constant $ C( X_0 ,q, \epsilon^{-1}, t) $ depending on $\epsilon^{-1}$ and $t$ such that
%}
Then
for all $\varphi \in \mathcal{B}_b(H)$,
\begin{equation}
    \Big| \mathbb{E} \big[\varphi \big(X(t, X_0) \big) \big] - \mathbb{E} \big[\varphi(\mathbb{X}^{\delta}(t, X_0)) \big] \Big| 
\leq
    C
 \epsilon^{-1}
  e^{
  \epsilon^{-1} 
  \widetilde{L}_f
% (1 + \beta \sqrt{\delta})
  t
  } 
 t^{\frac12}
 \| \varphi \|_0
\Big(
1 
+
\|X_0\|_V^{  4q-3 }
\Big)
\sqrt{\delta}.
\end{equation}
\end{prop}

\begin{proof}
For brevity, we denote $X(s):=X(s,X_0)$, $s \in [0,t]$.
Recalling $\nu^\delta(t, x) = \mathbb E[\varphi (\mathbb{X}^{\delta} (t,x))]$ and using It\^o's formula, 
together with the associated Kolmogorov equation \eqref{eq:Kolmogorov.nu[TV4AC-25]},
we get
\begin{equation}
\begin{split}
 &
 \Big| \mathbb{E} \big[\varphi \big(X(t, X_0) \big) \big] - \mathbb{E} \big[\varphi(\mathbb{X}^{\delta}(t, X_0)) \big] \Big|
\\
        &=
        \big| \mathbb{E} \big[ \nu^\delta \big(0, X(t) \big) \big] - \mathbb{E} \big[ \nu^\delta \big(t, X_0 \big) \big] \big|
\\
&=
 \mathbb E \bigg[ 
 \int_{0}^{t}
-\Big(
D \nu^\delta \big( t -s, X(s) \big). \big(-A X(s) + \epsilon^{-1} F_\delta \big( X(s) \big) \big) 
+  \frac{1}{2} \sum_{j \in \mathbb N} D^2 \nu^\delta \big( t -s, X(s) \big).(e_j, e_j)
\Big) \,\mathrm{d}s \\
&\quad
+
\int_{0}^{t}
\bigg(
D \nu^\delta \big( t-s, X(s) \big).\big( -A X(s) +  \epsilon^{-1} F ( X(s) )  \big)
%\\&\quad
+
\frac{1}{2} \sum_{j \in \mathbb N} D^2 \nu^\delta \big( t-s, X(s) \big).(e_j, e_j) 
\bigg)
\,\mathrm{d}s 
\bigg]\\
&= \mathbb E \bigg[ 
\int_{0}^{t} D \nu^\delta \big( t-s, X(s) \big). 
\big( \epsilon^{-1} F( X(s) )  -  \epsilon^{-1} F_\delta \big( X(s) \big) \big) \,\mathrm{d}s
\bigg].
\end{split}
\end{equation}
Using Lemma \ref{lem:markov_exponential_estimate_exp[TV4AC-25]} and 
the property
$    | f_{\delta}(u) - f(u) |
  \leq
    \sqrt{\delta}
     |u|^{ 2q-2 }
     |f(u)|$ acquired similarly as in Proposition \ref{prop:F_tau[TV4AC-25]},
one further deduces
\begin{equation}
\begin{aligned}
 &
 \Big| \mathbb{E} \big[\varphi \big(X(t, X_0) \big) \big] - \mathbb{E} \big[\varphi(\mathbb{X}^{\delta}(t, X_0)) \big] \Big|
\\
&\leq
 \epsilon^{-1}
 \| \varphi \|_0
\mathbb E \bigg[ \int_{0}^{t}  
     e^{
     \epsilon^{-1} 
     \widetilde{L}_f 
   %  (1 + \beta \sqrt{\delta})
     (t-s)
     }
     (t-s)^{-\frac12}
     \|  F( X(s) )  - F_\delta ( X(s) )  \| 
      \,\mathrm{d}s
     \bigg] 
\\
&\leq
\sqrt{\delta}
 \epsilon^{-1}
 \| \varphi \|_0
\mathbb E 
\bigg[ \int_{0}^{t}  
 %   \epsilon^{-1}
     e^{
     \epsilon^{-1}
     \widetilde{L}_f 
  %   (1 + \beta \sqrt{\delta})
     (t-s)
     } 
     (t-s)^{-\frac12}
     \|X(s)\|_V^{ 2q-2 } 
    \|F(X(s))\|
      \,\mathrm{d}s
     \bigg] 
\\
&\leq
C
 \epsilon^{-1}
  e^{
  \epsilon^{-1}
  \widetilde{L}_f 
 % (1 + \beta \sqrt{\delta})
  t} 
 t^{\frac12}
 \| \varphi \|_0
\Big(
1 
+
\|X_0\|_V^{  4q-3 }
\Big)
\sqrt{\delta},
\end{aligned}
\end{equation}
%{\color{blue}
%\begin{equation}
%\begin{aligned}
%&\big|\E[\varphi(X(t,X_0))]-%\E[\varphi(\mathbb{X}^{\delta}(t,X_0))]\big|
%\\
%&\leq \epsilon^{-1}\|\varphi\|_0\,\delta^\theta\,
%\mathbb{E}
%\left[
%\int_0^t e^{\frac{\widetilde L_f}{\epsilon}(t-s)}(t-s)^{-1/2}
%\|X(s)\|_V^{\frac{2q-1}{\alpha}}\|F(X(s))\|\,\mathrm ds
%\right]\\
%&\leq 
%\epsilon^{-1}
%C_*(X_0,q,\alpha,\theta)
%\|\varphi\|_0
%\delta^\theta
%\int_0^t e^{\frac{\widetilde L_f}{\epsilon}(t-s)}(t-s)^{-1/2}\,\mathrm ds \\
%& \leq
%e^{\frac{\widetilde L_f}{\epsilon}t}
%t^{1/2}
%\epsilon^{-1}
%C_*(X_0,q,\alpha,\theta)
%\|\varphi\|_0\,\delta^\theta\,
%\end{aligned}
%\end{equation}
%}
where Assumption \ref{assump:F(Nonlinearity)[TV4AC-25]} and the bound \eqref{eq:bound_X(t)[TV4AC-25]} were utilized.
The proof is completed.
\end{proof}

%{\color{red}
%Here, we remark that the constant $C( X_0 ,q, \epsilon^{-1}, t)$ in  Proposition \ref{prop:error_of_X^delta-X[TV4AC-25]} depends on $t$ and $\epsilon^{-1}$ exponentially. However, we always let $\delta \rightarrow 0$ in the subcequent analysis. Equipped with Proposition \ref{prop:error_of_X^delta-X[TV4AC-25]}, we are now well-prepared to get the estimate for the Kolmogorov equation that depends on $\epsilon^{-1}$ polynomially.
%}

Despite an exponential dependence on both $t$ and $\epsilon^{-1}$, such a rough estimate for $D\nu^{\delta}$ established in Lemma \ref{lem:markov_exponential_estimate_exp[TV4AC-25]}  
%{\color{red} the subsequent analysis will be carried out in the limit $\delta \to 0$, so that the associated approximation error vanishes.}
is invoked only in the estimate for the error between $X$ and $\mathbb{X}^{\delta}$, 
which is $O(\sqrt{\delta})$ and thus vanishes as $\delta \to 0$.
In the next proposition, we derive a refined bound of $D\nu^{\delta}$ with a polynomial dependence on $\epsilon^{-1}$.

\begin{prop}
\label{prop:markov_exponential_estimate_pol[TV4AC-25]}
Let Assumptions \ref{assump:A(linear_operator)[TV4AC-25]}-\ref{assump:X_0(Initial Value)[TV4AC-25]} 
%and condition \eqref{eq:epsilon_condition_f_delta[TV4AC-25]} 
hold.
Then for all $\varphi \in \mathcal{C}^0_b(H)$, 
$x, y \in H$, with $t > 0$,
there exist
positive constants 
$C(\vartheta, q
%\alpha
) 
,
C( \widetilde{L}_f, t, \epsilon^{-1} ) $ with $\vartheta \in [0, \frac34)$, such that

\begin{equation}
\begin{aligned}
    | D \nu^\delta(t, x).y | 
    \leq 
&  \  C(\vartheta, q ) \|\varphi\|_0  e^{-rt} 
(
1
+
(t \wedge \epsilon)^{ -\frac{1}{2}-\vartheta } 
)
%(1 + \beta \sqrt{\delta})^{\frac12 + \vartheta}
    \Big(
    1 
     + 
     \| 
     x
     \|^{2q-2}_{L^{4q-4}}
     \Big)
     \| A^{-\vartheta} y \|
\\
&\qquad     
  +
 C(  \widetilde{L}_f, t, \epsilon^{-1} )
%\epsilon^{-\frac32}
%e^{  2\widetilde{L}_f}
%  e^{\epsilon^{-1} \widetilde{L}_f (t-2\epsilon)} 
%( t-2\epsilon)^{\frac12}
 \| \varphi \|_0
 \|y\|
\Big(
1 
+
\sup_{s \geq 0}
\|\mathbb{X}^{\delta}(s,x)\|
^{  4q-3 }
_{
L^{  8q-6 }
( \Omega; V  )}
\Big)
\sqrt{\delta},
\end{aligned}
\end{equation}
where $r>0$ comes from Proposition \ref{prop:V-uniform_ergodicity_X(t)[TV4AC-25]}.

\end{prop}

%%%%%%%%%%%%%%%%%%%%%%%%%%%%
%%%%%%%%%%%%%%%%%%%%%%%%%%%%%%
% old version (no conditioned expressed) %
\iffalse
\begin{equation}
\begin{aligned}
    | D \nu^\delta(t, x).y | 
    \leq 
&    C(\vartheta) \|\varphi\|_0  e^{-rt} 
(
1
+
(t \wedge \epsilon)^{ -\frac{1}{2}-\vartheta } 
)
     \Big(1 + \sup_{s \in [0,t]} 
     \E 
     \big[ 
     \|
     {\color{cyan}
     \mathbb{X}^{\delta}(s,x)
     }
     \|_{ L^{2q-2} }^{2q-2} 
     \big] 
     \Big)
     \| A^{-\vartheta} y \|
\\
&\qquad     
{\color{cyan}
+
 C
\epsilon^{-\frac32}
e^{  2\widetilde{L}_f}
  e^{\epsilon^{-1} \widetilde{L}_f (t-2\epsilon)} 
 (t-2\epsilon)^{\frac12}
 \| \varphi \|_0
\big(
1 
+
\|\mathbb{X}^{\delta}(\epsilon,x)\|^{  \frac{2q-1}{\alpha} + 3 }_{
L^{  \frac{2(2q-1)}{\alpha} + 6 }
( \Omega; V  )}
\big)
\|y\|
\delta^{\theta}.
}
\end{aligned}
\end{equation}
%%%%%%%%%%%%%%%%%%%%%%%%%%%%
%%%%%%%%%%%%%%%%%%%%%%%%%%%%%%
% old version (no conditioned expressed) %

{\color{red}
\textbf{Notice:
$D \nu $ is a function of the (conditioned)  initial value $x$.
}
}
\fi

\begin{proof}

Below, we divide the proof into two parts:  
$t \in (0, 2\epsilon ]$ and $t > 2\epsilon$.
\begin{itemize}
    \item Part $1$: $t \in (0, 2\epsilon ]$.
\end{itemize}
%{\color{red}
%Denote the semigroup operators $(P_t^\delta)_{t \geq 0}$ satisfying $P_t^\delta \varphi(\cdot) = \nu^\delta(t,\cdot), t \geq 0$.
%}
Let \((P_t^\delta)_{t\ge0}\) be the Markov semigroup associated with \(\mathbb{X}^{\delta}\) in \eqref{eq:spde_auxiliary[TV4AC-25]}, so that
\(P_t^\delta\varphi(x)=\nu^\delta(t,x)\).
%
%{\color{red}
%As indicated in \cite{cerrai2001second}, for $\nu^\delta(0, \cdot) = \varphi(\cdot) \in \mathcal{C}^0_b(H)$, one knows $P_t^\delta \varphi \in \mathcal{C}^1_b(H) \subset \mathcal{C}^0_b(H)$ for $t >0$.
%Denote $ \| \varphi \|_1 := \sup _{x \in H}  \sup _{h \in H, \|h\| \leq 1 } |D \varphi(x). h| $ for any $\varphi \in \mathcal{C}_b^1(H)$.
%Recall the notation of $\| \cdot \|_1$, One thus infers from Lemma \ref{lem:markov_exponential_estimate_exp[TV4AC-25]} that for any $t>0$, $ \varphi \in \mathcal{C}_b^0(H)$,
%\begin{equation}\label{eq:markov_semigroup_estimate}
%    \| P_t^\delta \varphi \|_1
%    \leq
%      e^{\epsilon^{-1} \widetilde{L}_f t}
%     t^{-\frac{1}{2}}
%     \| \varphi \|_0 .
%\end{equation}
%}
Denote $ \| \varphi \|_1 := \sup _{x \in H}  \sup _{h \in H, \|h\| \leq 1 } |D \varphi(x). h| $ for any $\varphi \in \mathcal{C}_b^1(H)$.
As a direct consequence of Lemma \ref{lem:markov_exponential_estimate_exp[TV4AC-25]},
for any $t>0$ and $ \varphi \in \mathcal{C}_b^0(H)$,
one knows
$P_t^\delta \varphi \in \mathcal{C}^1_b(H) \subset \mathcal{C}^0_b(H)$ and
\begin{equation}\label{eq:markov_semigroup_estimate}
    \| P_t^\delta \varphi \|_1
    \leq
      e^{
      \epsilon^{-1}
      \widetilde{L}_f 
   %   (1 + \beta \sqrt{\delta})
      t
      }
     t^{-\frac{1}{2}}
     \| \varphi \|_0 .
\end{equation}
%
%{\color{red}
%Furthermore, for $t>0$, 
%the Markov property illustrates that
%for all  $ t, s \in (0, \infty)$, for $\varphi \in \mathcal{C}_b^0(H)$,
%$
%    P_{t + s}^\delta \varphi = P_t^\delta(P_s^\delta \varphi ) 
%$.
%Thus for all $t >0$, $\nu^\delta(t, \cdot) = P_{\frac{t}{2}}^\delta ( P_{\frac{t}{2}}^\delta \varphi(\cdot) )$, this along with the fact that $P_{\frac{t}{2}}^\delta \varphi \in \mathcal{C}_b^1(H)$ with $t >0$ yields
%\begin{equation}\label{D transformation of C_1}
%    D \nu^\delta(t, x).y
%    =
%    D P_{\frac{t}{2}}^\delta( %P_{\frac{t}{2}}^\delta \varphi(x) ).y
%    =
%    \mathbb E [D P_{\frac{t}{2}}^\delta \varphi(\mathbb{X}^{\delta}(\tfrac{t}{2},x)).\xi^\delta_y(\tfrac{t}{2}) ],
%\end{equation}
%for all $x, y \in H$.
%
%For all $x, y \in H$ and $t > 0$, one has
%\begin{equation}\label{prop_Du(t,x)_step1}
%   |D \nu^\delta(t, x).y| 
 %   \leq 
%    \|  P_{\frac{t}{2}}^\delta \varphi \|_1 \mathbb E [ \| \xi^\delta_y(\tfrac{t}{2}) \| ].
%\end{equation}
%}
Furthermore, the Markov property of $\mathbb{X}^{\delta}$ \eqref{eq:spde_auxiliary[TV4AC-25]} implies, for all $t>0$ and $x\in H$,  
\begin{equation}
\label{eq:D-nu-step1}
\nu^\delta(t,x) 
= 
\E
\left[ 
\nu^\delta
\left(
\tfrac{t}{2}, \mathbb{X}^{\delta}(\tfrac{t}{2}, x)
\right) 
\right]
= P_{\frac{t}{2}}^\delta
\nu^\delta
\left(
\tfrac{t}{2}, \cdot \right)
\big(
\mathbb{X}^{\delta}(\tfrac{t}{2}, x)
\big),
\end{equation}
By \eqref{eq:markov_semigroup_estimate} and Lemma \ref{lem:markov_exponential_estimate_exp[TV4AC-25]},
one notes that 
both functions 
$\nu^\delta(t, \cdot)$ and $\mathbb{X}^{\delta}(t,\cdot)$ are continuously differentiable for $t>0$,
implying that $\nu^\delta
\left(\tfrac{t}{2}, \mathbb{X}^{\delta}(\tfrac{t}{2},x)\right)$
is also continuously differentiable.
As a result, 
using the chain rule to \eqref{eq:D-nu-step1} gives
\begin{equation}\label{eq:chain_rule}
D
\left( \nu^\delta(\tfrac{t}{2}, \mathbb{X}^{\delta}(\tfrac{t}{2}, x)) 
\right) y
= D\nu^\delta(\tfrac{t}{2}, \mathbb{X}^{\delta}(\tfrac{t}{2}, x)) \, 
\xi^\delta_y(\tfrac{t}{2},x),
\quad
\forall x, y \in H,
\end{equation}
where $\xi^\delta_y$ solves the variational equation \eqref{eq:xi_def[TV4AC-25]},
leading to
\begin{equation}\label{prop_Du(t,x)_step1}
    |D \nu^\delta(t, x).y| 
    \leq 
    \|  P_{\frac{t}{2}}^\delta \varphi \|_1 \mathbb E [ \|  \xi^\delta_y(\tfrac{t}{2},x) \| ]
    \leq
    \sqrt{2}
    e^{
    \epsilon^{-1} 
    \widetilde{L}_f
   % (1 + \beta \sqrt{\delta})
    t}
     t^{-\frac{1}{2}}
     \| \varphi \|_0 
     \mathbb E [ \|  \xi^\delta_y(\tfrac{t}{2},x) \| ].
\end{equation}
By denoting
\begin{equation}
\label{eq:wlide-xi}
\widetilde{\xi}^\delta_y(t,x):= 
 \xi^\delta_y(t,x) - E(t) y, 
\quad
t \geq 0,
\end{equation}
and utilizing \eqref{eq:xi_def[TV4AC-25]}, one has
\begin{equation}
\mathrm{d}
\widetilde{\xi}^\delta_y(t,x)
=
( 
-A
+ 
\epsilon^{-1} 
F_\delta' (\mathbb{X}^{\delta}(t,x)) )
\widetilde{\xi}^\delta_y(t,x)
\,
\mathrm{d}t
+
\epsilon^{-1} 
F_\delta' (\mathbb{X}^{\delta}(t,x))
E(t)y \,\mathrm{d}t.
\end{equation}
As a result, 
\begin{equation}
\widetilde{\xi}^\delta_y(t,x)
=
\int_0^t \Gamma(t,s) 
\epsilon^{-1}
 F_\delta'(\mathbb{X}^{\delta}(s,x)) 
E(s)y \,\mathrm{d}s,
\end{equation}
where $\Gamma(t,s)$ is the evolution operator associated with the linear equation
\begin{equation}
    \frac{ \partial \Gamma(t,s)u }{ \partial t}
    =
    (-A + \epsilon^{-1} F_\delta'( \mathbb{X}^{\delta}(t,x) ) ) \Gamma(t,s)u,
    \quad
    \Gamma(s,s)u = u \in H.
\end{equation}
Clearly, following a similar way as in \eqref{eq:xi_y(deduce)[TV4AC-25]}-\eqref{eq:bound_xi_|y|[TV4AC-25]},
%recalling that $c_0 := \min\{ \lambda_1, \lambda_1 - \epsilon^{-1} \lambda_F\}$ 
one can deduce
\begin{equation}
    \| \Gamma(t,s) y \|  
    \leq
    e^{
    \epsilon^{-1}
    \widetilde{L}_f
  %  (1 + \beta \sqrt{\delta})
    (t-s) 
    } \|y\|.
\end{equation}
%
%{\color{red}
%any conditions on variation-of-constants formula?
%}
%
Therefore, one then shows that for any 
$\vartheta \in [0,\frac34)$,
\begin{equation}
\label{eq:estimate-wlide-xi}
\begin{split}
    \big\|
 \widetilde{\xi}^\delta_y(t,x)
    \big\|
    &
    =
\bigg\|
\int_0^t 
\Gamma(t,s) 
\epsilon^{-1} 
F_\delta'(\mathbb{X}^{\delta}(s,x))
E(s)y 
\,\mathrm{d}s 
\bigg\| 
\\
& 
\leq
\epsilon^{-1}  
\int_0^t
e^{
\epsilon^{-1} 
\widetilde{L}_f
%(1 + \beta \sqrt{\delta})
(t-s)
} 
\| F_\delta'(\mathbb{X}^{\delta}(s,x))\| 
\|E(s)y\|_V
\,\mathrm{d}s
\\
&
\leq
C 
\epsilon^{-1} 
\int_0^t
( 
1 
+ 
\|\mathbb{X}^{\delta}(s,x)\|_{ L^{4q-4} }^{2q-2}
)
e^{
\epsilon^{-1} 
\widetilde{L}_f
%(1 + \beta \sqrt{\delta})
(t-s) 
} 
( s \wedge 1 )^{-\frac14-\vartheta} 
e^{- cs} 
\,\mathrm{d}s
\|A^{-\vartheta} y\|,
\end{split}
\end{equation}
where we also used 
\eqref{eq:E(t)_semigroup_property[TV4AC-25]},
\eqref{eq:E(t)_semigroup_property_V[TV4AC-25]} 
%the property that
%$\|E(t)x\|_{V} \leq C (1 \wedge t)^{-\frac14} e^{-ct} \| x \| $
%for any $t >0$ and $x \in H$,
%{\color{red}
%any reference?
%}
and \eqref{eq:|Df_delta|_V[TV4AC-25]}.
Combining \eqref{eq:wlide-xi} and \eqref{eq:estimate-wlide-xi} yields, 
for $t \leq 2 \epsilon$ and $p \geq 1$,
%$t \leq 2\epsilon = \tfrac{2\epsilon}{(1 + \beta \sqrt{\delta} )}$,
\begin{equation}\label{eq:bound_xi_Ay[TV4AC-25]}
    \begin{split}
      %  \mathbb E [ 
        \| \xi^\delta_y(t,x) \|_{L^p(\Omega;H)}
       % ] 
        &\leq
         \|E(t) y\| 
         + 
       %  \mathbb E [ 
         \|\widetilde{\xi}^\delta_y(t,x)\|_{L^p(\Omega;H)}
        % ]
         \\
        &\leq
        C(\vartheta) 
        e^{ 2 \widetilde{L}_f }
        ( 1 + t^{-\vartheta} )
        \Big(1 
        + 
        \sup_{s \in [0,t]}
    %    \E
     %   \big[ 
        \|
        \mathbb{X}^{\delta}(s,x)
        \|
        ^{2q-2}
        _
        {L^{p(2q-2)}(\Omega; L^{4q-4} )}
    %    { L^{4q-4} }^{2q-2}
   %    \big] 
        \Big) 
        \| A^{-\vartheta} y \|.
    \end{split}
\end{equation}
Hence 
it follows from \eqref{prop_Du(t,x)_step1} and
\eqref{eq:bound_xi_Ay[TV4AC-25]} that
for $t \leq 2\epsilon$,
\begin{equation}\label{eq:markov-semigroup-estimate(t<epsilon)[TV4AC-25]}
\begin{aligned}
    |D \nu^\delta(t, x).y| 
   &\leq 
     C(\vartheta)  
     (1 + t^{-\frac{1}{2}-\vartheta}) 
     \|\varphi\|_0
     \Big(1 
     +
     \sup_{s \in [0,t]} 
 %    \E
 %    \big[
     \|
     \mathbb{X}^{\delta}(s,x)
     \|
         ^{2q-2}
        _
        {L^{2q-2}(\Omega; L^{4q-4} )}
     %_{ L^{4q-4} }^{2q-2} 
     %\big] 
     \Big) 
     \| A^{-\vartheta} y \| 
\\
   &\leq 
     C(\vartheta) 
     e^{2r\epsilon} 
     e^{-rt}
     (1 + t^{-\frac{1}{2}-\vartheta}) 
     \|\varphi\|_0
     \Big(
     1 
     + 
     \sup_{s \in [0,t]} 
 %    \E
 %    \big[
     \|
     \mathbb{X}^{\delta}(s,x)
     \|
         ^{2q-2}
        _
        {L^{2q-2}(\Omega; L^{4q-4} )}
     %_{ L^{4q-4} }^{2q-2} 
     %\big] 
     \Big) 
     \| A^{-\vartheta} y \|,
\end{aligned}
\end{equation}
as asserted.

\begin{itemize}
    \item Part $2$: $t > 2\epsilon$.
\end{itemize}

Denote $\nu(t,x) := \E[ \varphi( X(t,x) ) ], t \geq 0, x \in H$, for all $\varphi \in \mathcal{C}^0_b(H)$.
The Markov property (see, e.g., \cite{cerrai2001second}) ensures that
$\nu( t, x ) = \mathbb{E}[ \nu( t-\epsilon, X(\epsilon,x)  ) ]$ for $t > \epsilon$.
By Proposition \ref{prop:V-uniform_ergodicity_X(t)[TV4AC-25]}, there exists a unique invariant measure $\mu_X$ for \eqref{eq:spde[TV4AC-25]} (see, e.g., \cite{da2006introduction}), and it holds that
\begin{equation}
    \Big|
    \nu( t- \epsilon, x) - \int \varphi \, \mathrm{d} \mu_X
    \Big|
    \leq
    C e^{-r (t- \epsilon )} \|\varphi\|_0 .
\end{equation}
Denote $\phi(x) :=  \nu^\delta( t- \epsilon, x) - \int \varphi \, \mathrm{d} \mu_X$, and $\nu_\phi^\delta(t,x) := \E[\phi(\mathbb{X}^{\delta}(t,x))]$.
Using 
Proposition \ref{prop:V-uniform_ergodicity_X(t)[TV4AC-25]} and
Proposition \ref{prop:error_of_X^delta-X[TV4AC-25]} leads to
\begin{equation}\label{eq:phi_estimate_in_sec_part[TV4AC-25]}
\begin{aligned}
    |\phi(x)|
&\leq
    | \nu^\delta( t- \epsilon, x) - \nu( t- \epsilon, x) |
    +
    \Big|
    \nu( t- \epsilon, x) - \int \varphi \, \mathrm{d} \mu_X
    \Big|
\\&\leq
    \Big| \mathbb{E}
    \big[\varphi (\mathbb{X}^{\delta}(t-\epsilon,x)) \big] - \mathbb{E} \big[\varphi(X(t-\epsilon,x)) \big] \Big| 
   +
   C e^{-r (t- \epsilon)} \|\varphi\|_0
\\
&\leq
    C
 \epsilon^{-1}
  e^{
  \epsilon^{-1}
  \widetilde{L}_f
%  (1 + \beta \sqrt{\delta} )
  (t-\epsilon)
  } 
 (t-\epsilon)^{\frac12}
 \| \varphi \|_0
\big(
1 
+
\|x\|_V^{  4q-3}
\big)
\sqrt{\delta}
    +
    C e^{-r (t- \epsilon)} \|\varphi\|_0 . 
\end{aligned}
\end{equation}
Note that
$\nu^\delta(t,x) = \E[\phi(\mathbb{X}^{\delta}(\epsilon,x))] + \int \varphi \, \mathrm{d} \mu_X$.
%{\color{red} recalling \eqref{lem_D_of_C0} we follow a similar step as in Lemma \ref{lem:markov_exponential_estimate_exp[TV4AC-25]} and utilize \eqref{eq:bound_xi_|y|[TV4AC-25]}, \eqref{eq:phi_estimate_in_sec_part[TV4AC-25]} to get}
Combining \eqref{lem_D_of_C0} with  \eqref{eq:bound_xi_|y|[TV4AC-25]}, \eqref{eq:phi_estimate_in_sec_part[TV4AC-25]}, and following a similar argument as used in Lemma~\ref{lem:markov_exponential_estimate_exp[TV4AC-25]} shows
\begin{equation}
\begin{aligned}
 &
 |D \nu^\delta(t, x). y |
%\\
%  & 
=
   |D \nu^\delta_\phi(\epsilon,x). y |
\\
 &  =
 \bigg|
 \epsilon^{-1}
 \mathbb E
 \left[ 
 \phi(\mathbb{X}^{\delta}(\epsilon,x)) \int_0^{\epsilon}
 \langle  
 \xi^\delta_y(s,x)
 ,
 \mathrm{d} W(s)
 \rangle
 \right]  
 \bigg|
\\
 &\leq
 \epsilon^{-1} 
  \left(
 \E[
 |\phi(\mathbb{X}^{\delta}(\epsilon,x))|^2
 ]
 \right)^\frac12
\left( 
\int_0^{\epsilon}
\mathbb E
[ 
| \xi^\delta_y(s,x)|^2 
] 
\,
\mathrm{d} s 
\right)^{\frac{1}{2}}
\\
 &\leq
 C
\epsilon^{-\frac12}
e^{ 
%\epsilon^{-1}
\widetilde{L}_f
%(1 + \beta \delta^\theta)
%\epsilon
}
 \|y\|
\Big(
 \epsilon^{-1}
  e^{
  \epsilon^{-1}
  \widetilde{L}_f
%  (1 + \beta \sqrt{\delta} )
  (t-\epsilon)
  } 
 (t-\epsilon)^{\frac12}
 \| \varphi \|_0
\Big(
1 
+
\|\mathbb{X}^{\delta}(\epsilon,x)\|
^{  4q-3 }
_{
L^{ 8q-6 }
( \Omega; V  )}
\Big)
\sqrt{\delta}
    +
    C e^{-r (t- \epsilon)}
    \|\varphi\|_0 
\Big)
\end{aligned}
\end{equation}
for $t > \epsilon$. 
Consequently, for all $t>\epsilon$,
\begin{equation}
%    \| P_t^\delta \varphi \|_1
\| D\nu^\delta (t,x) \|
    \leq
      C e^{-r t}
    \|\varphi\|_0 
 \epsilon^{-\frac12}
 +
 C(  \widetilde{L}_f , t, \epsilon^{-1} )
%\epsilon^{-\frac32}
%e^{  \widetilde{L}_f}
%  e^{\epsilon^{-1} \widetilde{L}_f (t-\epsilon)} 
 %(t-\epsilon)^{\frac12}
 \| \varphi \|_0
\Big(
1 
+
\|\mathbb{X}^{\delta}(\epsilon,x)\|
^{  4q-3 }
_{
L^{ 8q-6 }
( \Omega; V  )}
\Big)
\sqrt{\delta}.
\end{equation}
In view of  \eqref{eq:bound_xi_|y|[TV4AC-25]}, \eqref{eq:chain_rule} and \eqref{eq:bound_xi_Ay[TV4AC-25]}, 
%one follows a similar approach as in \eqref{eq:D-nu-step1}-\eqref{prop_Du(t,x)_step1}
one obtains for all $t > 2 \epsilon$, i.e., $t - \epsilon > \epsilon$,
\begin{equation}\label{eq:markov-semigroup-estimate(t>epsilon)[TV4AC-25]}
\begin{aligned}
  &
  |D \nu^\delta(t, x).y| 
  \\
&\leq 
%    \|  P_{t-\epsilon }^\delta \varphi \|_1 
 \|
 D\nu^\delta(t-\epsilon, \mathbb{X}^{\delta}(\epsilon, x)) 
 \|_{L^2(\Omega; H)}
    \| 
    \xi^\delta_y(\epsilon,x)
    \|_{L^2(\Omega; H)}
\\
&\leq
\Big(
         C e^{-rt}
    \|\varphi\|_0 
\epsilon^{-\frac12}
+
   C ( \widetilde{L}_f, t, \epsilon^{-1})
%\epsilon^{-\frac32}
%  e^{  \widetilde{L}_f}
 % e^{
 % \epsilon^{-1}
 % \widetilde{L}_f 
%  ( t - 2\epsilon )
 % } 
% ( t - 2\epsilon )^{\frac12}
 \| \varphi \|_0
\Big(
1 
+
\sup_{s \geq 0 }
\|\mathbb{X}^{\delta}(s,x)\|
^{  4q-3 }
_{
L^{ 8q-6 }
( \Omega; V  )}
\Big)
\sqrt{\delta}
\Big)
%\mathbb E 
%[ 
\|
\xi^\delta_y(\epsilon,x)
\|_{L^2(\Omega;H)}
%]
 \\&\leq
   C( \vartheta )
   e^{-r t}
    \|\varphi\|_0 
\epsilon^{-\frac12 -\vartheta}
 \Big(1 +
        \sup_{s \in [0,\epsilon]}
        \|
        \mathbb{X}^{\delta}(s,x)
        \|
        ^{2q-2}
        _
        {L^{4q-4}(\Omega; L^{4q-4} )}
 %\sup_{s \in [0,\epsilon]}
 %\E[\|\mathbb{X}^{\delta}(s,x)\|_{ L^{4q-4} }^{2q-2}]
 \Big)
 \|A^{-\vartheta} y\|
\\
&\qquad
  +
 C( \widetilde{L}_f, t, \epsilon^{-1} )
%\epsilon^{-\frac32}
%e^{  2\widetilde{L}_f}
%  e^{\epsilon^{-1} \widetilde{L}_f (t-2\epsilon)} 
%( t-2\epsilon)^{\frac12}
 \| \varphi \|_0
\Big(
1 
+
\sup_{s \geq 0}
\|\mathbb{X}^{\delta}(s,x)\|
^{  4q-3 }
_{
L^{  8q-6 }
( \Omega; V  )}
\Big)
\|y\|
\sqrt{\delta}
.
\end{aligned}
\end{equation}
Combining \eqref{eq:markov-semigroup-estimate(t<epsilon)[TV4AC-25]} and \eqref{eq:markov-semigroup-estimate(t>epsilon)[TV4AC-25]} together,
one employs Lemma \ref{lem:X^delta-Bound[TV4AC-25]} and obtains
the desired result.
\end{proof}

\section{Convergence analysis}
\label{sec:convergence_analysis[TV4AC-25]}
This section is devoted to the error analysis of the proposed explicit discretization scheme.
%by exploiting the structure of the scheme and commutativity properties of the nonlinearity.
Error bounds measured under the TV distance are established for a near sharp interface limit $\epsilon \rightarrow 0$, with only polynomial dependence on the interface parameter $\epsilon$ (see Theorem \ref{Thm_main_sharpinterface[TV4AC-25]}). The uniform-in-time error bound is also provided in Corollary \ref{cor_main_uniform[TV4AC-25]} for $\epsilon=1$ fixed. 

%To carry out the convergence analysis,

For the purpose of the convergence analysis,
we introduce a continuous version of the time-stepping schemes \eqref{eq:time_discretization[TV4AC-25]}, defined by
\begin{equation}\label{eq:time_scheme_continuous_version[TV4AC-25]}
    X^{\tau}(t) 
    = 
    E(t) X^{\tau}_0
    + 
    \int_0^t 
     E( t - \lfloor s \rfloor_{\tau} ) \epsilon^{-1} F_{\tau}( X^{\tau}(\lfloor s \rfloor_{\tau}) )  \,\mathrm{d}s
    +  
    \mathcal{O}_t,   
    \quad
    X^{\tau}_0
    =
    X(0),
\end{equation}
where $t \geq 0$, and $\lfloor s\rfloor_{\tau} := t_k$ for $s \in [t_k, t_{k+1}), k \in \N_0 $ 
and $\mathcal{O}_t$ is given in Lemma \ref{lem:O_bound[TV4AC-25]}.
Moreover, we note the process \eqref{eq:time_scheme_continuous_version[TV4AC-25]} satisfies $  X^{\tau}(t)=
 X^{\tau}_{t_k} $ for $ t = t_k, k \in \N_0$ and
\begin{equation}\label{eq:dX^tau(t)[TV4AC-25]}
    \mathrm{d}
    X^{\tau}(t)
    =
    - 
    A
    X^{\tau}(t)
    \,\mathrm{d}t
    + 
    E( t - \lfloor t \rfloor_{\tau} )
    \epsilon^{-1}
    F_{\tau} 
    ( X_{\lfloor t \rfloor_{\tau} }^{\tau} ) 
    \,\mathrm{d}t
    + 
    \mathrm{d}W(t),
    \quad
    t>0.
\end{equation}

We now present uniform-in-time moment bounds for the process \eqref{eq:time_scheme_continuous_version[TV4AC-25]} in $L^{2\rho}(\mathcal{D})$, $V$, and $\dot{H}^\gamma$ norms,
which will be used in the subsequent convergence analysis.

\begin{prop}
\label{prop:X^tau(t)-bound[TV4AC-25]}
    Let Assumptions \ref{assump:A(linear_operator)[TV4AC-25]}-\ref{assump:X_0(Initial Value)[TV4AC-25]} and  condition \eqref{eq:epsilon_condition_f_tau[TV4AC-25]}
    hold. 
     Let $X^{\tau}(t), t \geq 0$ be defined by \eqref{eq:time_scheme_continuous_version[TV4AC-25]}.
     For any $p \geq 1$,
     $\gamma < \frac12$ and 
     any integer $\rho \geq 1$, there exist positive constants
     $  C( p, \rho, q, \alpha)$,
     $ C( p, q,  \alpha) $
     and 
     $ C( p, q, \alpha, \gamma) $
     independent of $\tau, \beta$
      such that
    \begin{align}
    &
%        \sup_{t \geq 0} 
        \big\|
          X^{\tau}(t) 
        \big\|_{L^p( \Omega; L^{2\rho} )}
     \leq
     C
     (\beta
     \tau^{\theta}
     )^\alpha
 \|  X_0 \|^{2q-1}_{L^{p(2q-1)}( \Omega; L^{2\rho(2q-1)} )}
+
C(p,\rho, q, \alpha)
\cdot
\big(
1
+
(
\beta
\tau^\theta
)^{  \alpha  (4q-1-\frac{1}{2\rho}) }
\big),
\\
&
  %      \sup_{t \geq 0} 
        \big\|
        X^{\tau}(t)
        \big\|_{L^p( \Omega; V )} 
  \leq
   e^{-\epsilon^{-1}t}
    \|X_0 \|_{L^p(\Omega;V)}
   +
   C(p,  q, \alpha)
   \epsilon^{-\frac14}
   \Big(
   1
   +
     (\beta
     \tau^{\theta}
     )^\alpha
   \|
     X_0
     \|^{2q-1}_{L^{(2q-1)p}(\Omega;L^{4q-2})} 
    +
      (
   \beta
   \tau^\theta
   )^{(4q-\frac32)\alpha} 
\Big),
\\
&
%        \sup_{t \geq 0} 
        \big\|
        X^{\tau}(t)
        \big\|_{L^p( \Omega;\dot{H}^\gamma)}
\leq
   e^{-\epsilon^{-1}t}
    \|X_0 \|_{L^p(\Omega;\dot{H}^\gamma)}
   +
   C(p,  q, \alpha,\gamma)
   \epsilon^{-\frac{\gamma}{2}}
   \Big(
   1
   +
     (\beta
     \tau^{\theta}
     )^\alpha
   \|
     X_0
     \|^{2q-1}_{L^{(2q-1)p}(\Omega;L^{4q-2})} 
    +
      (
   \beta
   \tau^\theta
   )^{(4q-\frac32)\alpha} 
\Big).
    \end{align}
\end{prop}

\begin{proof}
We split the proof into three steps.
%corresponding to the uniform-in-time moment estimates in the $L^{2\rho}$, $V$, and $\dot{H}^\gamma$ norms.

\noindent\textbf{Step~1: $L^{2\rho}$-estimate.}

Note that for any $ t \geq 0$,
\begin{equation}
    X^\tau(t)
    =
    E( t -  \lfloor t \rfloor_\tau)
    X^\tau(\lfloor t \rfloor_\tau)
    +
    (t -  \lfloor t \rfloor_\tau)
    E( t -  \lfloor t \rfloor_\tau)
    \epsilon^{-1}
    F_\tau
    ( X^\tau( \lfloor t \rfloor_\tau ) )
    +
    \int_{ \lfloor t \rfloor_\tau }^t
     E( t -  s)
     \,
     \mathrm{d}W(s)
    .
\end{equation}
%{\color{red}
%Applying {\color{red}Lemmas} \ref{lem:O_bound[TV4AC-25]} and Theorem \ref{thm:X^tau-bound[TV4AC-25]}, 
%together with the fact that $\|E(t) u \|_{L^{2\rho}} \leq \|x\|_{L^{2\rho}}$ for all $x \in L^{2\rho}, t\geq 0$ and any integer $\rho \geq 1$,
%one knows
%}
%{\color{blue}
%It follows from 
%\eqref{eq:epsilon_condition_f_tau[TV4AC-25]}, 
%Lemma \ref{lem:O_bound[TV4AC-25]}, Theorem \ref{thm:X^tau-bound[TV4AC-25]}
%and the 
%contractivity 
%$\|E(t)u\|_{L^{2\rho}}\le \|u\|_{L^{2\rho}}$, $u\in L^{2\rho}(\mathcal{D})$, $t\ge 0$, that
%}
By the contractivity 
$\|E(t)u\|_{L^{2\rho}}\le \|u\|_{L^{2\rho}}$, $u\in L^{2\rho}(\mathcal{D})$, $t\ge 0$,
it follows from Lemma \ref{lem:O_bound[TV4AC-25]}, Theorem \ref{thm:X^tau-bound[TV4AC-25]} and
condition
\eqref{eq:epsilon_condition_f_tau[TV4AC-25]}
that
\begin{equation}
\begin{aligned}
    \|  X^\tau(t) \|_{L^p( \Omega; L^{2\rho} )}
&\leq
     \|  X^\tau(\lfloor t \rfloor_\tau) \|_{L^p( \Omega; L^{2\rho} )}
     +
     \tau
     \epsilon^{-1}
     \|  X^\tau(\lfloor t \rfloor_\tau) \|^{2q-1}_{L^{p(2q-1)}( \Omega; L^{2\rho(2q-1)} )}
     +
     C(p)
\\
&\leq
     C
     \beta^\alpha
     \tau^{\theta \alpha}
 \|  X_0 \|^{2q-1}_{L^{p(2q-1)}( \Omega; L^{2\rho(2q-1)} )}
+
C(p,\rho, q, \alpha)
\cdot
\big(
1
+
(
\beta
\tau^\theta
)^{ (4q-1-\frac{1}{2\rho}) \alpha }
\big).
\end{aligned}
\end{equation}
%{\color{red}one more step?}
%where we used \eqref{eq:epsilon_condition_f_tau[TV4AC-25]} in the last inequality.
%
%

\noindent\textbf{Step~2: $V$-estimate.}

First, one 
%{\color{red}infers}
rewrites
\eqref{eq:dX^tau(t)[TV4AC-25]}
as: 
for $t>0$,
\begin{equation}
\mathrm{d}X^\tau(t)
    =
    -\big(A+\epsilon^{-1}I\big) X^\tau(t)\,\mathrm{d}t
    +
    \epsilon^{-1}\!\left[
    E\big(t - \lfloor t \rfloor_\tau\big)
    F_\tau\big(X^\tau(\lfloor t \rfloor_\tau)\big)
    +
    X^\tau(t)
    \right]\mathrm{d}t
    +
    \mathrm{d}W(t).
\end{equation}
In view of the property
%that $\| E(t) x\|_{V} \leq \|x \|_V$ for $x \in V, t \geq 0$ (see, e.g., \cite[(6.2)]{thomee2007galerkin})
%and 
\eqref{eq:E(t)_semigroup_property_V[TV4AC-25]},
we then arrive at
%$\|E(t)x\|_{V} \leq C (1 \wedge t)^{-\frac14} e^{-ct} \| x \| $
%for any $t >0$ and $x \in H$ (see, e.g., \cite[(2.4)]{brehier2022ESAIM}), 
%one arrives at
\begin{equation}
\begin{aligned}
  &
    \| X^\tau(t) \|_{L^p(\Omega;V)}
   \leq
    e^{-\epsilon^{-1}t}
    \| E(t) X^\tau_0 \|_{L^p(\Omega;V)}
    +
    \bigg\| 
    \int_0^t 
    e^{-\epsilon^{-1} (t-s) }
    \epsilon^{-1}
    E(t-\lfloor s \rfloor_\tau) 
   F_\tau(  X^\tau(\lfloor s \rfloor_\tau) )
    \,\mathrm{d}s
    \bigg\|_{L^p(\Omega;V)}
\\&\qquad
   +
    \bigg\| 
    \int_0^t 
    e^{-\epsilon^{-1} (t-s) }
    E(t-s) 
    \epsilon^{-1}
    X^\tau(s)
     \,\mathrm{d}s
    \bigg\|_{L^p(\Omega;V)}
    +
    \bigg\| 
    \int_0^t 
    e^{-\epsilon^{-1} (t-s) }
    E(t-s)
    \,
    \mathrm{d}W(s)
    \bigg\|_{L^p(\Omega;V)}
\\
   & \leq
   e^{-\epsilon^{-1}t}
    \| X^\tau_0 \|_{L^p(\Omega;V)}
    +
    C \epsilon^{-1}
    \int_0^t
    e^{-\epsilon^{-1} (t-s) }
%    e^{ 
%    -
%    \frac12
%    \lambda_1
%    (t-s)
%    }
    (t-s)^{-\frac14}
   \Big( 1 
   +
   \sup_{r \in [0,t]}
     \|
     X^\tau( \lfloor r \rfloor_\tau )
     \|^{2q-1}_{L^{(2q-1)p}(\Omega;L^{4q-2})} 
   \Big)
    \,
     \mathrm{d}s
\\
&\qquad
    +
    C \epsilon^{-1}
    \int_0^t
    e^{-\epsilon^{-1} (t-s) }
%    e^{ 
%    -
%    \frac12
%    \lambda_1
%    (t-s)
%    }
    (t-s)^{-\frac14}
   \Big( 1 
   +
   \sup_{r \in [0,t]}
     \|X^\tau(r)\|_{L^{p}(\Omega;H)} 
   \Big)
    \,
     \mathrm{d}s
    +
    C 
    \sup_{s \in [0,t]}
     \|\mathcal{O}_s\|_{L^{p}(\Omega;V)}
\\
   & \leq
   e^{-\epsilon^{-1}t}
    \| X^\tau_0 \|_{L^p(\Omega;V)}
    +
    C(p)
    \\
    &
    \qquad
    +
    C( p, q , \alpha)
    \epsilon^{-1}
    \int_0^t
    e^{-\epsilon^{-1} (t-s) }
%    e^{ 
%    -
%    \frac12
%    \lambda_1
%    (t-s)
%    }
    (t-s)^{-\frac14}
   \Big(
   1 
   +
     (\beta
     \tau^{\theta}
     )^\alpha
     \|
     X^\tau_0
     \|^{2q-1}_{L^{(2q-1)p}(\Omega;L^{4q-2})}  
   +
   (
   \beta
   \tau^\theta
   )^{(4q-\frac32)\alpha}
   \Big)
     \,
     \mathrm{d}s
\\
   & \leq
   e^{-\epsilon^{-1}t}
    \|X_0 \|_{L^p(\Omega;V)}
   +
   C(p,  q, \alpha)
   \epsilon^{-\frac14}
   \Big(
   1
   +
     (\beta
     \tau^{\theta}
     )^\alpha
   \|
     X_0
     \|^{2q-1}_{L^{(2q-1)p}(\Omega;L^{4q-2})} 
    +
      (
   \beta
   \tau^\theta
   )^{(4q-\frac32)\alpha} 
\Big)
    +
    C(p),
\end{aligned}
\end{equation}
where we also used the fact
$\epsilon^{-1}
    \int_0^t
    e^{-\epsilon^{-1} (t-s) }
%    e^{ 
%    -
%    \frac12
%    \lambda_1
%    (t-s)
%    }
  ( 1 \wedge (t-s) )^{-\frac14}
     \,
     \mathrm{d}s
\leq C \epsilon^{-\frac14}
$.

%{\color{blue}how about \textbf{Let $E_\epsilon(t):=e^{-\epsilon^{-1}t}E(t)$?}}
%
%
%

\noindent\textbf{Step~3: $\dot{H}^\gamma$-estimate.}

%{\color{red}
%In a similar manner, for any $\gamma < \tfrac12$, one arrives at
%}
\iffalse
{\color{blue}
The argument is similar to \textbf{Step 2}, using instead the property
$
    \|E(t)x\|_{\dot{H}^\gamma} \le C\,t^{-\gamma/2}e^{-ct}\|x\|, t>0, x\in H,
$
 inferred from \eqref{eq:E(t)_semigroup_property[TV4AC-25]}.
Then it holds that
}
\fi
The argument is similar to \textbf{Step 2},
except that we use
$
    \|E(t)x\|_{\dot{H}^\gamma} \le C\,t^{-\gamma/2}e^{-ct}\|x\|, t>0,x\in H,
$
which follows from \eqref{eq:E(t)_semigroup_property[TV4AC-25]},
to obtain
\begin{equation}
\begin{aligned}
  &
    \| X^\tau(t) \|_{L^p(\Omega;\dot{H}^\gamma)}
   \leq
    e^{-\epsilon^{-1}t}
    \| E(t) X^\tau_0 \|_{L^p(\Omega;\dot{H}^\gamma)}
    +
    \bigg\| 
    \int_0^t 
    e^{-\epsilon^{-1} (t-s) }
    \epsilon^{-1}
    E(t-\lfloor s \rfloor_\tau) 
   F_\tau(  X^\tau(\lfloor s \rfloor_\tau) )
    \,\mathrm{d}s
    \bigg\|_{L^p(\Omega;\dot{H}^\gamma)}
\\&\qquad
   +
    \bigg\| 
    \int_0^t 
    e^{-\epsilon^{-1} (t-s) }
    E(t-s) 
    \epsilon^{-1}
    X^\tau(s)
     \,\mathrm{d}s
    \bigg\|_{L^p(\Omega;\dot{H}^\gamma)}
    +
    \bigg\| 
    \int_0^t 
    e^{-\epsilon^{-1} (t-s) }
    E(t-s)
    \,
    \mathrm{d}W(s)
    \bigg\|_{L^p(\Omega;\dot{H}^\gamma)}
\\
   & \leq
   e^{-\epsilon^{-1}t}
    \| X^\tau_0 \|_{L^p(\Omega;\dot{H}^\gamma)}
    +
    C \epsilon^{-1}
    \int_0^t
    e^{-\epsilon^{-1} (t-s) }
%    e^{ 
%    -
%    \frac12
%    \lambda_1
%    (t-s)
%    }
    (t-s)^{-\frac{\gamma}{2} }
   \Big( 1 
   +
   \sup_{r \in [0,t]}
     \|
     X^\tau( \lfloor r \rfloor_\tau )
     \|^{2q-1}_{L^{(2q-1)p}(\Omega;L^{4q-2})} 
   \Big)
    \,
     \mathrm{d}s
\\
&\qquad
    +
    C \epsilon^{-1}
    \int_0^t
    e^{-\epsilon^{-1} (t-s) }
%    e^{ 
%    -
%    \frac12
%    \lambda_1
%    (t-s)
%    }
    (t-s)^{-\frac{\gamma}{2}}
   \Big( 1 
   +
   \sup_{r \in [0,t]}
     \|X^\tau(r)\|_{L^{p}(\Omega;H)} 
   \Big)
    \,
     \mathrm{d}s
    +
    C 
    \sup_{s \in [0,t]}
     \|\mathcal{O}_s\|_{L^{p}(\Omega;\dot{H}^\gamma)}
\\
   & \leq
   e^{-\epsilon^{-1}t}
    \|X_0 \|_{L^p(\Omega;\dot{H}^\gamma)}
   +
   C(p,  q,\alpha, \gamma)
   \epsilon^{-\frac{\gamma}{2}}
   \Big(
   1
   +
     (\beta
     \tau^{\theta}
     )^\alpha
   \|
     X_0
     \|^{2q-1}_{L^{(2q-1)p}(\Omega;L^{4q-2})} 
    +
      (
   \beta
   \tau^\theta
   )^{(4q-\frac32)\alpha} 
\Big)
    +
    C(p, \gamma),
\end{aligned}
\end{equation}
The proof is thus completed.
%
%
%The remaining assertion is obtained by a similar approach as Step 2 in Theorem \ref{thm:X^tau-bound[TV4AC-25]}. The proof is thus finished.
\end{proof}

Next we present its H\"older regularity property in negative Sobolev spaces as follows.
%, whose proof is similar to that in \cite[Lemma 4.10]{wang2020efficient}, 
%{\color{red}
%where we further utilize the condition \eqref{eq:epsilon_condition_f_tau[TV4AC-25]}.
%}
%with the additional use of condition \eqref{eq:epsilon_condition_f_tau[TV4AC-25]}.

\begin{lem}\label{lem:X^tau-holder-conti[TV4AC-25]}
     Let Assumptions \ref{assump:A(linear_operator)[TV4AC-25]}-\ref{assump:X_0(Initial Value)[TV4AC-25]} and the condition \eqref{eq:epsilon_condition_f_tau[TV4AC-25]}
     hold. Let $X^{\tau}(t), t \geq 0$ be defined by \eqref{eq:time_scheme_continuous_version[TV4AC-25]}. Then for any $p \in [2, \infty)$, $\eta \in [0, \frac12]$, $\gamma<\frac12$ and $t>0$ there exists a constant $C(X_0, \alpha, p,q, \eta, \gamma) > 0$ independent of $\tau$ and $\beta$
     %depending on $ X_0, \alpha, p,q, \eta, \gamma$, 
     such that 
     \begin{equation}
         \left\|
         X^{\tau}(t) - X^{\tau}( \lfloor t \rfloor_\tau ) \right\|_{L^p(\Omega;\dot{H}^{-\eta} )}
         \leq
         C(X_0, \alpha,  p, q, \eta, \gamma) 
         \epsilon^{- \frac{\gamma+\eta}{2} } 
         \tau^{ \frac{\gamma+\eta}{2} }
         \big(
         1
         +
        (\beta
     \tau^{\theta}
     )^{
         (4q- \frac32)\alpha
         %- \frac{\gamma+\eta}{2})\alpha
         }
         \big).
     \end{equation}
\end{lem}

\begin{proof}
    Note that for any $ t \geq 0$,
\begin{equation}
\begin{split}
&    X^\tau(t)
    - 
    X^\tau(\lfloor t \rfloor_\tau)
    \\
&    =
   (
   E( t -  \lfloor t \rfloor_\tau) 
   -
   I
   )
    X^\tau(\lfloor t \rfloor_\tau)
    +
    (t -  \lfloor t \rfloor_\tau)
    E( t -  \lfloor t \rfloor_\tau)
    \epsilon^{-1}
    F_\tau
    ( X^\tau( \lfloor t \rfloor_\tau ) )
    +
    \int_{ \lfloor t \rfloor_\tau }^t
     E( t -  s)
     \,
     \mathrm{d}W(s)
    .
\end{split}
\end{equation}
It is clear from  \cite[(4.57)]{wang2020efficient} that
\begin{equation}
       \left\|
      \int_{ \lfloor t \rfloor_\tau }^t
     E( t -  s)
     \,
     \mathrm{d}W(s)
     \right\|_{L^p(\Omega; \dot{H}^{-\eta} )}
         \leq
         C
         \tau^{ \frac{\gamma+\eta}{2} }.
\end{equation}
Moreover, using the property \eqref{eq:E(t)_semigroup_property[TV4AC-25]} and Proposition \ref{prop:X^tau(t)-bound[TV4AC-25]} ensures
 \begin{equation}
 \begin{aligned}
&
    \left\|
   (
   E( t -  \lfloor t \rfloor_\tau) 
   -
   I
   )
    X^\tau(\lfloor t \rfloor_\tau)
     \right\|_{L^p(\Omega; \dot{H}^{-\eta} )}
\\
&
\leq
   \|
   A^{ -\frac{\gamma+\eta}{2} }
   (
   E( t -  \lfloor t \rfloor_\tau) 
   -
   I
   )
   \|_{\mathcal{L}(H)}
    \left\|
    X^\tau(\lfloor t \rfloor_\tau)
     \right\|_{L^p(\Omega; \dot{H}^{\gamma} )}
\\
&
\leq
         C
         \epsilon^{ - \frac{\gamma}{2}  }
         \tau^{ \frac{\gamma+\eta}{2} }
         \big(
          1
          +
         \|X_0 \|_{L^p(\Omega;\dot{H}^\gamma)}
           +
           (\beta
     \tau^{\theta}
     )^\alpha
        \|
         X_0
        \|^{2q-1}_{L^{(2q-1)p}(\Omega;L^{4q-2})} 
        +
    (
   \beta
   \tau^\theta
   )^{\alpha(4q-\frac32)} 
         \big).
 \end{aligned}
\end{equation}
Armed with condition \eqref{eq:epsilon_condition_f_tau[TV4AC-25]}, 
%implies $\tau \epsilon^{-1} \leq C$,
one further employs
 Proposition \ref{prop:F_tau[TV4AC-25]} and Theorem \ref{thm:X^tau-bound[TV4AC-25]} to show
\begin{equation}
\begin{aligned}
&
\left\|
        (t -  \lfloor t \rfloor_\tau)
    E( t -  \lfloor t \rfloor_\tau)
    \epsilon^{-1}
    F_\tau
    ( X^\tau( \lfloor t \rfloor_\tau ) )
\right\|_{L^p(\Omega; \dot{H}^{-\eta} )}
\\
&
\leq
\left\|
\tau^{ \frac{\gamma+\eta}{2} }
(\tau \epsilon^{-1})^{ 1-\frac{\gamma+\eta}{2} }
\epsilon^{-\frac{\gamma+\eta}{2}}
    F_\tau
    ( X^\tau( \lfloor t \rfloor_\tau ) )
\right\|_{L^p(\Omega; H )}
\\
&
\leq
C
\tau^{ \frac{\gamma+\eta}{2} }
\epsilon^{-\frac{\gamma+\eta}{2}}
\big(
1
+
(
\beta
\tau^\theta
)^{\alpha ( 1 - \frac{\gamma+\eta}{2} ) }
\|X_0\|^{2q-1}_{L^{(2q-1)p}(\Omega; L^{4q-2} )} 
+
(
\beta
\tau^\theta
)^{ \alpha  (4q- \frac32 - \frac{\gamma+\eta}{2}) } 
\big).
\end{aligned}
\end{equation}
The desired result then follows immediately from Assumption \ref{assump:X_0(Initial Value)[TV4AC-25]}.
\end{proof}

The subsequent lemma reveals the commutativity properties of the nonlinearity,
%which is similar to \cite[Lemma 4.9]{wang2020efficient} when $q=2$. 
%whose proof is omitted, as is similar to \cite[Lemma 4.9]{wang2020efficient} (where $q=2$) and \cite[Lemma 5.7]{wang2025uniform} (where $\varsigma \in (0,1), \eta >1$).
which is a straightforward extension of \cite[Lemma~4.9]{wang2020efficient} for $q=2$. 
%and \cite[Lemma~5.7]{wang2025uniform} for $\varsigma \in (0,1)$, $\eta>1$, and we omit the proof here.

\begin{lem}
    Let the nonlinear operator $F\colon L^{4q-2}(\mathcal{D}) \rightarrow H, q > 1$ satisfy Assumption \ref{assump:F(Nonlinearity)[TV4AC-25]}. Then for any $\varsigma \in (0, \frac12)$, $\eta > \frac12$, and $ u, v \in V \cap \dot{H}^{\varsigma}$,  there exists a constant $C(\varsigma, \eta, q)  >0$ depending on $\varsigma, \eta, q$, such that 
\begin{align}
\label{eq:|F(u)-F(v)|(negativeSobolev)[TV4AC-25]}
   \| F'(u) v \|_{ -\eta } 
   &\leq
   C(\varsigma, \eta, q)
   \Big( 1 
              + 
              \max 
              \{
              \|u\|_V 
              ,
              \|u\|_{\varsigma}
              \}^{2q-2}
        \Big) \|  v \|_{-\varsigma} .
\end{align}
\end{lem}

Equipped with the above
preparations,
%lemmas,
we are now in a position to establish the convergence rate of the proposed time-stepping scheme in TV distance.

\begin{thm}[Convergence rate of the time-stepping scheme]\label{Thm_main_sharpinterface[TV4AC-25]}

Under Assumptions \ref{assump:A(linear_operator)[TV4AC-25]}-\ref{assump:X_0(Initial Value)[TV4AC-25]}, let $X(t), t\geq 0$ be the mild solution of \eqref{eq:spde[TV4AC-25]} and $X^{\tau}_{t_m}, m \in \N$ be the time-stepping scheme defined by \eqref{eq:time_discretization[TV4AC-25]} with condition \eqref{eq:epsilon_condition_f_tau[TV4AC-25]} fulfilled. 
Then for any $\gamma < \frac{1}{2} $, $m \in \N$, there exists a positive constant $ C( X_0 ,q, \gamma, \alpha)$ independent of $\tau, \epsilon$ and $\beta$
such that 
\begin{equation}
    d_{\mathrm{TV}}
    (
     \mathrm{law}( X^{\tau}_{t_m} )
     ,
     \mathrm{law} (X(t_m) )
    )
    \leq
     C( X_0 ,q, \gamma, \alpha)
     \beta
     \big(
     1
     +
     (
     \beta
     \tau^\theta
     )^{\kappa_0}
     \big)
     \min \! \big \{
     \exp( \epsilon^{-2} )
     ,
     1 + t_m
     \big \}
     \epsilon^{ - \frac{q+3+2\gamma}{2} }
     \tau^{ \gamma \wedge \theta },
\end{equation}
where 
$\kappa_0 
:=
\max\{
(4q-4)
+
\alpha ( 8 q^2 - 6 q -1 )
,
\alpha ( 16 q^2 - 17 q + 3 )
\}
$.
%
%where
%{\color{cyan} $r = O(\exp(\epsilon^{-2}) )$ comes from Proposition \ref{prop:V-uniform_ergodicity_X(t)[TV4AC-25]} and}
%$\theta$ is from Proposition \ref{prop:F_tau[TV4AC-25]}.
%
%
In particular,
%for the sharp interface limit case,
by taking $\beta^\alpha = \epsilon^{-1}$ and $\theta =\tfrac12$ and thus
replacing the strict condition \eqref{eq:epsilon_condition_f_tau[TV4AC-25]} with
%{\color{blue}by imposing}
the $\epsilon$-independent restriction $2
    c_3
    ^2 
    \tau^{ 1-\theta\alpha } 
    \leq
    c_0
    $, one obtains
%we obtain the estimate
%{\color{red}
%\quad
%$\beta e^{ C(1 + \beta^\alpha \tau^{\theta \alpha}) }$
%}
\begin{equation}
    d_{\mathrm{TV}}
    (
     \mathrm{law}( X^{\tau}_{t_m} )
     ,
     \mathrm{law} (X(t_m) )
    )
    \leq
     C( X_0 ,q, \gamma, \alpha)
     (
     1 + t_m
     )
     \epsilon^{ 
     - \frac{q+3+2\gamma }
     {2}
     -
     \frac{1 + \kappa_0}
     {\alpha}
     }
     \tau^{ \gamma }.
\end{equation}

\end{thm}

\begin{proof}
From \eqref{eq:TV-weakformula[TV4AC-25]}, the convergence in TV distance is equivalent to establishing the following weak error bound: for all $\varphi \in \mathcal{C}^0_b(H)$ and any $m \in \mathbb{N}$,
\begin{equation}
d_{\mathrm{TV}}
    (
     \mathrm{law}( X^{\tau}_{t_m} )
     ,
     \mathrm{law} (X(t_m) )
    )
=
\sup_{
{\varphi \in \mathcal{C}_b^0(H),  \| \varphi \|_0 \leq 1}
}
\big| \mathbb{E} \big[ \varphi \big( X^{\tau}_{t_m} \big) \big] - \mathbb{E} \big[ \varphi(X(t_m)) \big] \big|.
\end{equation}
To begin with, we
introduce an auxiliary process $\mathbb{X}^{\delta}$ with parameter $\delta > 0$ defined by \eqref{eq:spde_auxiliary[TV4AC-25]} 
and make a decomposition:
%split the left hand side of \eqref{Rate.Convergence-V2[TV4AC-25]} into two terms:
\begin{equation}
\begin{split}
    &\big| \mathbb{E} \big[ \varphi \big( X^{\tau}_{t_m} \big) \big] - \mathbb{E} \big[ \varphi(X(t_m)) \big] \big|  \\
    &\leq
    \big| \mathbb{E} \big[ \varphi \big( X^{\tau}_{t_m} \big) \big] - \mathbb{E} \big[ \varphi(\mathbb{X}^{\delta}(t_m)) \big] \big| 
    +
    \big| \mathbb{E} \big[  \varphi(\mathbb{X}^{\delta}(t_m))  \big] - \mathbb{E} \big[ \varphi(X(t_m)) \big] \big|  \\
    &=:
    I_{1}^{(\delta)} + I_{2}^{(\delta)}.
\end{split}
\end{equation}
%{\color{red}
%here, and throughout, we let $\delta \rightarrow 0$. Clearly, one has $I_{2}^{(\delta)} \rightarrow 0$ owing to Proposition \ref{prop:error_of_X^delta-X[TV4AC-25]}, and it suffices to estimate for the term $I_{1}^{(\delta)}$.
%}
Since $\mathbb{X}^{\delta}$ serves as an approximation of $X$, we have $I_{2}^{(\delta)} \to 0$ as $\delta \to 0$ owing to Proposition \ref{prop:error_of_X^delta-X[TV4AC-25]}. Therefore, it remains to estimate $I_{1}^{(\delta)}$ 
uniformly in $\delta$.
By the telescoping argument, we rewrite
the term $I_{1}^{(\delta)}$ as follows:
\begin{equation}
\begin{split}
        I_{1}^{(\delta)} 
        &=
        \big| \mathbb{E} \big[ \nu^\delta \big(0, X^{\tau}_{t_m} \big) \big] - \mathbb{E} \big[ \nu^\delta \big(t_m, X_0 \big) \big] \big| \\
        &=
        \bigg| \sum_{k=0}^{m-1} 
        \underbrace{
        \mathbb{E} \big[ \nu^\delta \big(t_m -t_{k+1}, X^{\tau}(t_{k+1}) \big) \big] - \mathbb{E} \big[ \nu^\delta  \big(t_m-t_k, X^{\tau}(t_{k}) \big) \big]
        }_{=: I_{1}^{(\delta,k)}}
        \bigg|.
\end{split}
\end{equation}
Recalling $\nu^\delta(t, x) = \mathbb E[\varphi (\mathbb{X}^{\delta} (t,x))]$ and the associated Kolmogorov equation and using %{\color{red}It\^o's formula}
the It\^o formula yield, for $I_{1}^{(\delta,k)}, k = 0, 1, ..., m-1$, 
%{\color{red}
%a straightforward calculation yields  that
%}
\begin{equation}
\begin{split}
I_{1}^{(\delta,k)}
&=
 \mathbb E \bigg[ 
 \int_{t_k}^{t_{k+1}}
-\Big(
D \nu^\delta \big( t_m -t, X^{\tau}(t) \big). \big(-A X^{\tau}(t) + \epsilon^{-1} F_\delta \big( X^{\tau}(t) \big) \big) \\
&
\quad
+  \frac{1}{2} \sum_{j \in \mathbb N} D^2 \nu^\delta \big( t_m -t, X^{\tau}(t) \big).(e_j, e_j)
\Big) 
\,\dd t \\
&\quad
+
\int_{t_k}^{t_{k+1}}
D \nu^\delta \big( t_m-t, X^{\tau}(t) \big).\bigg( -A X^{\tau}(t) +  E(t - t_k) \epsilon^{-1} F_\tau( X^{\tau}(t_k) )  \bigg)
\,\dd t \\
&\quad
+
\int_{t_k}^{t_{k+1}}
\frac{1}{2} \sum_{j \in \mathbb N} 
D^2 \nu^\delta \big( t_m-t, X^{\tau}(t) \big).(e_j, e_j) 
\,\dd t 
\bigg]\\
&= \mathbb E 
\bigg[ 
\int_{t_k}^{t_{k+1}} D \nu^\delta \big( t_m-t, X^{\tau}(t) \big). 
\bigg( 
E(t - t_k) \epsilon^{-1} F_\tau( X^{\tau}(t_k) )  - \epsilon^{-1} F_\delta \big( X^{\tau}(t) \big) 
\bigg)
\,\dd t
\bigg].
\end{split}
\end{equation}
Next we decompose 
\iffalse
the integrand in $D \nu^\delta \big( t_m-t, X^{\tau}(t) \big).(\cdot)$ into the following four parts:
\begin{itemize}
\item semigroup approximation error:
$\epsilon^{-1}\big( E(t-t_k)-I \big)\,F_\tau\big(X^\tau(t_k)\big)$;
\item error between $F_\tau$ and $F$: $\epsilon^{-1}\big( F_\tau(X^\tau(t_k))-F(X^\tau(t_k)) \big)$;
\item the temporal error of $F(X^\tau)$:
$\epsilon^{-1}\big( F(X^\tau(t_k))-F(X^\tau(t)) \big)$;
\item error between $F_\delta$ and $F$:
$\epsilon^{-1}\big( F(X^\tau(t))-F_\delta(X^\tau(t)) \big)$.
\end{itemize}

%
%
%
\fi
%Accordingly, we further split
% {\color{red}it} 
the term
$I_{1}^{(\delta,k)}$
into four parts
as 
\[
I_{1}^{(\delta,k)}= I_{1,1}^{(\delta,k)}+ I_{1,2}^{(\delta,k)} + I_{1,3}^{(\delta,k)} + I_{1,4}^{(\delta,k)},
\]
where we denote
\begin{align}
   I_{1,1}^{(\delta,k)}
  & :=
   \mathbb E \bigg[ 
\int_{t_k}^{t_{k+1}} D \nu^\delta \big( t_m-t, X^{\tau}(t) \big). 
\bigg(  E(t - t_k) \epsilon^{-1} F_\tau( X^{\tau}(t_k) )  - \epsilon^{-1} F_\tau( X^{\tau}(t_k) ) \bigg) 
\,\dd t
\bigg],
\\
   I_{1,2}^{(\delta,k)}
 &  :=
   \mathbb E \bigg[ 
\int_{t_k}^{t_{k+1}} D \nu^\delta \big( t_m-t, X^{\tau}(t) \big). 
\bigg( \epsilon^{-1} F_\tau( X^{\tau}(t_k) )
        -
        \epsilon^{-1} F( X^{\tau}(t_k) ) 
\bigg) 
\,\dd t
\bigg],
\\
   I_{1,3}^{(\delta,k)}
 &  :=
   \mathbb E \bigg[ 
\int_{t_k}^{t_{k+1}} D \nu^\delta \big( t_m-t, X^{\tau}(t) \big). 
\bigg(  
       \epsilon^{-1} F( X^{\tau}(t_k) ) 
        -
       \epsilon^{-1} F( X^{\tau}(t) )
\bigg) 
\,\dd t
\bigg],
\\
   I_{1,4}^{(\delta,k)}
 &  :=
   \mathbb E \bigg[ 
\int_{t_k}^{t_{k+1}} D \nu^\delta \big( t_m-t, X^{\tau}(t) \big). 
\bigg(  
       \epsilon^{-1} F( X^{\tau}(t) )
        -
       \epsilon^{-1} F_\delta( X^{\tau}(t) )
\bigg) 
\,\dd t
\bigg].
\end{align}
We first estimate the term $ I_{1,1}^{(\delta,k)}$.
Using Proposition \ref{prop:markov_exponential_estimate_pol[TV4AC-25]} %and Lemma \ref{lem:X^delta-Bound[TV4AC-25]},
together with
 Proposition
\ref{prop:X^tau(t)-bound[TV4AC-25]} leads to
%
%{\color{red}
%\textbf{Attention! the expectation is missing!}
%}
\begin{equation}\label{eq:I^k_11[TV4AC-25]}
\begin{aligned}
 & 
 \limsup_{\delta \to 0}
      \big| I_{1,1}^{(\delta,k)} \big|
\leq
        C( q, \gamma )
        \epsilon^{-1}
        \|\varphi\|_0
        \int_{ t_k }^{ t_{k+1} }
        \left(
              1 
              +
              (
                (t_m - t)
                \wedge
                \epsilon
              )^{-\frac{1}{2}-\gamma}
        \right)
        e^{-r(t_m - t)}
\\ &\qquad
        \cdot
        \E
        \Bigg[
        \Big(
            1 
            + 
          %    \sup_{ s \in [ 0, t_m - t ] } 
          %    \E\left[ 
           %     \left\| 
           %     \mathbb{X}^{\delta} 
           %     \big( s, X^{\tau}(t) \big) 
          %      \right\|_{L^{2q-2}}^{2q-2} 
          %      \right]
          \big\|
          X^{\tau}(t)
          \big\|^{2q-2}_{L^{4q-4}}
        \Big)
             \left\| 
               A^{-\gamma} 
               ( E( t - t_k ) - I ) 
               F_{\tau}
               \big( X_{t_k}^{\tau} \big)
            \right\|
        \Bigg]
        \,\mathrm{d}t
\\
&\leq
         C(q, \gamma )
        \tau^{ \gamma }
        \epsilon^{-1}
        \| \varphi\|_0
        \int_{ t_k }^{ t_{k+1} }
         \left(
              1 
              +
              (
                (t_m - t)
                \wedge
                \epsilon
              )^{-\frac{1}{2}-\gamma}
        \right)
        e^{-r (t_m - t)}
\\ &\qquad\qquad
\cdot
    \Big(
    1 
    +
    \|  X^\tau(t) \|^{2q-2}_{L^{4q-4}(\Omega; L^{4q-4} )}
    \Big)
            \big\|
            F_{\tau}
            \big( X_{t_k}^{\tau} \big) 
            \big\|_{L^2(\Omega;H)}
        \,\mathrm{d}t
\\
      & \leq
       C( X_0 , q, \alpha,  \gamma)
        \tau^{ \gamma }
        \epsilon^{-1}
        \big(
        1
        +
        (\beta\tau^\theta)
        ^{\alpha ( 8q^2 - 6q -1 ) }
        \big)
        \| \varphi\|_0
        \int_{ t_k }^{ t_{k+1} }
        \left(
              1 
              +
              (
                (t_m - t)
                \wedge
                \epsilon
              )^{-\frac{1}{2}-\gamma}
        \right)
        e^{-r (t_m - t)}
        \,\mathrm{d}t,
\end{aligned} 
\end{equation}
where the property \eqref{eq:E(t)_semigroup_property[TV4AC-25]}, Proposition \ref{prop:F_tau[TV4AC-25]} and Assumption \ref{assump:F(Nonlinearity)[TV4AC-25]} were also utilized.
%and we let $\delta \rightarrow 0$.
%
%
%
In a similar manner, we derive for the term $ I_{1,2}^{(\delta,k)}$ that
\begin{equation}\label{eq:I^k_12[TV4AC-25]}
\begin{aligned}
&
\limsup_{\delta\to 0}
      \big| I_{1,2}^{(\delta,k)} \big|
       \leq
        C( q )
        \epsilon^{-1}
        \|\varphi\|_0
        \int_{ t_k }^{ t_{k+1} }
        \left(
              1 
              +
              (
               (t_m - t)
                \wedge
                \epsilon
              )^{-\frac{1}{2}}
        \right)
        e^{-r(t_m - t)}   
\\ &\qquad 
        \cdot
        \E
        \Bigg[
     \Big(
            1 
            + 
          %    \sup_{ s \in [ 0, t_m - t ] } 
          %    \E\left[ 
           %     \left\| 
           %     \mathbb{X}^{\delta} 
           %     \big( s, X^{\tau}(t) \big) 
          %      \right\|_{L^{2q-2}}^{2q-2} 
          %      \right]
          \big\|
          X^{\tau}(t)
          \big\|^{2q-2}_{L^{4q-4}}
        \Big)
             \left\| 
               F_{\tau}
               \big( X_{t_k}^{\tau} \big)
               -
                F
               \big( X_{t_k}^{\tau} 
               \big)
            \right\|
        \Bigg]
        \,\mathrm{d}t
\\
      & \leq
        C( X_0 , q )
         \epsilon^{  - 1 }
        \| \varphi\|_0
        \int_{ t_k }^{ t_{k+1} }
         \left(
              1 
              +
              (
               (t_m - t)
                \wedge
                \epsilon
              )^{-\frac{1}{2}}
        \right)
        e^{-r (t_m - t)}
\\&\qquad
\cdot
    \Big(
    1 
    +
    \|  X^\tau(t) \|^{2q-2}_{L^{4q-4}(\Omega; L^{4q-4} )}
    \Big)
             \left\| 
               F_{\tau}
               \big( X_{t_k}^{\tau} \big)
               -
                F
               \big( X_{t_k}^{\tau} 
               \big)
            \right\|_{L^2( \Omega; H )}
        \,\mathrm{d}t
\\ 
      & \leq
        C( X_0 , q, \alpha )
        \beta
        \tau^{ \theta}
         \epsilon^{ -1 }
      \big(
        1
        +
        (\beta\tau^\theta)
        ^{(4q-4) + \alpha ( 8q^2 - 6q -1 ) }
        \big)
        \| \varphi\|_0
        \int_{ t_k }^{ t_{k+1} }
        \left(
              1 
              +
              (
               (t_m - t)
                \wedge
                \epsilon
              )^{-\frac{1}{2}}
        \right)
        e^{-r (t_m - t)}
        \,\mathrm{d}t.
\end{aligned} 
\end{equation}
%for some integer $\rho_0 \geq 1$.
%
%
%
Similarly, for the term $ I_{1,3}^{(\delta,k)}$, one further utilizes the property \eqref{eq:|F(u)-F(v)|(negativeSobolev)[TV4AC-25]} and the Taylor formula, together with Lemma \ref{lem:X^tau-holder-conti[TV4AC-25]}, to deduce
\begin{equation}\label{eq:I^k_13[TV4AC-25]}
\begin{aligned}
&
\limsup_{\delta \to 0}
      \big| I_{1,3}^{(\delta,k)} \big|
       \leq
        C( q )
        \epsilon^{-1}
        \|\varphi\|_0
        \int_{ t_k }^{ t_{k+1} }
        \left(
              1 
              +
              (
                (t_m - t)
                \wedge
                \epsilon
              )^{-\frac{1}{2}-\frac{3}{8}}
        \right)
        e^{-r(t_m - t)}
\\ &\qquad 
        \cdot
  \Big(
            1 
            + 
          %    \sup_{ s \in [ 0, t_m - t ] } 
          %    \E\left[ 
           %     \left\| 
           %     \mathbb{X}^{\delta} 
           %     \big( s, X^{\tau}(t) \big) 
          %      \right\|_{L^{2q-2}}^{2q-2} 
          %      \right]
          \big\|
          X^{\tau}(t)
          \big\|^{2q-2}_{L^{4q-4}(\Omega; L^{4q-4})}
        \Big)
\bigg(
    \E\left[
%     \Big(
%            1 
%            + 
          %    \sup_{ s \in [ 0, t_m - t ] } 
          %    \E\left[ 
           %     \left\| 
           %     \mathbb{X}^{\delta} 
           %     \big( s, X^{\tau}(t) \big) 
          %      \right\|_{L^{2q-2}}^{2q-2} 
          %      \right]
%          \big\|
%          X^{\tau}(t)
%          \big\|^{2q-2}_{L^{2q-2}}
%        \Big)
            \left\| 
               F( X_{t_k}^{\tau} )
               -
                F( X^{\tau}(t) )
            \right\|^2_{-\frac{3}{4} }
        \right]
\bigg)^{1/2}
        \,\mathrm{d}t
\\
      & \leq
        C( X_0 ,q,  \gamma)
        \epsilon^{  - 1 }
        \big(
        1
        +
        (
        \beta
        \tau^\theta
        )
        ^
        {
        \alpha(8q^2 - 10q + \frac32)
        }
        \big)
        \| \varphi\|_0
        \int_{ t_k }^{ t_{k+1} }
        \left(
              1 
              +
              (
                (t_m - t)
                \wedge
                \epsilon
              )^{-\frac{7}{8}}
        \right)
        e^{-r (t_m - t)}
        \,\mathrm{d}t
\\
& \quad
\cdot
\bigg(
        \E\left[ 
        \left( 1 + 
            \big\| 
            X^{\tau}_{t_k} 
            \big\|_V^{4q-4}
            +
            \big\| 
            X^{\tau}_{t_k} 
            \big\|^{4q-4}_{\gamma} 
            +
            \big\|
            X^{\tau}(t) 
            \big\|_V^{4q-4}
            +
            \big\|
            X^{\tau}(t) 
            \big\|^{4q-4}_{\gamma} 
        \right)
        \big\|
          X^{\tau}_{t_k}
          - 
          X^{\tau}(t) 
        \big\|_{-\gamma}^2
        \right]
\bigg)^{1/2}
        \,\mathrm{d}t
\\ 
      & \leq
        C( X_0 ,q, \alpha)
        \epsilon^{  -1  - \frac{q-1}{2} }
        \big(
        1
        +
        (
        \beta
        \tau^\theta
        )
        ^
        {
        \alpha ( 16 q^2 - 21 q + \frac{9}{2} )
        }
        \big)
        \| \varphi\|_0
\\
&\qquad
\cdot
        \int_{ t_k }^{ t_{k+1} }
        \left(
              1 
              +
              (
                (t_m - t)
                \wedge
                \epsilon
              )^{-\frac{7}{8}}
        \right)
        e^{-r (t_m - t)}
                \big\|
          X^{\tau}_{t_k}
          - 
          X^{\tau}(t) 
        \big\|_{ L^4(\Omega; \dot{H}^{-\gamma})} 
                \,\mathrm{d}t
\\ 
      & \leq
         C( X_0 ,q, \alpha)
       \tau^{\gamma}
        \epsilon^{  - \frac{q+1+2\gamma}{2} }
        \big(
        1
        +
        (
        \beta
        \tau^\theta
        )
        ^
        {
        \alpha ( 16 q^2 - 17 q + 3 )
        }
        \big)
        \| \varphi\|_0
        \int_{ t_k }^{ t_{k+1} }
        \!
        \left(
              1 
              +
              (
                (t_m - t)
                \wedge
                \epsilon
              )^{-\frac{7}{8}}
        \right)
        e^{-r (t_m - t)}
                \,\mathrm{d}t    .   
\end{aligned} 
\end{equation}
For the term $ I_{1,4}^{(\delta,k)}$, following same steps as  \eqref{eq:I^k_12[TV4AC-25]} and noting $F_\delta$ satisfies a similar property as \eqref{eq:|F_tau(u)-F(u)|[TV4AC-25]}, one gets $I_{1,4}^{(\delta,k)} \rightarrow 0$ as $\delta \rightarrow 0$. 
As a result, one gathers the above estimates to get
\begin{equation}
\limsup_{\delta \to 0}
|I_1^{(\delta,k)}|
\leq
        C( X_0 ,q, \gamma, \alpha)
        \epsilon^{  -\frac{q+1+2\gamma}{2} }
        \beta
       \big(
        1
        +
        (
        \beta
        \tau^\theta
        )
        ^
        {
        \kappa_0
        }
        \big)
        \| \varphi\|_0
       \tau^{ \gamma \wedge \theta }
     \hspace{-0.4em}
        \int_{ t_k }^{ t_{k+1} }
      \hspace{-0.5em}
        \left(
           1
           +
           (
           (t_m-t)
          \hspace{-0.2em}
           \wedge  
           \hspace{-0.2em}
           \epsilon
           )^{
           -\kappa_1
           }  
        \right)
        e^{-r (t_m - t)}
                \,\mathrm{d}t ,
\end{equation}
where 
$\kappa_0 
=
\max\{
(4q-4)
+
\alpha ( 8 q^2 - 6 q -1 )
,
\alpha ( 16 q^2 - 17 q + 3 )
\}
$
and
$
\kappa_1
=
\max
\{
\frac78
,
\frac12+\gamma
\}
$.
Summing over $k$
and noting the fact that $r^{-1} = O( \exp( \epsilon^{-2} ) )$ (see Proposition \ref{prop:V-uniform_ergodicity_X(t)[TV4AC-25]}) lead to the desired result.
\end{proof}

In the typical case $q=2$,
Theorem \ref{Thm_main_sharpinterface[TV4AC-25]} shows that the proposed explicit scheme achieves a polynomial dependence of order at most three on $\epsilon^{-1}$, which coincides with the cubic scaling with $\epsilon^{-1}$ identified in \cite{Cui_Sun2024weaksharp} for the weak convergence analysis with smooth test functions.
Moreover, the error bound has a linear dependence with respect to the terminal time $t_{m}$, which can be regarded as a trade-off between the interface singularity and temporal accumulation.

As a direct consequence of Theorem \ref{Thm_main_sharpinterface[TV4AC-25]}, we get a uniform-in-time TV convergence rate for fixed $\epsilon=1$, extending the existing ones in the globally Lipschitz regime \cite{brehier2024FCM,Brehier2025PA} to a non-globally Lipschitz setting.

\begin{cor}[Uniform-in-time convergence rate for fixed $\epsilon=1$]
\label{cor_main_uniform[TV4AC-25]}

%Under Assumptions \ref{assump:A(linear_operator)[TV4AC-25]}-\ref{assump:X_0(Initial Value)[TV4AC-25]}, let $X(t), t\geq 0$ be the mild solution of \eqref{eq:spde[TV4AC-25]} and $X^{\tau}_{t_m}, m \in \N$ be the time-stepping scheme defined by \eqref{eq:time_discretization[TV4AC-25]}.
Let all assumptions in Theorem \ref{Thm_main_sharpinterface[TV4AC-25]} hold and let
$\epsilon=1$ be fixed.
let $X(t), t\geq 0$ be the mild solution of \eqref{eq:spde[TV4AC-25]} and $X^{\tau}_{t_m}, m \in \N$ be the time-stepping scheme defined by \eqref{eq:time_discretization[TV4AC-25]} with $\theta \geq \tfrac12$.
%further assume 
%{\color{red}
%$\epsilon \geq c_\epsilon$
%}
%for some constant $c_\epsilon>0$.
Then for any $\gamma < \frac{1}{2} $, there exists a positive constant $ C( X_0 ,q, \gamma, \alpha)$ independent of $\tau, \beta$ and $t_m$,
such that for all $m \in \N$,
\begin{equation}
    d_{\mathrm{TV}}
    (
     \mathrm{law}( X^{\tau}_{t_m} )
     ,
     \mathrm{law} (X(t_m) )
    )
    \leq
     C( X_0 ,q, \gamma, \alpha)
     \beta
     \big(
     1
     +
     (
     \beta
     \tau^\theta
     )^{\kappa_0}
     \big)
     \tau^{ \gamma },
\end{equation}
where 
$\kappa_0 
=
\max\{
(4q-4)
+
\alpha ( 8 q^2 - 6 q -1 )
,
\alpha ( 16 q^2 - 17 q + 3  )
\}
$.
\end{cor}

\section{Numerical experiments}
\label{sec:Numerical_experiments[TV4AC-25]}
In this section, we perform some numerical experiments to verify the theoretical results in the previous sections. Consider the following one-dimensional stochastic Allen-Cahn equation driven by the space-time white noise:
\begin{equation}\label{numeri}
   \left\{\begin{array}{l}
\frac{\partial u}{\partial t} (t,x)
=
\frac{\partial^2 u}{\partial x^2} 
(t,x) 
+ 
\epsilon^{-1} u(t,x)
- 
\epsilon^{-1} u^3(t,x) 
+
\dot{W}(t)
, 
\quad
t \in(0,1]
,
\quad 
x \in(0,1)
, \\
u(0, x) 
= 
\sin (\pi x)
,
\quad 
x \in(0,1)
,
\\
u(t, 0) 
=
u(t, 1)
= 
0
, 
\quad
t \in(0,1]
,
\end{array}\right.
\end{equation}
where we set $\epsilon=0.01$ and $\{ W(t) \}_{t \geq 0}$ is the cylindrical Wiener process as defined in Assumption \ref{assump:W(noise)[TV4AC-25]}.
To perform numerical experiments on a computer, we do the spatial discretization using the spectral Galerkin method ($N = 2^6$) and the temporal discretization using the proposed time-stepping scheme \eqref{eq:time_discretization[TV4AC-25]}.
%
%
%

%
%
%
%Set the endpoint of error estimation by $T = 1$.

%together with the exact solution by employing sufficiently small step-sizes of the temporal discretizations.
In what follows, 
%by choosing different parameters in  \eqref{eq:f_tau[TV4AC-25]},
we simulate the weak convergence at the endpoint $T=1$, where we choose a typical scheme with parameters $ \theta = \frac12, \alpha = 1, \beta = 5$.
The test function $\varphi$ is chosen 
as a step function 
%with a partition size of $0.1$,  expressed 
as follows:
%Specifically, $\varphi( X )$ is a piecewise constant function for $X \in H^N$, expressed as:
\begin{equation}
\label{eq:test_func[TV4AC-25]}
    \varphi(X) = \begin{cases} 
\sin(a), & \text{if } \|X\| \in [a, a+0.1), \\
\sin(a+0.1), & \text{if } \|X\| \in [a+0.1, a+0.2), \\
\vdots & \vdots \\
\sin(a+0.9), & \text{if } \|X\| \in [a+0.9, a+1),
\end{cases}
\end{equation}
where $a \in \N_0$ are any integers.
Meanwhile, we approximate the expectations appearing in weak errors by computing the average of over $1000$ samples
and identify the ``exact" solution with the numerical one produced by the implicit Euler method with
%taking 
sufficiently small step-size
($M_{\text{exact}} = 2^{14}$).
%The weak errors are simulated at the endpoint $T=1$.
Using five different step-sizes with $M = T \tau^{-1} = 2^i, i = 8, ..., 12$,
we plot weak errors in Fig.\ref{fig:error}\subref{subfig:error_beta5} 
on a log-log scale.
A comparison with the other three reference lines indicates that the convergence rate is close to $\frac12$, 
%which matches with the theoretical result well.
matching well the theoretical result.
%
%This is further quantified in Table \ref{table1}
%This can be detected more transparently from Table \ref{table1} for schemes with different choices of $\alpha$,
%demonstrating 
%which further demonstrates an improved convergence order as $\alpha$ decreases.
%Furthermore, we test for schemes with different choices of $\alpha$ and list the errors in Table \ref{table1},
We also list errors in Table \ref{tab:errors[TV4AC-25]} for schemes with different choices of $\alpha$.
An interesting observation is that, decreasing the degree $\alpha$ seemingly improves the computational accuracy for the same step-size.
%
%and list the total variation errors for various step-sizes in Table \ref{tab:errors[TV4AC-25]}, with test function chosen as in \eqref{eq:test_func[TV4AC-25]}. 
%The results are listed in Table \ref{tab:errors[TV4AC-25]}, further showing the better accuracy for a smaller $\alpha$ under the same step-size.
%
%given in Table \ref{tab:errors[TV4AC-25]}, 
%
Additionally, we test the parameter choice with
$\beta^{\alpha}= \epsilon^{-1}$, where we take $\alpha=1 $ and thus $ \beta=100$. We compute the errors with various step-sizes ($M = T \tau^{-1} = 2^i, i = 5, \dots, 9$).
As shown in Fig.\ref{fig:error}\subref{subfig:error_beta100},
the convergence rate also agrees with the theoretical one.

%For the same model and test function as previously simulated,
%we adjust the scheme parameter to $\beta=100$ and compute the errors for step-sizes corresponding to $M = T \tau^{-1} = 2^i, i = 5, \dots, 9$.

\begin{figure}[htp]   
  \centering 
\subfloat[$\beta=5$]    
  {\label{subfig:error_beta5}
  \includegraphics[width=0.45\textwidth]{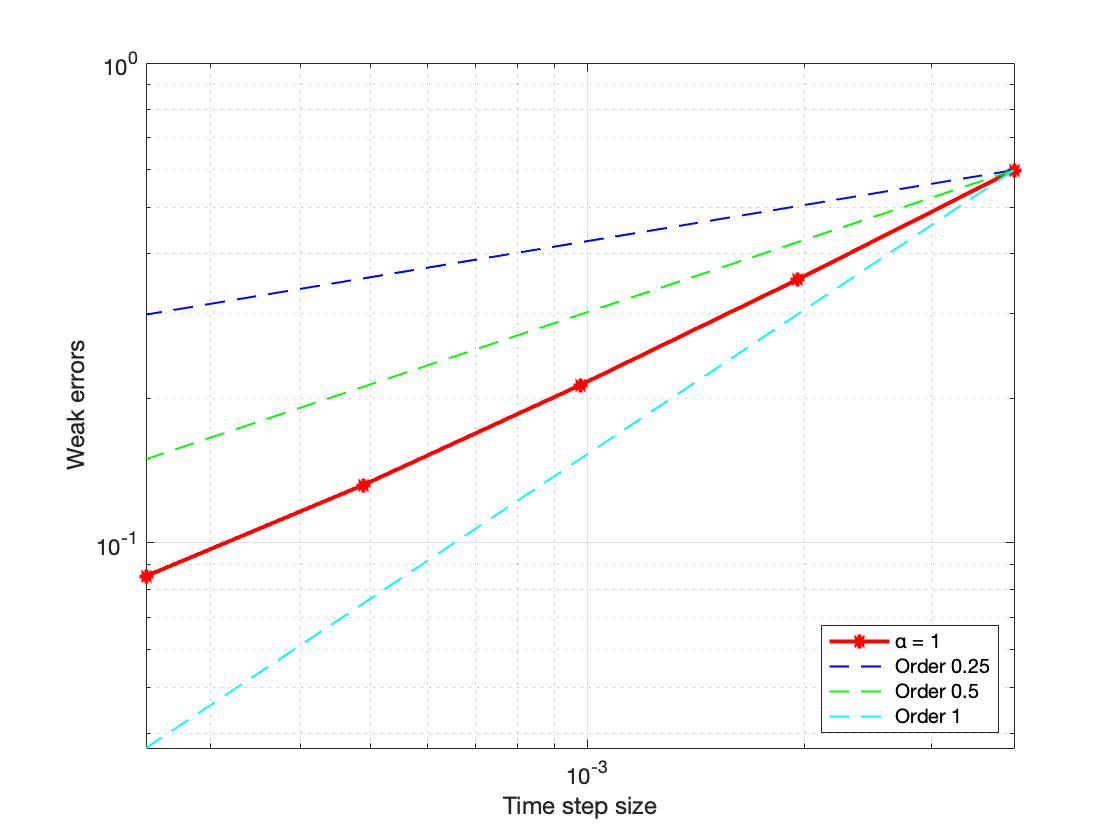}
  }
       \hspace{0.5em}
  \subfloat[$\beta=100$]
  {\label{subfig:error_beta100}
      \includegraphics[width=0.45\textwidth]{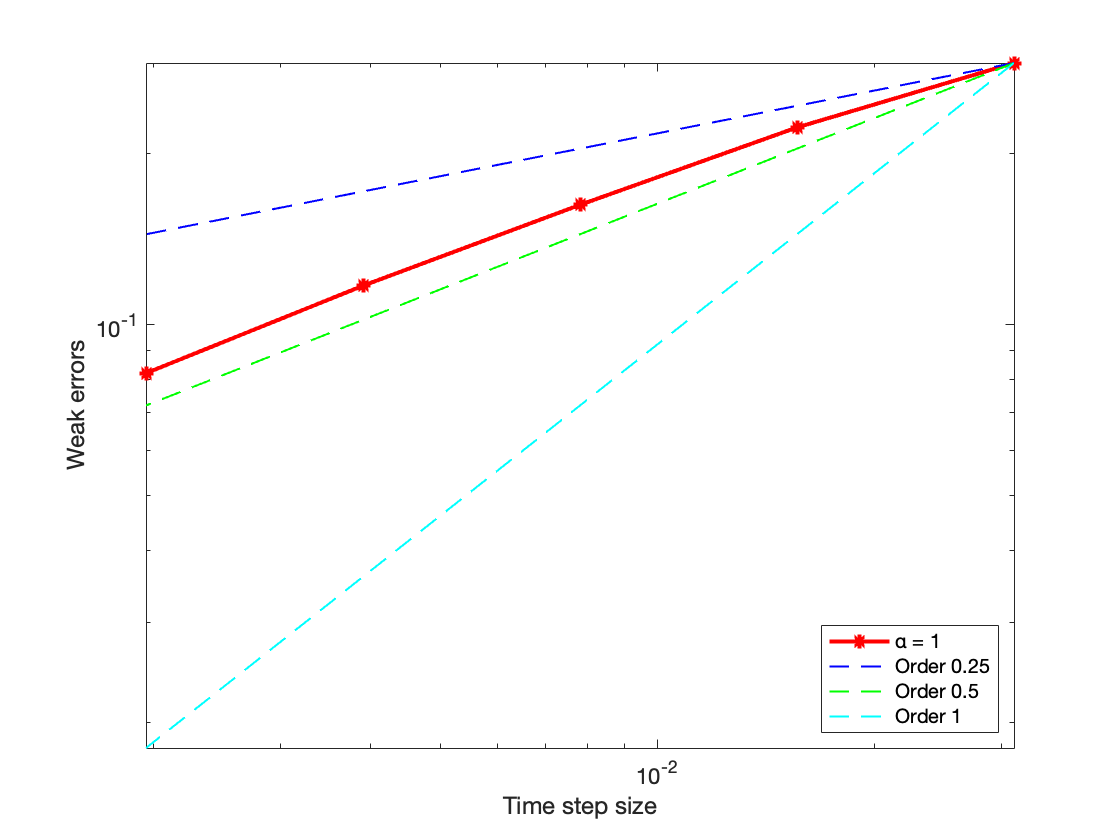}
  }
  \caption{Weak convergence rates of explicit time-stepping schemes}
\label{fig:error} 
\end{figure}

\iffalse
\begin{figure}
\centering 
\includegraphics[width=0.7\textwidth]{TVerror_beta5.png}
\caption{The weak convergence rate of the explicit time-stepping scheme.} 
\label{Fig.error} 
\end{figure}
\fi

\begin{table}[htp]
	\centering
 \footnotesize
	\setlength{\tabcolsep}{7mm}
        \caption{Errors of numerical schemes with different $\alpha$}\label{tab:errors[TV4AC-25]}
	\begin{tabular}{
    m{2cm}<{\centering} m{1cm}<{\centering} m{1cm}<{\centering} m{1cm}<{\centering} <{\centering} m{1cm}<{\centering}
    }
		\toprule[2pt]
		  & 
          $\alpha = 1$
          & 
          $\alpha = \frac12$
          & 
          $\alpha = \frac13$
          &
          $\alpha = \frac14$
          \\
		\midrule 
		 $\tau = 2^{-8}$
         & 
         0.5982
         &  
         0.4184
         &  
         0.3791
         &
         0.3626
         \\
		 $\tau = 2^{-9}$
         & 
         0.3533
         &  
         0.2347
         &  
         0.2030
         &
         0.1896
         \\
		 $\tau = 2^{-10}$
         & 
         0.2131
         &  
         0.1283
         &  
         0.1080
         &
         0.0993
         \\
		 $\tau = 2^{-11}$
         & 
         0.1319
         &  
         0.0765
         &  
         0.0622
         &
         0.0567
         \\
		 $\tau = 2^{-12}$
         & 
         0.0853
         &  
         0.0477
         &  
         0.0364
         &
         0.0340
         \\
		\bottomrule [2pt]
	\end{tabular}
	\vspace{2pt}
 \label{table1}
\end{table}

\iffalse
\begin{figure}
\centering 
\includegraphics[width=0.7\textwidth]{TVerror_alpha1_beta100.png}
\caption{The weak convergence rate of the explicit time-stepping scheme ($\beta^\alpha= \epsilon^{-1}$).} 
\label{Fig.error_beta100} 
\end{figure}
\fi

\begin{figure}[htp]   
  \centering 
  \subfloat[$\epsilon=0.01$]    
  {
  \includegraphics[width=0.45\textwidth]{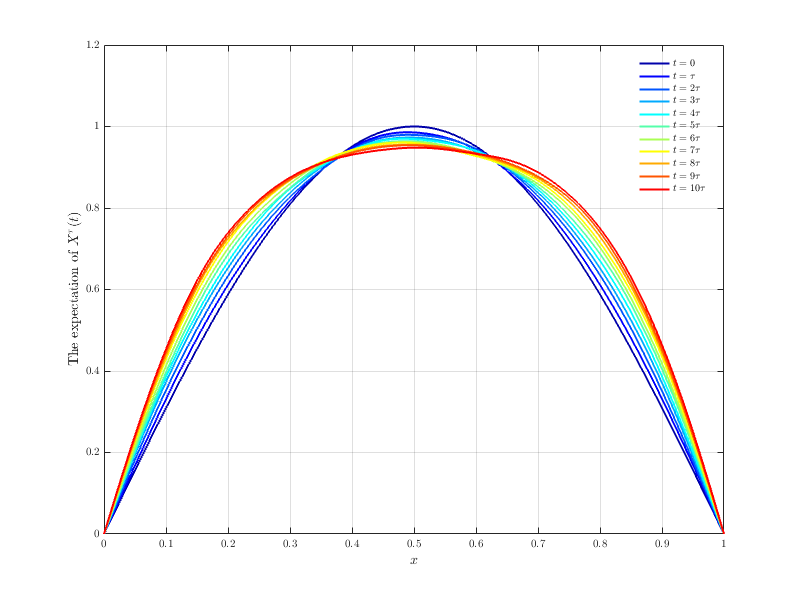}
  }
       \hspace{0.5em}
  \subfloat[$\epsilon=0.001$]
  {
      \includegraphics[width=0.45\textwidth]{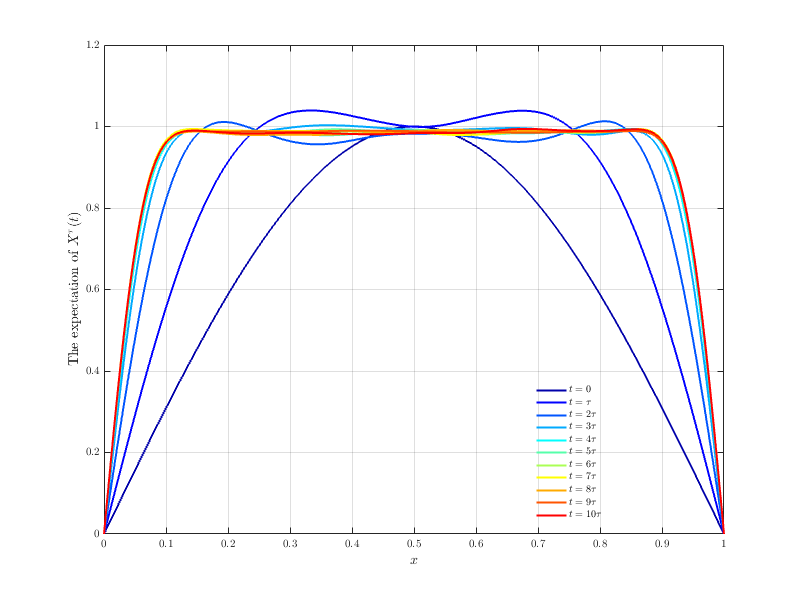}
  }
  \caption{Interface-capturing of the explicit time-stepping scheme}
\label{fig:solution} 
\end{figure}

%{\color{red}Finally, we conduct numerical experiments to demonstrate the ability of the proposed time-stepping scheme in interface-capturing.}
Finally, we conduct numerical experiments to illustrate the interface-capturing performance of the proposed time-stepping scheme.
Setting $\alpha=1, \beta=5, \theta=\tfrac12$ and using a fixed time step-size $\tau=2^{-10}$,
we investigate profiles of interface for the model \eqref{numeri} with
two different values of the interface width,
i.e., $\epsilon=0.01$ and $\epsilon=0.001$.
Via approximating the expectation with the average of 1000 samples,
the expected values of numerical solutions at various time are depicted in Fig.\ref{fig:solution} ($\epsilon=0.01$ for the left picture and $\epsilon=0.001$ for the right one). We observe that the profiles of interface can be well captured.

\vskip6mm
\bibliography{ref}

\end{document}